\documentclass[11pt]{amsart}
\usepackage[centertags]{amsmath}
\usepackage{amsfonts}
\usepackage{amssymb}
\usepackage{amsthm}
\usepackage{newlfont}
\usepackage{amscd}
\usepackage{mathrsfs}
\usepackage[all,cmtip]{xy}
\usepackage{tikz}
\usepackage[pagebackref=false,colorlinks]{hyperref}
\hypersetup{pdffitwindow=true,linkcolor=blue,citecolor=blue,urlcolor=cyan}
\usepackage{cite}
\usepackage{fancyhdr}
\usepackage{stmaryrd}

\usepackage[OT2,OT1]{fontenc}
\newcommand\cyr{%
\renewcommand\rmdefault{wncyr}%
\renewcommand\sfdefault{wncyss}%
\renewcommand\encodingdefault{OT2}%
\normalfont
\selectfont}
\DeclareTextFontCommand{\textcyr}{\cyr}

\newcommand{\Mod}[1]{\ (\text{mod}\ #1)}

\usepackage[toc, page] {appendix}

\numberwithin{equation}{section}

\newtheorem{thm}{Theorem}[section]
\newtheorem{cor}[thm]{Corollary}
\newtheorem{lem}[thm]{Lemma}
\newtheorem{prop}[thm]{Proposition}

\newtheorem{assu}[thm]{Assumption}
\newtheorem{conj}[thm]{Conjecture}

\newtheorem{algo}[thm]{``Algorithm"}

\theoremstyle{definition}
\newtheorem{defn}[thm]{Definition}
\newtheorem{rem}[thm]{Remark}

\newtheorem{ques}[thm]{Question}

\newcommand{\KS}{\mathbf{KS}}
\newcommand{\KSbar}{\overline{\mathbf{KS}}}

\newcommand{\ES}{\mathbf{ES}}
\newcommand{\ks}{\boldsymbol{\kappa}}

\begin{document}
\title[Indivisibility of Kato's Euler systems]{On the indivisibility of derived Kato's Euler systems and the main conjecture for modular forms}
\author{Chan-Ho Kim}
\address{(C.-H. Kim) Korea Institute for Advanced Study, 85 Hoegiro, Dongdaemun-gu, Seoul 02455, Republic of Korea}
\email{chanho.math@gmail.com}
\author{Myoungil Kim}
\address{(M. Kim) Department of Mathematical Sciences, Ulsan National Institute of Science and Technology, Ulsan, Korea}
\email{mikimmath@gmail.com}
\author{Hae-Sang Sun}
\address{(H.-S. Sun) Department of Mathematical Sciences, Ulsan National Institute of Science and Technology, Ulsan, Korea}
\email{haesang.sun@gmail.com}
\date{\today}
\subjclass[2010]{11R23 (Primary); 11F67 (Secondary)}
\keywords{Iwasawa theory, Iwasawa main conjectures, Kato's Euler systems, Euler systems, Kolyvagin systems, modular symbols, Hida families}
\begin{abstract}
We provide a simple and efficient numerical criterion to verify the Iwasawa main conjecture and the indivisibility of derived Kato's Euler systems for modular forms of weight two at \emph{any} good prime under mild assumptions.
In the ordinary case, the criterion works for all members of a Hida family once and for all.
The key ingredient is the explicit computation of the integral image of the derived Kato's Euler systems under the dual exponential map.
We provide explicit new examples at the end.
This work does not appeal to the Eisenstein congruence method at all.
\end{abstract}
\maketitle

\section{Introduction}
\subsection{Overview}
The theme of this article is to apply the refined nature of Kolyvagin systems to the context of Kato's Euler systems and the Iwasawa main conjecture for modular forms.

In his celebrated work \cite{kato-euler-systems}, Kato proved one divisibility of the Iwasawa main conjecture for modular forms over the cyclotomic $\mathbb{Z}_p$-extension of $\mathbb{Q}$, which gives an upper bound of Selmer groups.
The key ingredient of his proof is of course the Euler system constructed by himself.

The theory of Euler systems itself is refined later in terms of Kolyvagin systems due to Mazur and Rubin \cite{mazur-rubin-book}.
Kolyvagin systems arise from the derivative process of Euler systems, but they are more organized and admit a more rigid structure than the Kolyvagin derivative classes, which we call derived Euler systems.
In the theory of Kolyvagin systems, the concepts of primitivity and $\Lambda$-primitivity are introduced. These notions provide the criterion to obtain the \emph{exact} bound of Selmer groups and the \emph{equality} of the Iwasawa main conjecture \`{a} la Kato, respectively.
Furthermore, K.~B\"{u}y\"{u}kboduk observed that the primitivity of Kato's Kolyvagin systems implies the $\Lambda$-primitivity of Kato's $\Lambda$-adic Kolyvagin systems.
Also, it turns out that the primitivity of Kolyvagin systems is equivalent to the indivisibility of derived Euler systems.
However, it seems highly non-trivial to show that such a mod $p$ non-vanishing of cohomology classes in a direct way.

In the anticyclotomic context, the indivisibility of derived Heegner points is conjectured by Kolyvagin \cite{kolyvagin-selmer},  and proved by Wei Zhang \cite{wei-zhang-mazur-tate} under certain assumptions \emph{using} the relevant main conjecture as the combination of \cite{kato-euler-systems} and \cite{skinner-urban} and the Bertolini--Darmon--Jochnowitz congruence argument \cite{bertolini-darmon-imc-2005}.

In his unpublished thesis \cite{grigorov-thesis}, G.~Grigorov tried to understand the mod $p$ non-vanishing problem 
in terms of the mod $p$ non-vanishing of certain modular symbols via the mod $p$ dual exponential map for elliptic curves over $\mathbb{Q}$ with good ordinary reduction.
His work depends heavily on the dual exponential computation of K.~Rubin in \cite[$\S$3.5]{rubin-book} for \emph{unramified} extensions of $\mathbb{Q}_p$ and the modular symbol computation of S.~R.~Williams' unpublished thesis \cite{analytic_kolyvagin}.
Although the computation in Williams' thesis is purely analytic, the content is nothing but the mod $p$ formal Taylor expansion of Kolyvagin derivatives of Mazur--Tate elements.
We remark that our work generalizes Grigorov's thesis to modular forms of weight two with \emph{arbitrary} Fourier coefficients and deals with ordinary and non-ordinary forms on equal footing.
Furthermore, our methodology generalizes to the case of elliptic curves with \emph{additive} reduction and it is carried out in the work of the first-named author and K.~Nakamura \cite{kim-nakamura}.

Our goal is to provide a \emph{numerical} criterion for the mod $p$ non-vanishing of the Kolyvagin derivatives of Mazur--Tate elements, which implies the indivisibility of the derived Kato's Euler systems via the (integral) dual exponential map.
As an analytic and cyclotomic analogue of Kolyvagin's conjecture, the conjecture of Kurihara (Conjecture \ref{conj:kurihara}) expects that the numerical criterion always works at least for elliptic curves with good ordinary reduction. Indeed, Kurihara proved his conjecture (Theorem \ref{thm:kurihara}) \emph{using} the main conjecture and the non-degeneracy of the $p$-adic height pairing for elliptic curves with good ordinary reduction.

Although our numerical criterion works under a certain minimal level condition, the results on the main conjecture generalize to modular forms of arbitrary level via congruences following the idea of \cite{gv}, \cite{epw}, \cite{greenberg-iovita-pollack}, and \cite{kim-lee-ponsinet} under the $\mu = 0$ assumption. Therefore, the criterion checks the equality of the Iwasawa main conjecture for \emph{families} of modular forms once and for all.

In their work \cite{skinner-urban}, Skinner and Urban proved the opposite divisibility of the Iwasawa main conjecture for modular forms at good ordinary primes under certain assumptions. Thus, they obtained the equality of the Iwasawa main conjecture for a large class of modular forms.
In their argument, they introduced a technical assumption on the ramification of the residual representation. Namely, the residual representation should have a semistable but unstable prime in the tame level.
Later, Xin Wan removed the technical asumption in \cite{wan_hilbert} via the base change trick under another assumption on the existence of a certain real quadratic field concerning the period issue. See \cite[Theorem 3, Theorem 4, and Remark 5]{wan_hilbert} for detail.  Also, note that the character of modular forms is assumed to be trivial in their work.
For the recent development of the non-ordinary case, see  \cite{wan-main-conj-ss-ec}, \cite{sprung-main-conj-ss},\cite{wan-main-conj-nonord}, and \cite{castella-ciperiani-skinner-sprung}.
See $\S$\ref{subsec:current_status_IMC} for precise statements for the current development of the Iwasawa main conjecture for modular forms.

We confirm various new examples of the Iwasawa main conjecture for modular forms in $\S$\ref{sec:examples}.
\subsection{The statements}
\subsubsection{Main Theorem}
Let $p > 2$ be a prime.
Let $f = \sum_n a_n(f) q^n \in S_2(\Gamma_1(N), \psi)$ be a newform with character $\psi$ and assume $(N,p) = 1$.
Let $\mathbb{Q}_{f,\lambda}$ be the Hecke field of $f$ over $\mathbb{Q}_p$, $\mathbb{Z}_{f,\lambda}$ be its ring of integers, $\lambda$ be a uniformizer, and $\mathbb{F}_\lambda$ be the residue field.
If $a_p(f)$ is a $\lambda$-adic unit, let $\alpha$ be the unit root of the Hecke polynomial $X^2 - a_p(f)X - \psi(p)p$ of $f$ at $p$.
Let $\overline{\rho} = \overline{\rho}_f: G_\mathbb{Q} = \mathrm{Gal}(\overline{\mathbb{Q}}/\mathbb{Q}) \to \mathrm{GL}_2(\mathbb{F}_\lambda)$ be the residual Galois representation of $f$ with the tame conductor $N(\overline{\rho})$ following the cohomological convention as described in $\S$\ref{subsec:modular_galois_repns}.

Let $\mathbb{Q}_\infty$ be the cyclotomic $\mathbb{Z}_p$-extension of $\mathbb{Q}$.
Let $n$ be a square-free product of Kolyvagin primes (Definition \ref{defn:kolyvagin_primes}) and
$$\left[ \frac{a}{n} \right]^+_f := \frac{ 1 }{2 {\Omega^+_f} } \cdot  \left(\int^{a/n}_{i\infty} f(z)dz + \int^{-a/n}_{i\infty} f(z)dz \right) \in \mathbb{Z}_{f,\lambda}$$
be the $(+)$-part of the modular symbol where $\frac{a}{n} \in \mathbb{Q}$ and $\Omega^+_f$ is the $(+)$-part of an integral canonical period of $f$ defined in $\S$\ref{subsec:mod_p_multi_one}.
Let $\overline{\left[ \frac{a}{n} \right]^+_f} \in \mathbb{F}_\lambda$ be the reduction of $\left[ \frac{a}{n} \right]^+_f$ modulo $\lambda$.
Since $n$ is a product of Kolyvagin primes, any prime divisor $\ell$ of $n$ satisfies $\ell \equiv 1 \Mod{p}$.
For each $\ell$, we fix a primitive root $\eta_\ell$ mod $\ell$ and define 
$\mathrm{log}_{\mathbb{F}_\ell} (a) \in \mathbb{Z}/(\ell - 1)\mathbb{Z}$ by $\eta^{ \mathrm{log}_{\mathbb{F}_\ell} (a) }_{\ell} \equiv a \Mod{\ell}$. Let $\overline{ \mathrm{log}_{\mathbb{F}_\ell} (a) } \in \mathbb{Z}/p\mathbb{Z}$ be the image of $\mathrm{log}_{\mathbb{F}_\ell} (a)$ mod $p$.

\begin{thm}[Main Theorem] \label{thm:main_theorem}
Assume the following conditions:
\begin{itemize}
\item[(NA)] $a_p(f) \not\equiv 1 \Mod{\lambda}$ and $a_p(f) \not\equiv \psi(p) \Mod{\lambda}$;
\item[(Im)] the image of $\overline{\rho}$ contains a conjugate of $\mathrm{SL}_2(\mathbb{F}_p)$;
\item[(Tam)] $N =N(\overline{\rho})$;
\item[($N$-imp)] $p \nmid { \displaystyle \left( \prod_{q \mid N_{\mathrm{sp}}} ( q-1 ) \right) \cdot \left( \prod_{q \mid N_{\mathrm{ns}}} ( q +1 )  \right) }$
where $ { \displaystyle N_{\mathrm{sp}} := \prod_{q \Vert N(\overline{\rho}), a_q(f) = 1} q }$ and $ { \displaystyle N_{\mathrm{ns}} := \prod_{q \Vert N(\overline{\rho}), a_q(f) = -1} q }$.
\end{itemize}
If
 $$\widetilde{\delta}_n := \sum_{a \in (\mathbb{Z}/n\mathbb{Z})^\times} \left( \overline{ \left[ \frac{a}{n} \right]^+_f } \cdot  \prod_{\ell \vert n} \overline{ \mathrm{log}_{\mathbb{F}_\ell} (a) }  \right) \neq 0 \in \mathbb{F}_\lambda$$
for some $n$,
then
\begin{enumerate}
\item the derived Kato's Euler system does not vanish modulo $\lambda$, and
\item the Iwasawa main conjecture \`{a} la Kato (Conjecture \ref{conj:kato-main-conjecture}) holds for $(f, \mathbb{Q}_\infty/\mathbb{Q})$.
\end{enumerate}
\end{thm}
The first statement should be viewed as the cyclotomic analogue of the Kolyvagin conjecture on the indivisibility of derived Heegner points (\hspace{1sp}\cite[Conjecture A]{kolyvagin-selmer}).
The proof of Theorem \ref{thm:main_theorem} is sketched in $\S$\ref{subsec:the_main_idea} and the formal proof is given in $\S$\ref{subsec:the_proof}. 

We call $\widetilde{\delta}_n$ the \textbf{Kurihara number at $n$} since Kurihara extensively studied the properties of the number $\widetilde{\delta}_n$ in the context of Kolyvagin systems of Gauss sum type in \cite{kurihara-iwasawa-2012}.
Theorem \ref{thm:main_theorem} strengthens \cite[Theorem 4.(2)]{kurihara-iwasawa-2012}.
The number $\widetilde{\delta}_n$ itself depends on the choices of $\eta_\ell$ for $\ell \vert n$, but the mod $\lambda$ non-vanishing property is independent of the choices.
In \cite[$\S$3.8]{grigorov-thesis}, Grigorov provided the table of the non-vanishing of $\widetilde{\delta}_n$ for \emph{almost all} (optimal) elliptic curves over $\mathbb{Q}$ of conductor $< 30,000$ with $p \geq 5$ such that the $p$-part of the analytic order of the Shafarevich--Tate groups is non-trivial. We complete the table in Corollary \ref{cor:computation} and add several numerical examples in $\S$\ref{sec:examples}.
In \cite[Theorem 4.9]{ota-thesis}, Kazuto Ota gave a lower bound of the number of divisors of $n$ to have $\widetilde{\delta}_n \neq 0$. It should be larger than or equal to the $\mathbb{F}_\lambda$-rank of the $p$-strict Selmer group of $\overline{\rho}$ over $\mathbb{Q}$. Note that Ota studied the Mazur--Tate conjecture using the \emph{divisibility} of higher derived Kato's Euler systems.

Condition (Tam) is a necessary but a very mild condition.
Indeed, if we have $\widetilde{\delta}_n \neq 0$ for some $n$, then Condition (Tam) is automatic (Remark \ref{rem:tamagawa}).
Note that we can always find a newform satisfying Condition (Tam) in the set of congruent forms via level lowering.
Then we can spread the equality of the Iwasawa main conjecture from one form (checked by Theorem \ref{thm:main_theorem}) to all the congruent forms via congruences. Although we consider the ordinary and the non-ordinary cases separately in this article, the application of the congruences can also be studied simultaneously as in \cite{kim-lee-ponsinet}.

\subsubsection{Extension of Theorem \ref{thm:main_theorem} via congruences I}
In this subsection, we assume that $a_p(f)$ is a $\lambda$-adic unit satisfying Condition (NA) in Theorem \ref{thm:main_theorem}.

Let $f_\alpha \in S_2(\Gamma_1(N) \cap \Gamma_0(p), \psi)$ be the $p$-stabilization of $f$ (defined in $\S$\ref{subsec:p-adic-L-functions}) with the unit $U_p$-eigenvalue $\alpha = \alpha_p(f)$.
\begin{cor}[The ordinary forms of arbitrary weight]
As well as the assumptions in Theorem \ref{thm:main_theorem}, we further assume that
\begin{itemize}
\item $p > 3$, and
\item the $\mu$-invariant of the $p$-adic $L$-function $L_p(\mathbb{Q}_\infty,f_\alpha)$ (defined in $\S$\ref{subsec:p-adic-L-functions}) vanishes.
\end{itemize}
If $$\widetilde{\delta}_n  \neq 0 \in \mathbb{F}_\lambda$$
for some $n$,
then the Iwasawa main conjecture \`{a} la Mazur--Greenberg (Conjecture \ref{conj:iwasawa_main_conjecture}) holds for all members (without Condition (Tam)) of the Hida family of $\overline{\rho}$.
\end{cor}
\begin{proof}
It directly follows from \cite[Corollary 1]{epw} and
 the equivalence of the main conjectures of Kato and Mazur--Greenberg (\hspace{1sp}\cite[$\S$17.13]{kato-euler-systems}).
\end{proof}
\begin{rem} \label{rem:cor_ordinary}
Since \cite{epw} depends on \cite{diamond-taylor-non-optimal}, the $p=3$ case is excluded in the statement.
It should be emphasized that the \emph{Hida family} here means not only one tame level branch (``$\mathbb{I}$-adic") but also all ordinary forms congruent to $\overline{\rho}$. Thus, Condition (Tam) is removed in the statement.
If we apply \cite[Corollary 2.7]{ochiai-two-variable} instead of \cite[Corollary 1]{epw}, then we obtain the two-variable main conjecture \cite[Conjecture 2.4]{ochiai-two-variable} \emph{over the minimal tame level branch} without the $\mu=0$ assumption. Via \cite[Theorem 4.1.1 and Corollary 4.1.3]{fouquet-etnc}, the $\mu=0$ assumption could be removed and the $p=3$ case could be allowed, but we keep them in the statement because it has not been published yet.
\end{rem}

\subsubsection{Extension of Theorem \ref{thm:main_theorem} via congruences II}
In this subsection, we assume that $a_p(f)=0$ and $\psi$ is the trivial character, i.e. $\psi = \mathbf{1}$. Although this part depends on \cite{greenberg-iovita-pollack}, which has not been published yet, the result is now more or less well-known to the experts. Notably, the algebraic side of \cite{greenberg-iovita-pollack} is already covered in \cite{bdkim-supersingular-invariants} and \cite{hatley-lei}.

Let $S_2(\overline{\rho})[T_p]$ be the set of newforms of weight two such that
\begin{itemize}
\item their residual representations are isomorphic to $\overline{\rho}$,
\item their $p$-th Fourier coefficients are zero, and
\item their characters are trivial.
\end{itemize}

\begin{cor}[The non-ordinary forms of weight two with $a_p(f)=0$] \label{cor:pm-application}
As well as the assumptions in Theorem \ref{thm:main_theorem}, we further assume that
\begin{itemize}
\item $a_p(f) = 0$ and $\psi$ is trivial;
\item the $\mu$-invariants of the $\pm$-$p$-adic $L$-functions of $f$ vanish.
\end{itemize}
If $$\widetilde{\delta}_n  \neq 0 \in \mathbb{F}_\lambda$$
for some $n$,
then Kobayashi's $\pm$-main conjectures (Conjecture \ref{conj:pm_main_conjecture}) hold for all forms in $S_2(\overline{\rho})[T_p]$ (without Condition (Tam)).
\end{cor}
\begin{proof}
It is a direct application of \cite{greenberg-iovita-pollack} and the equivalence of the main conjectures of Kobayashi and Kato (\hspace{1sp}\cite[Theorem 7.4]{kobayashi-thesis}).
\end{proof}
\begin{rem}
The conditions $a_p(f) = 0$ and $\psi = \mathbf{1}$ are required to use the formulation of $\pm$-Iwasawa theory although they are not required in Theorem \ref{thm:main_theorem}.
As in Remark \ref{rem:cor_ordinary}, if we apply \cite[Theorem 4.1.1 and Corollary 4.1.3]{fouquet-etnc}, then the validity of the main conjecture would extend to all modular points in a certain universal Hecke algebra of $\overline{\rho}$ (with a certain $R=\mathbb{T}$ theorem).
\end{rem}

\subsubsection{Further consequences and the indivisibility of derived Kato's Euler systems} \label{subsubsec:further_consequences}
One of the advantages of Theorem \ref{thm:main_theorem} is that we can numerically compute the Kurihara numbers.
Indeed, as indicated in {\cite[Page 320 and 321]{kurihara-iwasawa-2012}}, the numerical computation of $\widetilde{\delta}_n$ is easy and we even can easily find $n$ such that $\widetilde{\delta}_n \neq 0$ in $\mathbb{F}_\lambda$, at least for elliptic curves over $\mathbb{Q}$ with good ordinary reduction. 
This yields the following practical and effective ``algorithm" for the verification of the main conjecture, and the relevant SAGE code  due to Alexandru Ghitza is available at
\begin{center}
\url{https://github.com/aghitza/kurihara_numbers}.
\end{center}
See $\S$\ref{sec:examples} to observe how it yields new examples of the main conjecture.

\begin{algo} $ $ \label{algo:algorithm}
\begin{enumerate}
\item Check whether a given modular form $f$ satisfies the assumptions of Theorem \ref{thm:main_theorem}.
\item Choose $s$ Kolyvagin primes $\ell_1, \cdots, \ell_s$ and let $\mathcal{N}$ be the set of square-free products of the chosen primes.
\item Compute $\widetilde{\delta}_n$ for all $n \in \mathcal{N}$ until we get $\widetilde{\delta}_n \in \mathbb{F}^\times_\lambda$. If we get all zeros, then go back to (2) and choose different Kolyvagin primes.
\item[(4-1)] In the ordinary case, compute $\vartheta(\mathbb{Q}_{r}, f_\alpha) \Mod{\lambda}$ for $r \geq 1$ until we get $\vartheta(\mathbb{Q}_{r}, f_\alpha) \pmod{\lambda}$ is non-zero in $\mathbb{F}_\lambda[\mathrm{Gal}( \mathbb{Q}_{r} / \mathbb{Q} )]$ where $\vartheta(\mathbb{Q}_{r}, f_\alpha)$ is defined in $\S$\ref{subsec:p-adic-L-functions}.
\item[(4-2)] In the non-ordinary case, compute $\theta(\mathbb{Q}_{r}, f) \Mod{\lambda}$ for $r \geq 1$ until we get $\theta(\mathbb{Q}_{r}, f) \pmod{\lambda}$ are non-zero in $\mathbb{F}_\lambda[\mathrm{Gal}( \mathbb{Q}_{r} / \mathbb{Q} )]$ for some odd and even $r$ where $\theta(\mathbb{Q}_{r}, f)$ is defined in $\S$\ref{subsec:p-adic-L-functions}.
\end{enumerate}
\end{algo}
\begin{rem}
We consider the above statement as an algorithm due to the following reasons. 
\begin{enumerate}
\item If $f$ does not satisfy Condition (Tam), then we replace $f$ by a congruent form satisfying Condition (Tam) via level lowering.
\item Here, $s$ should be larger than or equal to the $\mathbb{F}_\lambda$-rank of the $p$-strict Selmer group of $\overline{\rho}$ over $\mathbb{Q}$ (\hspace{1sp}\cite[Theorem 4.9]{ota-thesis}).
\item This one will terminate if we believe Kurihara's conjecture (Conjecture \ref{conj:kurihara}) at least in the ordinary case.
\item[(4-1)] This one will terminate if we believe Greenberg's conjecture \cite[Conjecture 1.11]{greenberg-lnm} on vanishing of $\mu$-invariants.
\item[(4-2)] See \cite[Proposition 6.18]{pollack-thesis} for the relation between $\theta(\mathbb{Q}_{r}, f)$ and $\pm$-$p$-adic $L$-functions. This one will terminate if we believe Pollack's conjecture \cite[Conjecture 6.3]{pollack-thesis} on vanishing of $\mu^{\pm}$-invariants.
\end{enumerate}
\end{rem}
It is natural to ask whether it is always possible to find a square-free product of Kolyvagin primes $n$ such that $\widetilde{\delta}_n \neq 0$. The following conjecture predicts that the answer is yes, at least for ordinary forms.
\begin{conj}[Kurihara; {\cite[Conjecture 1]{kurihara-iwasawa-2012}}] \label{conj:kurihara}
Under the assumptions of Theorem \ref{thm:main_theorem} with a $p$-ordinary form $f$, 
there always exists an integer $n$ such that $\widetilde{\delta}_n \neq 0$ in $\mathbb{F}_\lambda$.
\end{conj}
For the application of Kurihara's conjecture to the structure of Selmer groups, see  \cite[Theorem 3]{kurihara-iwasawa-2012}.
This conjecture should be viewed as ``the cyclotomic Kolyvagin conjecture" because it implies Theorem \ref{thm:main_theorem}.(1). 
The Kolyvagin conjecture is proved by Wei Zhang \cite[Theorem 1.1]{wei-zhang-mazur-tate} by using the main conjecture under certain assumptions.
Kurihara himself proved the following theorem toward Conjecture \ref{conj:kurihara}.
\begin{thm}[Kurihara; {\cite[Theorem 2]{kurihara-iwasawa-2012}}] \label{thm:kurihara}
Assume that $a_p(f)$ is a $\lambda$-adic unit satisfying Condition (NA) and all the other conditions in Theorem \ref{thm:main_theorem}.
If we further assume the main conjecture and the non-degeneracy of the $p$-adic height pairing, then Conjecture \ref{conj:kurihara} holds.
\end{thm}
In some sense, our main theorem (Theorem \ref{thm:main_theorem}) can be thought of a partial converse to Theorem \ref{thm:kurihara}. Ashay Burungale, Francesc Castella and the first-named author investigate an anticyclotomic analogue of this aspect in \cite{burungale-castella-kim}.
In the process of the proof of Theorem \ref{thm:main_theorem}.(2), it is observed that the numerical criterion implies the indivisibility of derived Kato's Euler systems (without making any ordinary assumption), i.e. Theorem \ref{thm:main_theorem}.(1).

\subsection{The main idea, the reduction of proof, and the organization} \label{subsec:the_main_idea}
We give a rough sketch of the proof of the main theorem (Theorem \ref{thm:main_theorem}).
The logical flow towards the main conjectures in this article is as follows:
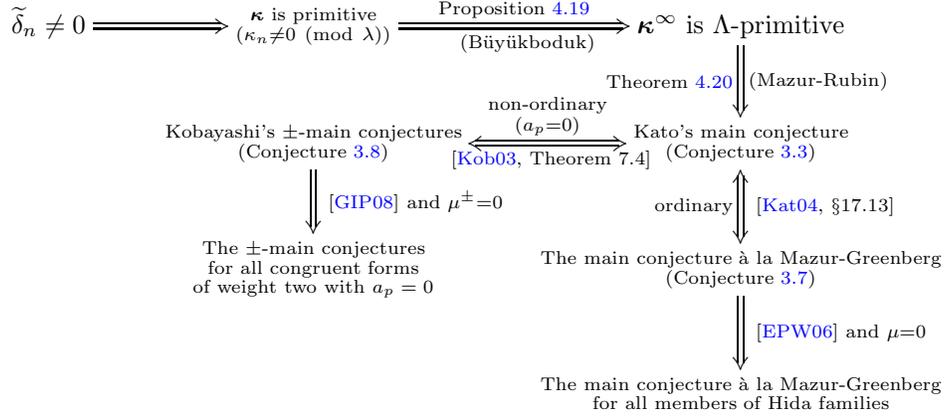
\[
\xymatrix{
\widetilde{\delta}_n \neq 0 \ar@{=>}[r] & { \substack{\textrm{$\ks$ is primitive}\\ \left( \kappa_n \neq 0 \Mod{\lambda} \right) } } \ar@{=>}[r]^-{\textrm{Proposition \ref{prop:Lambda-primitivity}}}_{\textrm{(B\"{u}y\"{u}kboduk)}} & \textrm{$\ks^{\infty}$ is $\Lambda$-primitive} \ar@{=>}[d]_-{\textrm{Theorem \ref{thm:primitivity_kato_main_conjecture}}}^-{\textrm{(Mazur--Rubin)}}\\
& { \substack{ \textrm{Kobayashi's $\pm$-main conjectures}\\\textrm{(Conjecture \ref{conj:pm_main_conjecture})} }  } \ar@{=>}[d]^-{\textrm{\cite{greenberg-iovita-pollack} and }\mu^{\pm}=0} & { \substack{\textrm{Kato's main conjecture}\\\textrm{(Conjecture \ref{conj:kato-main-conjecture})}} } \ar@{<=>}[d]_-{\textrm{ordinary}}^-{\textrm{\cite[$\S$17.13]{kato-euler-systems}}} \ar@{<=>}[l]_-{  {\substack{ \textrm{non-ordinary} \\ (a_p=0) } } }^-{\textrm{\cite[Theorem 7.4]{kobayashi-thesis}}} \\
& { \substack{ \textrm{The $\pm$-main conjectures}\\ \textrm{for all congruent forms} \\ \textrm{of weight two with $a_p=0$} } } & { \substack{ \textrm{The main conjecture \`{a} la Mazur--Greenberg} \\ \textrm{(Conjecture \ref{conj:iwasawa_main_conjecture})} } } \ar@{=>}[d]^-{\textrm{\cite{epw} and }\mu=0}\\
& & { \substack{ \textrm{The main conjecture \`{a} la Mazur--Greenberg}\\ \textrm{for all members of Hida families} } }
}
\]
where $\ks$ is Kato's Kolyvagin system (Theorem \ref{thm:euler_to_kolyvagin}) and $\ks^\infty$ is Kato's $\Lambda$-adic Kolyvagin system (Theorem \ref{thm:euler_to_Lambda-adic}).

In $\S$\ref{sec:setup}, we fix the stage we work on. Also, we explain how the assumptions in Theorem \ref{thm:main_theorem} are used. 

In $\S$\ref{sec:main_conj}, we review various Iwasawa main conjectures for modular forms and their equivalence:
\begin{enumerate}
\item Kato's main conjecture (Conjecture \ref{conj:kato-main-conjecture}),
\item Iwasawa main conjecture \`{a} la Mazur--Greenberg (Conjecture \ref{conj:iwasawa_main_conjecture}), and
\item Kobayashi's $\pm$-main conjecture (Conjecture \ref{conj:pm_main_conjecture}).
\end{enumerate}
Also, we quickly review the current state of the art toward proving the Iwasawa main conjecture for modular forms.

In $\S$\ref{sec:kolyvagin}, 
we recall the necessary material of Kolyvagin systems and explain how Kato's main conjecture can be deduced from the primitivity of Kolyvagin systems. More precisely, Proposition \ref{prop:Lambda-primitivity} shows that
\begin{quote}
$\ks$ is primitive ($\kappa_n \neq 0 \Mod{\lambda}$ for some $n$) $\Rightarrow$ $\ks^{\infty}$ is $\Lambda$-primitive.
\end{quote}

More formally, we give the following reduction of proof of Theorem \ref{thm:main_theorem}.
\begin{proof}[Reduction of Proof of Theorem \ref{thm:main_theorem}]
By Proposition \ref{prop:Lambda-primitivity} and Theorem \ref{thm:primitivity_kato_main_conjecture} (\cite[Theorem 5.3.10.(iii)]{mazur-rubin-book}), it suffices to check
$$\kappa_n \neq 0 \Mod{\lambda} $$
for some square-free product of Kolyvagin primes $n$.
\end{proof}

Thus, most content of this article is devoted to prove
$$\widetilde{\delta}_n \neq 0 \Rightarrow \kappa_n \neq 0 \Mod{\lambda}.$$

In $\S$\ref{sec:the_image_of_dual_exp}, we compute the image of $\mathrm{H}^1(\mathbb{Q}_p(\mu_n), T_{\overline{f}}(1)) \subseteq \mathrm{H}^1(\mathbb{Q}_p(\mu_n), V_{\overline{f}}(1))$ under the composition of the de Rham pairing ($\S$\ref{subsec:kato_explicit_formula}) with the dual basis $\omega^*_{\overline{f}}$, which can be detected by the Eichler--Shimura isomorphism ($\S$\ref{subsec:eichler-shimura}), and the dual exponential map.
Here, ${\displaystyle \mathbb{Q}_p(\mu_n) := \prod_{v \vert p} \mathbb{Q}(\mu_n)_v }$ and we also write ${ \displaystyle \mathbb{Z}_p[\mu_n] := \prod_{v \vert p} \mathbb{Z}[\mu_n]_v }$.
We denote the image by
\begin{align*}
\mathscr{L} & := \langle \omega^*_{\overline{f}} , \mathrm{exp}^{*} \left( \mathrm{H}^1_s(\mathbb{Q}_p(\mu_n), T_{\overline{f}}(1)) \right) \rangle_{\mathrm{dR}} \\
& \subseteq \mathbf{D}_{\mathrm{dR},\mathbb{Q}_p(\mu_n)}(\mathbb{Q}_{f, \lambda}(1)) = \mathbb{Q}_{f, \lambda} \otimes_{\mathbb{Q}_p} \mathbb{Q}_p(\mu_n),
\end{align*}
which is a $\mathbb{Z}_{f, \lambda} \otimes \mathbb{Z}_p[\mu_n]$-lattice (Proposition \ref{prop:the_image}).
The method for the computation is to compute its dual (= the image of the local points of modular abelian varieties under the logarithm map) and the Tate local duality ($\S$\ref{subsec:log_formal_groups} and $\S$\ref{subsec:computing_image}).

In $\S$\ref{sec:explicit_construction}, we explicitly construct  the ``mod $p$" Kolyvagin system from Kato's Euler system and deduce the following relation 
$$\kappa_n \neq 0 \Mod{\lambda} \Leftrightarrow D_n c^+_{\mathbb{Q}(\mu_n)} \neq 0 \Mod{\lambda}$$
where $D_n c^+_{\mathbb{Q}(\mu_n)}$ is the $(+)$-part of the derived Kato's Euler system at $\mathbb{Q}(\mu_n)$ (Proposition \ref{prop:weak_vs_derived}).

In $\S$\ref{sec:zeta_modular}, we compute the image of the localization of derived Kato's Euler systems under the dual exponential map and express it as the Kolyvagin derivative of Mazur--Tate elements as in (\ref{eqn:D_n-equivariant}) below. 
Note that Mazur--Tate elements naturally appear as the image of localized Kato's Euler systems under the dual exponential map.

Let $\mathrm{loc}_p : \mathrm{H}^1(\mathbb{Q}(\mu_n), T_{\overline{f}}(1)) \to 
\mathrm{H}^1(\mathbb{Q}_p(\mu_n), T_{\overline{f}}(1))$ be the localization map to the semi-local cohomology.
Since $\mathrm{loc}_p$ is $\mathrm{Gal}(\mathbb{Q}(\mu_n)/\mathbb{Q})$-equivariant, we have
$$D_n \mathrm{loc}_p  c^+_{\mathbb{Q}(\mu_n)} = \mathrm{loc}_p D_n c^+_{\mathbb{Q}(\mu_n)} \neq 0 \Mod{\lambda} \Rightarrow D_n c^+_{\mathbb{Q}(\mu_n)} \neq 0 \Mod{\lambda} .$$
Since the dual exponential map is also $\mathrm{Gal}(\mathbb{Q}(\mu_n)/\mathbb{Q})$-equivariant in this setting, we have
$$\mathrm{exp}^* \left( D_n \mathrm{loc}_p  c^{+}_{\mathbb{Q}(\mu_n)} \right) = D_n \mathrm{exp}^* \left(  \mathrm{loc}_p  c^{+}_{\mathbb{Q}(\mu_n)} \right) $$
in $\mathscr{L}$.
Also, the de Rham pairing defined in Theorem \ref{thm:kato_formula} is also $\mathrm{Gal}(\mathbb{Q}(\mu_n)/\mathbb{Q})$-equivariant, we have
\begin{equation} \label{eqn:D_n-equivariant}
\langle \omega^*_{\overline{f}} ,  D_n  \mathrm{exp}^* \left(\mathrm{loc}_p  c^+_{\mathbb{Q}(\mu_n)} \right) \rangle_{\mathrm{dR}}
=  D_n  \langle \omega^*_{\overline{f}} , \mathrm{exp}^* \left(\mathrm{loc}_p  c^+_{\mathbb{Q}(\mu_n)} \right) \rangle_{\mathrm{dR}} 
\end{equation}
in $\mathscr{L}$.

In $\S$\ref{subsec:williams}, we prove that the Kolyvagin derivative of the Mazur--Tate element at $\mathbb{Q}(\mu_n)$ and $\widetilde{\delta}_n$, the Kurihara number at $n$, are congruent modulo $\lambda$ (Theorem \ref{thm:computation_KS}). In other words,
$$ D_n \langle \omega^*_{\overline{f}} , \mathrm{exp}^* \left(\mathrm{loc}_p  c^+_{\mathbb{Q}(\mu_n)} \right) \rangle_{\mathrm{dR}} \not\in \lambda\mathscr{L} \Leftrightarrow \widetilde{\delta}_n \neq 0 \Mod{\lambda} .$$

To sum up, we have the following implication in $\S$\ref{subsec:the_proof}
\begin{align*}
\widetilde{\delta}_n \neq 0 \Mod{\lambda}
& \Leftrightarrow  D_n \langle \omega^*_{\overline{f}} ,  \mathrm{exp}^* \left( \mathrm{loc}_p c^+_{\mathbb{Q}(\mu_n)} \right) \rangle_{\mathrm{dR}} \not\in \lambda\mathscr{L}  \\
 & \Rightarrow
D_n c^+_{\mathbb{Q}(\mu_n)}  \neq 0 \Mod{\lambda}  \\
&\Leftrightarrow
\kappa_n \neq 0 \Mod{\lambda} .
\end{align*}
Therefore, Theorem \ref{thm:main_theorem} immediately follows.

In $\S$\ref{sec:examples}, we examine ``Algorithm" \ref{algo:algorithm} to confirm various new examples of the main conjecture for elliptic curves with good reduction and modular forms at good primes. These examples are not covered by any other former work.

As a result, we understand the Kurihara number $\widetilde{\delta}_n$ at $n$  as the mod $\lambda$ localized image of Kolyvagin derivative of the $(+)$-part of Kato's Euler system at $\mathbb{Q}(\mu_n)$ under the dual exponential map.
Since Kurihara obtained $\widetilde{\delta}_n$ from his Euler systems of Gauss sum type, it seems natural to ask the following question.
\begin{ques}
What is the explicit relation between Kato's Euler systems \cite{kato-euler-systems} and the Euler systems of Gauss sum type \`{a} la Kurihara \cite{kurihara-munster}, \cite{kurihara-iwasawa-2012}?
\end{ques}
\section{Setup and remarks on the conditions in Theorem \ref{thm:main_theorem}} \label{sec:setup}
\subsection{Fixed embeddings}
Let $p$ be a prime $> 2$. Fix embeddings $\iota_\infty : \overline{\mathbb{Q}} \hookrightarrow \mathbb{C}$, $\iota_p: \overline{\mathbb{Q}} \hookrightarrow \overline{\mathbb{Q}}_p$, and an abstract field isomorphism $\iota: \mathbb{C} \simeq \overline{\mathbb{Q}}_p$ such that $\iota \circ \iota_\infty =\iota_p$. For a field $F$, let $G_F$ be the absolute Galois group of $F$.
\subsection{Modular forms}
For any ring $R$, let $S_2(\Gamma_1(N), R)$ be the space of cuspforms whose Fourier coefficients lie in $R$.

Let $f = \sum a_n(f) q^n \in S_2(\Gamma_1(N), \overline{\mathbb{Q}})$ be a newform with character $\psi$.
Let $\mathbb{Q}_f$ be the Hecke field of $f$ over $\mathbb{Q}$, which is totally real or CM depending on $\psi$, $\mathbb{Z}_f$ be the ring of integers of $\mathbb{Q}_{f}$.
Let $S(f)$ be a quotient $\mathbb{Q}$-vector space of $S_2(\Gamma_1(N), \overline{\mathbb{Q}})$ corresponding to $f$ following \cite[$\S$6.3]{kato-euler-systems}. Then $S(f)$ is one-dimensional over $\mathbb{Q}_f$. 

Let $\overline{f} = \sum \overline{a_n(f)} q^n \in S_2(\Gamma_1(N), \overline{\mathbb{Q}})$ be the dual modular form of $f$ as in \cite[$\S$6.5]{kato-euler-systems} where $\overline{a_n(f)}$ is the complex conjugate of $a_n(f)$. Then the character of $\overline{f}$ is $\overline{\psi} = \psi^{-1}$.

Let $\lambda$ be the place of $\mathbb{Q}_f$ dividing $p$ and compatible with $\iota_p$.
Let $\mathbb{Q}_{f,\lambda}$ be the completion of $\mathbb{Q}_f$ at $\lambda$.
Let $\mathbb{Z}_{f,\lambda}$ be the ring of integers of $\mathbb{Q}_{f,\lambda}$.
Then we have
$$\mathbb{Q}_f \otimes_\mathbb{Q} \mathbb{Q}_p = \prod_{\lambda' \vert p} \mathbb{Q}_{f, \lambda'}$$
where $\lambda'$ runs over the primes of $\mathbb{Q}_f$ dividing $p$.
Let $\mathbb{F}_\lambda := \mathbb{Z}_{f,\lambda} / \lambda\mathbb{Z}_{f,\lambda}$ be the residue field of $\mathbb{Q}_{f,\lambda}$.

\subsection{Hecke algebras} \label{subsec:hecke_algebra}
Let $\mathbb{T}$ be the full Hecke algebra over $\mathbb{Z}_p$ acting faithfully on $S_2(\Gamma_1(N), \overline{\mathbb{Z}}_p)$ where $\overline{\mathbb{Z}}_p$ is the integral closure of $\mathbb{Z}_p$ in $\overline{\mathbb{Q}}_p$.
Let $\wp_f \subseteq \mathbb{T}$ be the ideal generated by $T_\ell - a_\ell(f)$ for all primes $\ell$ and $\langle a \rangle$ for $a \in (\mathbb{Z}/N\mathbb{Z})^\times$.
Let $\mathfrak{m}  \subseteq \mathbb{T}$ be the maximal ideal generated by $\wp_f$ and $\lambda$, which corresponds to the residual representation of $f$.
Then $\mathbb{T} / \wp_f$ is an order of $\mathbb{Q}_{f,\lambda}$ and  
 $\mathbb{T} / \mathfrak{m} = \mathbb{F}_\lambda$. We denote the localization of $\mathbb{T}$ at $\mathfrak{m}$ by $\mathbb{T}_{\mathfrak{m}}$.
 
\subsection{Modular Galois representations} \label{subsec:modular_galois_repns}
Let $\rho_f : G_\mathbb{Q} \to \mathrm{GL}_2(\mathbb{Q}_{f, \lambda}) \simeq \mathrm{GL}(V_f)$ be the $\lambda$-adic Galois representation associated to $f$ arising from the \'{e}tale cohomology of a modular curve. Then $\rho_f$ satisfies the following properties \cite[$\S$14.10]{kato-euler-systems}:
\begin{enumerate}
\item $\mathrm{det} ( \rho_f ) = \chi^{-1}_{\mathrm{cyc}} \cdot \psi^{-1}$
where $\chi_{\mathrm{cyc}}$ is the cyclotomic character ($\S$\ref{subsec:cyclo_extns_iwasawa_alg});
\item for any prime $\ell$ not dividing $Np$, we have
$$\mathrm{det} \left( 1- \rho_f \left( \mathrm{Fr}^{-1}_\ell \right) \cdot u  : \left( V_f \right)^{I_\ell} \right) = 1 - a_{\ell}(f) u +  \psi  (\ell) \cdot \ell \cdot u^2$$
where $\mathrm{Fr}_\ell$ is the arithmetic Frobenius at $\ell$ and $I_\ell$ is the inertia subgroup of $G_{\mathbb{Q}_\ell}$;
\item for the prime number $p$ lying under $\lambda$,  we have
$$\mathrm{det} \left( 1- \varphi \cdot u  : \mathbf{D}_{\mathrm{cris}}  ( V_f ) \right) = 1 - a_{p}(f) u +  \psi  (p) \cdot p \cdot u^2$$
where $\varphi$ is the Frobenius operator acting on $\mathbf{D}_{\mathrm{cris}}  ( V_f )$, Fontaine's crystalline Dieudonn\'{e} module associated to the restriction of $V_f$ to $G_{\mathbb{Q}_p}$.
\end{enumerate}
For any Galois module $M$ over $\mathbb{Z}_p$, let $M(k) := M \otimes_{\mathbb{Z}_p} \mathbb{Z}_p(k)$ be the $k$-th Tate twist of $M$ for $k \in \mathbb{Z}$.
Let $\Sigma = \Sigma(N)$ be the finite set of places of $\mathbb{Q}$ consisting of $p$, $\infty$, the places dividing $N$.
Let $\mathbb{Q}_\Sigma$ be the maximal extension of $\mathbb{Q}$ unramified outside $\Sigma$.
Then $\rho_f$ factors through $\mathrm{Gal}(\mathbb{Q}_\Sigma/\mathbb{Q})$.

Let $\overline{\rho} : G_\mathbb{Q} \to \mathrm{GL}_2(\mathbb{F}_\lambda)$ be the residual Galois representation of $V_f$.
Due to Condition (Im) in Theorem \ref{thm:main_theorem}, all the content of this article is independent of the choice of a Galois-stable $\mathbb{Z}_{f,\lambda}$-lattice $T_f$ of $V_f$. Let $A_f := V_f/T_f$.

Let $J_1(N)_f = J_1(N)_{f, \mathbb{Q}}$ be the modular abelian variety over $\mathbb{Q}$ attached to $f$ as the quotient of $J_1(N)$ by the ideal $\wp_f$ in the Hecke algebra $\mathbb{T}_{\mathbb{Z}}$ over $\mathbb{Z}$.
Then it is an abelian variety over $\mathbb{Q}$ with endomorphism ring $\mathrm{End} (J_1(N)_f) = \mathbb{T}_{\mathbb{Z}}/\wp_f$.
Note that all Galois conjugates of $f$ define abelian varieties which are isomorphic each other.
Let $\mathfrak{J}_1(N)_f$ be the N\'{e}ron model of $J_1(N)_f$ over $\mathbb{Z}$ and $\widehat{\mathfrak{J}_1(N)}_f$ be the formal group of $\mathfrak{J}_1(N)_f$. This formal group appears in $\S$\ref{subsec:log_formal_groups}.

Let $V_{\lambda}(J_1(N)_f)$ be the Galois representation arising from the $\lambda$-adic Tate module of $J_1(N)_f$.
For a vector space $V$ over a field $F$, let $V^*$ be the $F$-dual of $V$. 
Following \cite{conrad-shimura}, we have
$V_{f} \simeq V_{\lambda}(J_1(N)_{f})^*$ and
$V_f (1) \simeq V_{\lambda}(J_1(N)_{\overline{f}})$.
More precisely, we have $a_\ell (f) = \mathrm{tr}(\rho_f(\mathrm{Fr}^{-1}_\ell)) = \mathrm{tr}(\rho_f(1)(\mathrm{Fr}_\ell))$.
Due to the duality of modular Galois representations \cite[(14.10.1)]{kato-euler-systems}
$$ V_f(1)^*(1) \simeq V_{\overline{f}}(1) ,$$
we also consider the dual representation
 $$\rho_{\overline{f}}(1) : \mathrm{Gal}(\overline{\mathbb{Q}}/\mathbb{Q}) \to \mathrm{Aut}_{\mathbb{Q}_{f,\lambda}}(V_{\overline{f}}(1)) \simeq \mathrm{GL}_2(\mathbb{Q}_{f,\lambda})$$
and denote the corresponding $\mathbb{Z}_{f,\lambda}$-lattice by $T_{\overline{f}}(1)$.

Let $R$ be any $p$-adic ring including $\mathbb{Q}_{f,\lambda}$, $\mathbb{Z}_{f,\lambda}$, and $\mathbb{Z}_{f,\lambda} / \lambda^i$. Then, for any $R$-module $M$, we set $M^* := \mathrm{Hom}_R(M, R)$.
Also, the $R$-torsion part of $M$ is denoted by $M_{\mathrm{tors}}$.

\subsection{Cyclotomic extensions and Iwasawa algebras} \label{subsec:cyclo_extns_iwasawa_alg}
Let $\mathbb{Q}(\mu_{p^\infty})$ be the full cyclotomic extension of $\mathbb{Q}$ with Galois group $G_{\infty} := \mathrm{Gal}( \mathbb{Q}(\mu_{p^\infty})/\mathbb{Q})$, and
let
 \[
\xymatrix{
 \chi_{\mathrm{cyc}} : G_{\infty} \ar[r]^-{\simeq} & \mathbb{Z}^\times_p
}
\]
be the cyclotomic character. For $c \in \mathbb{Z}^\times_p$, let $\sigma_{c} \in G_\infty$ be the unique element such that 
$\chi_{\mathrm{cyc}} (\sigma_{c}) = c$.

Let $\mathbb{Q}_\infty \subset \mathbb{Q}(\mu_{p^\infty})$ be the cyclotomic $\mathbb{Z}_p (\simeq 1+ p\mathbb{Z}_p)$-extension of $\mathbb{Q}$ and $\Gamma_\infty := \mathrm{Gal}( \mathbb{Q}_\infty/\mathbb{Q}) \simeq \mathbb{Z}_p$.
Then we have
$$G_\infty \simeq \Gamma_\infty \times \Delta$$
where $\Delta \simeq (\mathbb{Z}/p\mathbb{Z})^\times \simeq \mathbb{Z}/(p-1)$.
Let $\mathbb{Q}_r \subseteq \mathbb{Q}(\mu_{p^{r+1}})$ be the cyclic extension of $\mathbb{Q}$ in $\mathbb{Q}(\mu_{p^{r+1}})$ of degree $p^r$.
Also, for $a \in (\mathbb{Z}/n\mathbb{Z})^\times$, write $\sigma_{a^{-1}} \in \mathrm{Gal}(\mathbb{Q}(\mu_n)/\mathbb{Q})$ as the image of $a$ under the global Artin map, which behaves like the inverse of the $p$-adic cyclotomic character. Then $\sigma_\ell$ is the arithmetic Frobenius at $\ell$ in $\mathrm{Gal}(\mathbb{Q}(\mu_n)/\mathbb{Q})$ with $\ell \nmid n$.

Let $\widetilde{\Lambda} = \mathbb{Z}_{f,\lambda} \llbracket G_\infty \rrbracket$ be the extended cyclotomic Iwasawa algebra and $\Lambda := \mathbb{Z}_{f,\lambda} \llbracket \mathrm{Gal}( \mathbb{Q}_\infty/\mathbb{Q})\rrbracket$ be the cyclotomic Iwasawa algebra over $\mathbb{Z}_{f,\lambda}$.

For convenience, if $(n,p)=1$, we always choose $p$ as a uniformizer for $\mathbb{Q}(\mu_n)_v \subseteq \mathbb{Q}_p(\mu_n)$ for any $v \vert p$ since $\mathbb{Q}(\mu_n)_v /\mathbb{Q}_p$ is unramified.
\subsection{Remarks on the conditions in Theorem \ref{thm:main_theorem}} \label{subsec:rem}
We briefly review how the conditions in Theorem \ref{thm:main_theorem} are used in this article. 
\begin{rem}[NA]
The non-anomalous assumption $a_p(f) \not\equiv 1 \Mod{\lambda}$ removes the exceptional zero case, which would harm Theorem \ref{thm:specialization_surjective}, so it would also violate Proposition \ref{prop:Lambda-primitivity}.
The assumption $a_p(f) \not\equiv \psi(p) \Mod{\lambda}$ is crucially used in $\S$\ref{subsec:the_proof}.
If $\psi(p) \equiv 1 \Mod{\lambda}$, then two conditions obviously coincide.
Note that
\begin{itemize}
\item $a_p(f) \not\equiv 1 \Mod{\lambda} \Leftrightarrow$ the Euler factor of $L(f,s)$ at $p$ at $s = 0$ is not congruent to 0 mod $\lambda$;
\item $a_p(f) \not\equiv \psi(p) \Mod{\lambda} \Leftrightarrow$ $p \cdot$(the Euler factor of $L(f,s)$ at $p$ at $s = 1$) is not congruent to 0 mod $\lambda$.
\end{itemize}
Note that both conditions $a_p(f) \not\equiv 1 \Mod{\lambda}$ and $a_p(f) \not\equiv \psi(p) \Mod{\lambda}$ are observed in the context of the Bloch--Kato conjecture \cite[(5.15.1)]{bloch-kato}.
\end{rem}
\begin{rem}[Im]
The residual image assumption ensures the integrality of $p$-adic $L$-functions. See \cite[Theorem 12.5.(4) (Page 222) and Theorem 17.4.(3) (Page 273)]{kato-euler-systems} for detail. Following the argument of \cite[$\S$2.5]{skinner-pacific}, the assumption could be slightly relaxed as the irreducibility of $\overline{\rho}$ and Assumption \ref{assu:abstract}.(H2). Also, if the tame conductor of the residual representation has a semi-stable prime, Assumption \ref{assu:abstract}.(H2) is always satisfied.
\end{rem}

\begin{rem}[Tam] \label{rem:tamagawa}
It removes all the Tamagawa defect (Lemma \ref{lem:conditions}.(2)).
If it is violated, then the Tamagawa defect must happen in the context of the ``quantitative level lowering" See \cite[Conjecture 1.4 and $\S$6.6]{pw-mu}, \cite{kim-ota} for detail.
In this situation, the corresponding Euler system cannot produce a primitive Kolyvagin system as in \cite[Proposition 6.2.6]{mazur-rubin-book}, \cite{kazim-tamagawa}.
\end{rem}

\begin{rem}[$N$-imp] \label{rem:N-imp}
It removes the discrepancy coming from the difference of the valuations between $N$-primitive and $N$-imprimitive $L$-values.
It is observed in \cite[Theorem 6.2.4]{mazur-rubin-book} and \cite[Proposition 4.3.(E3)]{kazim-Lambda-adic}.
Note that ``$N$-imprimitive Kato's Euler systems" are used in this setting. One can compare \cite[Theorem 6.6 and Theorem 12.5]{kato-euler-systems} for this issue.
\end{rem}

\section{Iwasawa main conjectures for modular forms} \label{sec:main_conj}
\subsection{Selmer groups} \label{subsec:selmer_groups}
Let $F$ be an algebraic extension of $\mathbb{Q}$. Then the \textbf{Selmer group of $A_f(1)$ over $F$} is defined by
$$\mathrm{Sel}(F, A_f(1)) := \mathrm{ker} \left(  \mathrm{H}^1(F, A_f(1)) \to \prod_v \frac{ \mathrm{H}^1(F_v, A_f(1)) }{ \mathrm{H}^1_f(F_v, A_f(1)) } \right) $$
where $v$ runs over all places of $F$ and  $\mathrm{H}^1_f(F_v, A_f(1))$ is the image of the local Kummer map at $v$ defined by
$$\mathrm{Kum}_v: J_1(N)_{\overline{f}, \lambda}(F_v)  \otimes_{ \mathbb{Z}_{f,\lambda} } \mathbb{Q}_{f,\lambda} / \mathbb{Z}_{f,\lambda} \to \mathrm{H}^1(F_v, A_f(1)) .$$
This classical definition is equivalent to the Bloch--Kato Selmer group (\hspace{1sp}\cite[$\S$14.1]{kato-euler-systems}) using the crystalline period ring \`{a} la Fontaine. See also \cite[Page 70]{greenberg-lnm}. 

\subsubsection{Ordinary forms}
Suppose that $a_p(f)$ is a $\lambda$-adic unit satisfying Condition (NA) in Theorem \ref{thm:main_theorem}. Then our definition of Selmer groups also coincides with that of Greenberg ordinary Selmer groups for Hida deformation \cite[$\S$4.1]{epw}. See \cite[Page 572]{epw}. For the higher weight generalization to apply \cite{epw}, we replace Selmer groups by Greenberg ordinary Selmer groups. Although \cite{epw} includes the exceptional zero case by using Greenberg ordinary Selmer groups, we do not allow the exceptional zero case to examine the $\Lambda$-primitivity of Kato's Kolyvagin system.

\subsubsection{Non-ordinary forms}
Suppose that $a_p(f)=0$ and $\psi = \mathbf{1}$.
Following \cite{kobayashi-thesis}, we define the $\mathbb{Z}_{f, \lambda}$-submodules of $J_1(N)_{\overline{f}, \lambda}(\mathbb{Q}_{n,p})$ by
\begin{align*}
J_1(N)^{+}_{\overline{f}, \lambda}(\mathbb{Q}_{n,p}) & := \lbrace P \in J_1(N)_{\overline{f}, \lambda}(\mathbb{Q}_{n,p}) : \mathrm{Tr}_{n/m+1} (P) \in J_1(N)_{\overline{f}, \lambda}(\mathbb{Q}_{m,p}) \textrm{ for even } m \ (0 \leq m < n)\rbrace \\
J_1(N)^{-}_{\overline{f}, \lambda}(\mathbb{Q}_{n,p}) & := \lbrace P \in J_1(N)_{\overline{f}, \lambda}(\mathbb{Q}_{n,p}) : \mathrm{Tr}_{n/m+1} (P) \in J_1(N)_{\overline{f}, \lambda}(\mathbb{Q}_{m,p}) \textrm{ for odd } m \ (0 \leq m < n)\rbrace
\end{align*}
where $\mathrm{Tr}_{n/m+1} : J_1(N)_{\overline{f}, \lambda}(\mathbb{Q}_{n,p}) \to J_1(N)_{\overline{f}, \lambda}(\mathbb{Q}_{m+1,p})$ is the trace map.
Then the \textbf{$\pm$-Selmer groups of $f$ over $\mathbb{Q}_n$} is defined by
$$\mathrm{Sel}^{\pm}(\mathbb{Q}_n, A_f(1)) 
: = \mathrm{ker} \left( 
\mathrm{Sel}(\mathbb{Q}_n, A_f(1)) \to 
 \dfrac{\mathrm{H}^1(\mathbb{Q}_{n,p}, A_f(1))}{J_1(N)^{\pm}_{\overline{f}, \lambda}(\mathbb{Q}_{n,p}) \otimes \mathbb{Q}_{f,\lambda}/\mathbb{Z}_{f, \lambda}} 
\right) $$
and the \textbf{$\pm$-Selmer groups of $f$ over $\mathbb{Q}_\infty$} by
$$\mathrm{Sel}^{\pm}(\mathbb{Q}_\infty, A_f(1)) : = \varinjlim_n \mathrm{Sel}^{\pm}(\mathbb{Q}_n, A_f(1)) ,$$
respectively.

\subsection{Mazur--Tate elements and $p$-adic $L$-functions} \label{subsec:p-adic-L-functions}
We quickly review the Mazur--Tate elements and $p$-adic $L$-functions of modular forms of weight two.
\subsubsection{Mazur--Tate elements}
Let $\mathbb{Q}(\mu_{n})^+$ be the maximal totally real subfield of $\mathbb{Q}(\mu_{n})$.
We define \textbf{Mazur--Tate element of $f$ at $\mathbb{Q}(\mu_{n})^+$} by
$$\theta^+(\mathbb{Q}(\mu_{n}),f) := \sum_{a \in (\mathbb{Z}/n\mathbb{Z})^{\times}/ \lbrace \pm 1\rbrace }\left[ \frac{a}{n} \right]^+_f \cdot \sigma_{a} \in \mathbb{Z}_{f, \lambda}[\mathrm{Gal}(\mathbb{Q}(\mu_{n})^+/\mathbb{Q})],$$
where 
$$\left[ \frac{a}{n} \right]^{\pm}_f := \frac{ 1 }{2 {\Omega^{\pm}_f} } \cdot  \left( \int^{a/n}_{i\infty} f(z)dz \pm \int^{-a/n}_{i\infty} f(z)dz  \right) \in \mathbb{Z}_{f,\lambda}$$
and
$$\left[ \frac{a}{n} \right]_f  := \left[ \frac{a}{n} \right]^+_f  + \left[ \frac{a}{n} \right]^-_f $$
where $\Omega^{\pm}_f$ is the $(\pm)$-part of an integral canonical period of $f$. See $\S$\ref{subsec:mod_p_multi_one} for the definition of the periods.

\subsubsection{$p$-adic $L$-functions}
Suppose that $a_p(f)$ is a $\lambda$-adic unit satisfying Condition (NA) in Theorem \ref{thm:main_theorem}.
Let $\beta$ be the non-unit root of the Hecke polynomial $X^2 - a_p(f)X - \psi(p)p$ of $f$ at $p$.
Then the $p$-stabilization $f_\alpha$ of $f$ is defined by
$f_\alpha(z) := f(z) - \beta \cdot f(pz) $.
Let
\begin{align*}
\pi^{r}_{r-1} : \mathbb{Z}_{f,\lambda}[\mathrm{Gal}(\mathbb{Q}(\mu_{p^{r}})^+/\mathbb{Q})] & \to \mathbb{Z}_{f,\lambda}[\mathrm{Gal}(\mathbb{Q}(\mu_{p^{r-1}})^+/\mathbb{Q})], \\
\nu^{r}_{r-1} : \mathbb{Z}_{f,\lambda}[\mathrm{Gal}(\mathbb{Q}(\mu_{p^{r-1}})^+/\mathbb{Q})] & \to \mathbb{Z}_{f,\lambda}[\mathrm{Gal}(\mathbb{Q}(\mu_{p^{r}})^+/\mathbb{Q})]
\end{align*}
be the natural projection and the norm map defined by
${\displaystyle \sigma \mapsto \sum_{\pi^{r}_{r-1}:\tau \mapsto \sigma} \tau }$, respectively.
Then we define
$$\vartheta^+(\mathbb{Q}(\mu_{p^r}),f_\alpha) := \frac{1}{\alpha^r} \cdot \left(  \theta^+(\mathbb{Q}(\mu_{p^r}),f)  - \frac{1}{\alpha}  \cdot \nu^{r}_{r-1} \left( \theta^+(\mathbb{Q}(\mu_{p^{r-1}}),f) \right)   \right)$$
and $\vartheta(\mathbb{Q}_{r},f_\alpha)$ to be the natural image of $\vartheta^+(\mathbb{Q}(\mu_{p^{r+1}}),f_\alpha)$ in $\mathbb{Z}_{f,\lambda}[\mathrm{Gal}(\mathbb{Q}_{p^{r}}/\mathbb{Q})]$.
Then the sequence $\left( \vartheta^+(\mathbb{Q}(\mu_{p^r}), f_\alpha) \right)_r$ forms a projective system and the limit defines the $p$-adic $L$-functions of $f$ for $\mathbb{Q}(\mu_{p^{\infty}})^+/\mathbb{Q}$
$$L_p(\mathbb{Q}(\mu_{p^\infty})^+, f_\alpha) := \varprojlim_r \vartheta^+(\mathbb{Q}(\mu_{p^r}), f_\alpha) \in \widetilde{\Lambda}^+ $$
where $\widetilde{\Lambda}^+ := \mathbb{Z}_{f,\lambda} \llbracket\mathrm{Gal}(\mathbb{Q}(\mu_{p^{\infty}})^+/\mathbb{Q})\rrbracket $.

The \textbf{$p$-adic $L$-function of $f$} for the cyclotomic $\mathbb{Z}_p$-extension of $\mathbb{Q}$ is defined by
the image of $L_p(\mathbb{Q}(\mu_{p^\infty})^+, f_\alpha)$ in $\Lambda$ under the natural projection $\widetilde{\Lambda}^+ \to \Lambda$. We denote it by $L_p(\mathbb{Q}_\infty, f_\alpha)$.
\begin{rem}
For the construction of the $p$-adic $L$-functions of modular forms of higher weight, see \cite[$\S$3.2]{epw}. 
Each $p$-adic $L$-function can be understood as an integrally coherent weight specialization of ``two variable" $p$-adic $L$-functions as explained in \cite[$\S$3.3 and $\S$3.4]{epw}.
\end{rem}
\subsubsection{$\pm$-$p$-adic $L$-functions}
Suppose that $a_p(f) = 0$.
Rather than recalling the construction of $\pm$-$p$-adic $L$-functions in \cite{pollack-thesis}, we recall the characterization of $L^{\pm}_p(\mathbb{Q}_\infty,f) \in \Lambda$ by their interpolation property.
Let $\Phi_m$ be the $p^m$-th cyclotomic polynomial and
\[
\xymatrix@R=0em{
{\displaystyle \widetilde{\omega}^{+}_n = \widetilde{\omega}^{+}_n(X)  := \prod_{2 \leq m \leq n, m: \textrm{ even}}\Phi_m(1+X) } , &
{\displaystyle \widetilde{\omega}^{-}_n = \widetilde{\omega}^{-}_n(X)  := \prod_{1 \leq m \leq n, m: \textrm{ odd}}\Phi_m(1+X) } .
}
\]
Then we have the following interpolation property \cite[(3.4)--(3.7), Page 7]{kobayashi-thesis}, \cite[(10), (11), and (12)]{pollack-rubin}:
\begin{align*}
\chi \left( L^+_p(\mathbb{Q}_\infty,f) \right) & = (-1)^{(n+1)/2} \cdot \dfrac{\tau(\chi)}{\chi(\widetilde{\omega}^+_n)} \cdot \dfrac{L(f, \chi^{-1}, 1)}{\Omega^+_f} & \textrm{ if $\chi$ has order $p^n$ with $n$ odd } \\
\chi \left( L^-_p(\mathbb{Q}_\infty,f) \right) & = (-1)^{(n/2) + 1} \cdot \dfrac{\tau(\chi)}{\chi(\widetilde{\omega}^-_n)} \cdot \dfrac{L(f, \chi^{-1}, 1)}{\Omega^+_f} & \textrm{ if $\chi$ has order $p^n>1$ with $n$ even } \\
\mathbf{1} \left( L^+_p(\mathbb{Q}_\infty,f) \right) & = (p-1) \cdot \dfrac{L(f, 1)}{\Omega^+_f} \\
\mathbf{1} \left( L^-_p(\mathbb{Q}_\infty,f) \right) & = 2 \cdot \dfrac{L(f, 1)}{\Omega^+_f}
\end{align*}
where $\chi$ is a character on $\mathrm{Gal}(\mathbb{Q}_{\infty}/\mathbb{Q})$ of $p$-power order, $\mathbf{1}$ is the trivial character, and $\tau(\chi)$ is the Gauss sum of $\chi$.

\subsection{The Iwasawa main conjecture for modular forms \`{a} la Kato}
We recall Kato's reformulation of the Iwasawa main conjecture for $T_{\overline{f}}(1)$ over $\mathbb{Q}_{\infty}$.
See \cite[Chapter I.$\S$3 (especially Conjecture 3.2.2)]{kato-lecture-1} and \cite[Chapter 4 (especially $\S$4.3.4 and $\S$4.4.5.Examples.(ii))]{perrin-riou-book} for the background of this formulation.
Indeed, Kato's original formulation is given for $T_{\overline{f}}$  in \cite{kato-euler-systems} and two formulations are equivalent up to the twist by Teichm{\"{u}}ller character. See \cite[$\S$6.5]{rubin-book}.
We write $\mathrm{H}^i(F/K, M) = \mathrm{H}^i(\mathrm{Gal}(F/K), M)$.
Consider
\begin{align*}
\mathrm{H}^1(\mathbb{Q}_\infty, T_{\overline{f}}(1)) & \simeq \mathrm{H}^1(\mathbb{Q}_\Sigma/\mathbb{Q}_\infty, T_{\overline{f}}(1))  & \textrm{\cite[Lemma 5.3.1.(iii)]{mazur-rubin-book}} \\
& \simeq \varprojlim_n \mathrm{H}^1(\mathbb{Q}_\Sigma/\mathbb{Q}_n, T_{\overline{f}}(1)) & \textrm{\cite[Lemma 5.3.1.(i)]{mazur-rubin-book}}
\end{align*}
 and $\kappa^{\infty}_1 := \varprojlim_n c^+_{\mathbb{Q}_n} \in \mathrm{H}^1(\mathbb{Q}_\infty, T_{\overline{f}}(1))$ be the $\Lambda$-adic Kato's Kolyvagin system at $\mathbb{Q}_{\infty}$ constructed from Kato's Euler system (Theorem \ref{thm:euler_to_Lambda-adic}).
Let $j_n : \mathrm{Spec}(\mathbb{Q}_n) \to \mathrm{Spec}(\mathcal{O}_{\mathbb{Q}_n}[1/p])$ be the natural map and
we define the \textbf{$i$-th Iwasawa cohomology} by
$$\mathbb{H}^i(T_{\overline{f}}(1)) := \varprojlim_{n} \mathrm{H}^i_{\mathrm{\acute{e}t}}( \mathrm{Spec}(\mathcal{O}_{\mathbb{Q}_n}[1/p]), j_{n,*}T_{\overline{f}}(1)) $$
where $\mathrm{H}^i_{\mathrm{\acute{e}t}}( \mathrm{Spec}(\mathcal{O}_{\mathbb{Q}_n}[1/p]),j_{n,*}T_{\overline{f}}(1))$ is the \'{e}tale cohomology group. Then $\mathbb{H}^1(T_{\overline{f}}(1)) \simeq \mathrm{H}^1(\mathbb{Q}_\infty, T_{\overline{f}}(1))$ by \cite[Proposition 7.1.(i)]{kobayashi-thesis}.
\begin{thm}[{\hspace{1sp}\cite[Theorem 12.4.(1) and (3)]{kato-euler-systems}}]
$ $
\begin{enumerate}
\item $\mathbb{H}^2(T_{\overline{f}}(1))$ is a finitely generated torsion module over $\Lambda$.
\item $\mathbb{H}^1(T_{\overline{f}}(1))$ is free of rank one over $\Lambda$.
\end{enumerate}
\end{thm}

Let $\mathbf{z}_{\mathrm{Kato}} \in \mathbb{H}^1(T_{\overline{f}}(1))$ be Kato's $p$-adic zeta element, which is ``$\mathbf{z}^{(p)}_\gamma \otimes (\zeta_{p^n})_n$" in \cite[Theorem 12.5]{kato-euler-systems}.
The main conjecture \`{a} la Kato is as follows.
\begin{conj}[{\hspace{1sp}\cite[Conjecture 12.10]{kato-euler-systems}, \cite[Conjecture 6.1]{kurihara-invent}}] \label{conj:kato-main-conjecture}
$$\mathrm{char}_{\Lambda} \left(  \mathbb{H}^1(T_{\overline{f}}(1)) / \Lambda \mathbf{z}_{\mathrm{Kato}}  \right)
= \mathrm{char}_{\Lambda} \left( \mathbb{H}^2(T_{\overline{f}}(1)) \right) .$$
\end{conj}
Note that the $\Lambda$-torsion property of $\mathbb{H}^1(T_{\overline{f}}(1)) / \Lambda \mathbf{z}_{\mathrm{Kato}}$ is  due to \cite[Theorem 12.5.(2)]{kato-euler-systems}.
Indeed, there is an explicit relation between $\kappa^{\infty}_1$ and $\mathbf{z}_{\mathrm{Kato}}$ via \cite[Lemma 13.10 and $\S$13.12]{kato-euler-systems} and we have
$$\Lambda \kappa^{\infty}_1 \subseteq \Lambda  \mathbf{z}_{\mathrm{Kato}}$$ \emph{with finite index} via \cite[Theorem 12.6]{kato-euler-systems}. Thus, we have the following proposition.
\begin{prop} \label{prop:reduction-of-kato-main-conjecture}
$$\mathrm{char}_{\Lambda} \left(  \mathbb{H}^1(T_{\overline{f}}(1)) / \Lambda \mathbf{z}_{\mathrm{Kato}}  \right)
= \mathrm{char}_{\Lambda} \left(  \mathbb{H}^1(T_{\overline{f}}(1)) / \Lambda \kappa^{\infty}_1 \right) .$$
\end{prop}
\begin{rem} $ $
 The statement of Conjecture \ref{conj:kato-main-conjecture} implicitly assumes the canonical choice of Kato's zeta element since the RHS is independent of the choice of Kato's zeta element. See \cite[Remark 1.11.(2)]{ochiai-families} for detail.
\end{rem}

\subsection{The Iwasawa main conjecture for modular forms \`{a} la Mazur--Greenberg}
Suppose that $a_p(f)$ is a $\lambda$-adic unit satisfying Condition (NA) in Theorem \ref{thm:main_theorem}.
For a $\mathbb{Z}_{f,\lambda}$-module $M$, we define the Pontryagin dual by
$$M^\vee := \mathrm{Hom}_{\mathbb{Z}_{f,\lambda}} (M, \mathbb{Q}_{f,\lambda}/\mathbb{Z}_{f,\lambda}) .$$
\begin{rem}
In \cite[$\S$17.3]{kato-euler-systems}, Kato took $\mathrm{Hom}_{\mathbb{Z}_{f,\lambda}} (M(-1), \mathbb{Q}_{f,\lambda}/\mathbb{Z}_{f,\lambda})$ as the Pontryagin dual since he formulated the conjecture for $T_{\overline{f}}$ not for $T_{\overline{f}}(1)$.
\end{rem}
Due to the work of Kato \cite[Theorem 1.5]{greenberg-lnm}, \cite[Theorem 17.4.(1)]{kato-euler-systems} and Rohrlich \cite{rohrlich-nonvanishing}, the finitely generated $\Lambda$-module $\mathrm{Sel}(\mathbb{Q}_\infty, A_f(1))^\vee$ is $\Lambda$-torsion.

\begin{conj}[The Iwasawa main conjecture for $(A_f(1), \mathbb{Q}_\infty/\mathbb{Q})$; {\cite[Conjecture 2]{greenberg-general-iwasawa}}, {\cite[Conjecture 17.6]{kato-euler-systems}}] \label{conj:iwasawa_main_conjecture}
As ideals of $\mathbb{Z}_{f,\lambda}\llbracket \mathrm{Gal}(\mathbb{Q}_\infty/\mathbb{Q})\rrbracket$, the following equality holds
$$\left( L_p(\mathbb{Q}_\infty, f_\alpha) \right) =  \mathrm{char}_{\Lambda} \left( \mathrm{Sel}(\mathbb{Q}_\infty, A_f(1))^\vee \right) .$$
\end{conj}
Following \cite[$\S$17.13]{kato-euler-systems}, Conjecture \ref{conj:kato-main-conjecture} and Conjecture \ref{conj:iwasawa_main_conjecture} are equivalent.

\subsection{The Iwasawa main conjecture for modular forms \`{a} la Kobayashi}
Suppose that $a_p(f) = 0$ and $\psi = \mathbf{1}$. 
Then the finitely generated $\Lambda$-module $\mathrm{Sel}^{\pm}(\mathbb{Q}_\infty, A_f(1))^\vee$ are $\Lambda$-torsion (\hspace{1sp}\cite[Theorem 7.3.ii)]{kobayashi-thesis}).
\begin{conj}[{\hspace{1sp}\cite[Conjecture in $\S$5]{kobayashi-thesis}}] \label{conj:pm_main_conjecture}
As ideals of $\mathbb{Z}_{f,\lambda}\llbracket \mathrm{Gal}(\mathbb{Q}_\infty/\mathbb{Q})\rrbracket$, the following equalities hold
$$ \left( L^{\mp}_p(\mathbb{Q}_\infty, f_\alpha) \right) = \mathrm{char}_{\Lambda} \left( \mathrm{Sel}^{\pm}(\mathbb{Q}_\infty, A_f(1))^\vee \right) .$$
\end{conj}
Following \cite[Theorem 7.4]{kobayashi-thesis}, Conjecture \ref{conj:kato-main-conjecture} and each $\pm$-one of Conjecture \ref{conj:pm_main_conjecture} are equivalent.

\subsection{Remarks on the current status of the Iwasawa main conjecture for modular forms} \label{subsec:current_status_IMC}

As a digression, we quickly review the current status of the Iwasawa main conjecture for modular forms.
\begin{thm}[Skinner--Urban, X.~Wan]
Assume that $f$ is good ordinary at $p$, $\psi = \mathbf{1}$, and
the image of $\overline{\rho}$ contains a conjugate of $\mathrm{SL}_2(\mathbb{F}_p)$ (Condition (Im) in Theorem \ref{thm:main_theorem}).
\begin{itemize}
\item[\cite{skinner-urban}] If there exists a prime $q \Vert N$ such that $\overline{\rho}$ is ramified at $q$, then
then Conjecture \ref{conj:iwasawa_main_conjecture} holds.
\item[\cite{wan_hilbert}] If there exists a real quadratic field $F/\mathbb{Q}$ such that
\begin{itemize}
\item $p$ is unramified in $F$,
\item any prime $q$ dividing $N$ such that $q \equiv -1 \Mod{p}$ is inert in $F/\mathbb{Q}$, and any other prime dividing $N$ splits in $F/\mathbb{Q}$,
\item the canonical period of $f$ over $F$ is the square of its canonical period over $\mathbb{Q}$ up to a $p$-adic unit,
\end{itemize}
then Conjecture \ref{conj:iwasawa_main_conjecture} holds.
\end{itemize}
\end{thm}
In \cite[Theorem 4]{wan_hilbert}, it is required to find a suitable real quadratic field. It does not seem easy to find it (at least algorithmically). See \cite[Remark 5]{wan_hilbert} for this issue.

\begin{thm}[X.~Wan, Sprung, Castella--\c{C}iperiani--Skinner--Sprung]
Suppose that $f$ is non-ordinary at $p$, $\psi = \mathbf{1}$, and
the image of $\overline{\rho}$ contains a conjugate of $\mathrm{SL}_2(\mathbb{F}_p)$ (Condition (Im) in Theorem \ref{thm:main_theorem}). Then Conjecture \ref{conj:kato-main-conjecture} holds if one of the following assumptions hold:
\begin{itemize}
\item[\cite{wan-main-conj-ss-ec}] $a_p(f) = 0$, $\mathbb{Q}_f = \mathbb{Q}$, $N$ is square-free;
\item[\cite{sprung-main-conj-ss}] $\mathbb{Q}_f = \mathbb{Q}$, $N$ is square-free;
\item[\cite{wan-main-conj-nonord}] there exists a prime $q \Vert N$ such that the local automorphic representation at $q$ is the Steinberg representation twisted by the character sending $q$ to $-1$, or $N$ is square-free and there exist two primes $q_1$ and $q_2$ exactly dividing $N$ such that $\overline{\rho}$ is ramified at $q_1$ and $q_2$;
\item[\cite{castella-ciperiani-skinner-sprung}] $N$ is square-free.
\end{itemize}
\end{thm}
Note that \cite{wan-main-conj-ss-ec}, \cite{sprung-main-conj-ss}, \cite{wan-main-conj-nonord}, and \cite{castella-ciperiani-skinner-sprung} are not published yet, and our approach is completely different from theirs.
For the application of the Iwasawa main conjecture to the size of (twisted) Selmer groups, see \cite[Theorem in Introduction]{kato-euler-systems} and  \cite[Theorem 3.35 and 3.36]{skinner-urban}.

\section{A quick review of Kolyvagin systems} \label{sec:kolyvagin}
The goal of this section is to review Kolyvagin systems with a focus on Kolyvagin systems arising from Kato's Euler systems for the dual Galois representation and to explain
\begin{quote}
$\ks$ is primitive $\Rightarrow$ $\ks^{\infty}$ is $\Lambda$-primitive $\Rightarrow$ Kato's main conjecture,
\end{quote}
via Proposition \ref{prop:Lambda-primitivity} and Theorem \ref{thm:primitivity_kato_main_conjecture}. See \cite{mazur-rubin-book} and \cite{kazim-Lambda-adic} for detail.

\subsection{Local preliminaries}
Let $\rho_{\overline{f}}(1)$ be the dual representation defined in $\S$\ref{subsec:modular_galois_repns}.
Then $\rho_{\overline{f}}(1)$ also factors through $\mathrm{Gal}(\mathbb{Q}_\Sigma/\mathbb{Q})$.
For any prime $\ell \not\in \Sigma$, we define $P_\ell(x) \in  \mathbb{Z}_{f, \lambda}[x]$ by
\begin{align*}
P_\ell(x) & := \mathrm{det}(\mathrm{Id} - \rho_{\overline{f}}(1)(\mathrm{Fr}_\ell) x : T_{\overline{f}}(1)) \\
& = 1 - \overline{a_\ell(f)}\ell^{-1}x + \overline{\psi}(\ell)\ell^{-1} x^2 
\end{align*}
where $\mathrm{Fr}_\ell \in \mathrm{Gal}(\mathbb{Q}_{\Sigma}/\mathbb{Q})$ is the arithmetic Frobenius at $\ell$
as in \cite[Example 13.3]{kato-euler-systems}.

\subsection{Selmer structures}
Following \cite[$\S$2.1]{mazur-rubin-book}, we define the Selmer structure $\mathcal{F}$ on $T_{\overline{f}}(1)$ by
$$\mathrm{H}^1_{\mathcal{F}}(F, T_{\overline{f}}(1)) = \mathrm{Sel}(F,  T_{\overline{f}}(1))$$
where $\mathrm{Sel}(F,  T_{\overline{f}}(1))$ is the compact Selmer group (defined in terms of the orthogonal local conditions via the Tate local duality) and $F$ is an algebraic extension of $\mathbb{Q}$ as in \cite[$\S$6.2]{mazur-rubin-book} with help of $\S$\ref{subsec:selmer_groups}.
We also recall the ``canonical" Selmer structure $\mathcal{F}_{\mathrm{can}}$ on $T_{\overline{f}}(1)$ as in \cite[Definition 3.2.1]{mazur-rubin-book}.
The canonical structure $\mathcal{F}_{\mathrm{can}}$ is obtained from $\mathcal{F}$ by relaxing the condition at $p$; in other words,
$$\mathrm{H}^1_{ \mathcal{F}_{\mathrm{can}} }(\mathbb{Q}_\ell, T_{\overline{f}}(1)) = 
\left \lbrace
    \begin{array}{ll}
  \mathrm{H}^1(\mathbb{Q}_p, T_{\overline{f}}(1))  & \textrm{if} \ \ell =p \\ 
\mathrm{H}^1_{\mathcal{F}}(\mathbb{Q}_\ell, T_{\overline{f}}(1)) & \textrm{if} \ \ell \neq p .
    \end{array}
    \right.$$

\begin{defn}[Kolyvagin primes] \label{defn:kolyvagin_primes}
A rational prime $\ell$ is a \textbf{Kolyvagin prime (for $\rho_{\overline{f}}(1)$)} if it satisfies the following properties:
\begin{enumerate}
\item $\rho_{\overline{f}}(1)$ is unramified at $\ell$,
\item $\ell \equiv 1 \Mod{\lambda}$,
\item $\overline{a_\ell(f)} \equiv \ell + 1 \Mod{\lambda}$, and
\item $\overline{\psi}(\ell) \equiv 1 \Mod{\lambda}$.
\end{enumerate}
\end{defn}
From now on, we further assume that $\ell$ is a Kolyvagin prime.
Let $I_\ell \subset \mathbb{Z}_{f,\lambda}$ be the ideal generated by $\ell -1$ and $P_\ell(1)$.
Then  $I_\ell \subseteq \lambda \mathbb{Z}_{f,\lambda}$.
Let $I_n = \sum_{\ell \vert n} I_\ell \subseteq \mathbb{Z}_{f, \lambda}$.
Then the finite-singular map $\phi^{\mathrm{fs}}_\ell$ is defined by the commutative diagram
\[
\xymatrix@R=1.5em{
\mathrm{H}^1_\mathrm{fin} (\mathbb{Q}_\ell, T_{\overline{f}}(1)/I_{n\ell} T_{\overline{f}}(1) ) \ar[r]^-{\phi^{\mathrm{fs}}_\ell} \ar[d]^-{\simeq} & \mathrm{H}^1_\mathrm{sing} (\mathbb{Q}_\ell,  T_{\overline{f}}(1)/I_{n\ell}T_{\overline{f}}(1)) \otimes G_\ell \\
\frac{ T_{\overline{f}}(1)/I_{n\ell} T_{\overline{f}}(1)}{(\mathrm{Fr}_\ell - 1)  T_{\overline{f}}(1)/I_{n\ell} T_{\overline{f}}(1)}  \ar[r]^-{Q(\mathrm{Fr}^{-1}_\ell)} & \left(  T_{\overline{f}}(1)/I_{n\ell} T_{\overline{f}}(1)  \right)^{\mathrm{Fr}_\ell -1} \ar[u]^-{\simeq}
}
\]
where
$\mathrm{H}^1_\mathrm{fin}  = \mathrm{H}^1_\mathrm{ur} $ is the unramified cohomology group,
$\mathrm{H}^1_\mathrm{sing}  = \mathrm{H}^1  / \mathrm{H}^1_\mathrm{fin} $,
 $Q(x) = P_\ell(x)/(x-1)$, and $G_\ell := \mathrm{Gal}(\mathbb{Q}(\mu_\ell)/\mathbb{Q})$.

Let $\mathcal{F}(n)$ be the Selmer structure defined by $\mathcal{F}$ and the transverse local condition at primes dividing $n$ defined in \cite[Example 2.1.8]{mazur-rubin-book}.
We compare different Selmer structures as follows:
\[
\xymatrix@R=1.5em{
\mathrm{H}^1_{\mathcal{F}(n)}(\mathbb{Q},  T_{\overline{f}}(1)/I_{n} T_{\overline{f}}(1) ) \otimes G_n \ar[r]^-{\mathrm{loc}_\ell} & \mathrm{H}^1_{\mathrm{fin}}(\mathbb{Q}_\ell,  T_{\overline{f}}(1)/I_{n\ell} T_{\overline{f}}(1) ) \otimes G_n \ar[d]^-{ \phi^{\mathrm{fs}}_\ell \otimes 1 } \\
\mathrm{H}^1_{\mathcal{F}(n\ell)}(\mathbb{Q},  T_{\overline{f}}(1)/I_{n\ell} T_{\overline{f}}(1) ) \otimes G_{n\ell} \ar[r]^-{\mathrm{loc}_\ell} & \mathrm{H}^1_{\mathrm{sing}}(\mathbb{Q}_\ell,  T_{\overline{f}}(1)/I_{n\ell} T_{\overline{f}}(1) ) \otimes G_{n\ell} .
}
\]
where ${\displaystyle G_{n} := \otimes_{\ell \vert n} G_\ell }$.

\subsection{Selmer triples}
Let $\mathcal{P}$ be the set of Kolyvagin primes for $T_{\overline{f}}(1)$ and $\mathcal{N}$ be the set of square-free product of primes in $\mathcal{P}$.
Then we call $(T_{\overline{f}}(1), \mathcal{F}_{\mathrm{can}}, \mathcal{P})$ a \textbf{Selmer triple} and recall the basic assumptions on the triple as in \cite[$\S$3.5]{mazur-rubin-book} and \cite[$\S$2.2]{kazim-Lambda-adic}.
\begin{assu} \label{assu:abstract} $ $
\begin{itemize}
\item[(H.1)] $T_{\overline{f}}(1) / \lambda T_{\overline{f}}(1)$ is absolutely irreducible.
\item[(H.2)] There is a $\tau \in \mathrm{Gal}( \overline{\mathbb{Q}}/\mathbb{Q})$ such that $\tau = 1$ on $\mu_{p^\infty}$ and the $\mathbb{Z}_{f,\lambda}$-module $T_{\overline{f}}(1) / (\tau - 1)T_{\overline{f}}(1)$ is free of rank one.
\item[(H.3)] $\mathrm{H}^1( \mathbb{Q}(T_{\overline{f}}(1),\mu_{p^\infty})/\mathbb{Q}, T_{\overline{f}}(1) / \lambda T_{\overline{f}}(1) ) = \mathrm{H}^1( \mathbb{Q}(T_{\overline{f}}(1),\mu_{p^\infty})/\mathbb{Q}, A_f(1)[\lambda] ) = 0$. Here $\mathbb{Q}(T_{\overline{f}}(1))$ is the smallest extension of $\mathbb{Q}$ such that the $G_\mathbb{Q}$-action on $T_{\overline{f}}(1)$ factors through $\mathrm{Gal}(\mathbb{Q}(T_{\overline{f}}(1))/\mathbb{Q})$ and $\mathbb{Q}(T_{\overline{f}}(1),\mu_{p^\infty}) = \mathbb{Q}(T_{\overline{f}}(1))(\mu_{p^\infty})$.
\item[(H.4)] Either $\mathrm{Hom}_{\mathbb{F}_\lambda\llbracket G_\mathbb{Q} \rrbracket } \left( T_{\overline{f}}(1)/\lambda T_{\overline{f}}(1), A_f(1)[\lambda] \right)$ or $p > 4$.
\end{itemize}
\end{assu}
\begin{lem}[{\hspace{1sp}\cite[Lemma 6.2.3]{mazur-rubin-book}}]
Condition (Im) in Theorem \ref{thm:main_theorem} implies all the conditions of Assumption \ref{assu:abstract}.
\end{lem}
\begin{assu} \label{assu:more_conditions} $ $
\begin{itemize}
\item[(H.T)] Tamagawa condition : $\mathrm{H}^0 (I_\ell, V_{\overline{f}}(1)/T_{\overline{f}}(1))$ is divisible for every $\ell \neq p$ where $I_\ell$ is the inertia subgroup at $\ell$.
\item[(H.sEZ)] Strong exceptional zero-like condition : $\mathrm{H}^0(\mathbb{Q}_p, A_f(1)) = 0$.
\item[(H.EZ)] exceptional zero-like condition : $\mathrm{H}^0(\mathbb{Q}_p, A_f(1))$ is finite.
\end{itemize}
\end{assu}

\begin{lem} \label{lem:conditions} $ $
\begin{enumerate}
\item Condition (NA) in Theorem \ref{thm:main_theorem} ($a_p(f) \not\equiv 1 \Mod{\lambda}$) implies Assumption \ref{assu:more_conditions}.(H.sEZ).
\item Condition (Tam) in Theorem \ref{thm:main_theorem} implies Assumption \ref{assu:more_conditions}.(H.T). 
\end{enumerate}
\end{lem}
\begin{proof}
The first statement is obvious.
See \cite[Lemma 4.1.2]{epw} for the second statement.
\end{proof}
\begin{lem} \label{lem:the_set_of_kolyvagin_primes} The set $\mathcal{P}$ satisfies the following properties.
\begin{enumerate}
\item $T_{\overline{f}}(1)/(\mathrm{Fr}_\ell -1)T_{\overline{f}}(1)$ is a cyclic $\mathbb{Z}_{f, \lambda}$-module for every $\ell \in \mathcal{P}$.
\item $\mathrm{Fr}^{p^k}_\ell -1$ is injective on $T_{\overline{f}}(1)$ for every $\ell \in \mathcal{P}$ and every $k \geq 0$.
\end{enumerate}
\end{lem}
\begin{proof}
See \cite[Lemma 4.1.3]{rubin-book} with ``$\mathcal{R}_{\mathbb{Q}, p}$" in \cite[Definition 4.1.1]{rubin-book} with help of Assumption \ref{assu:abstract}.(H.2). Note that Assumption \ref{assu:abstract}.(H.2) comes from Condition (Im) in Theorem \ref{thm:main_theorem}.
\end{proof}

\subsection{Kolyvagin systems and generalized Kolyvagin systems}
\begin{defn}[Kolyvagin systems; {\cite[Definition 3.1.3]{mazur-rubin-book}}]
A \textbf{Kolyvagin system for the Selmer triple $(T_{\overline{f}}(1), \mathcal{F}_{\mathrm{can}}, \mathcal{P})$} is a collection of cohomology classes $\kappa_n \in \mathrm{H}^1_{\mathcal{F}_{\mathrm{can}}(n)}(\mathbb{Q}, T_{\overline{f}}(1) / I_n T_{\overline{f}}(1) )  \otimes G_{n}$ such that
if $\ell$ is a prime and $n\ell \in \mathcal{N}$, then
$$\left( \kappa_{n\ell} \right)_{\ell,s} = \phi^{\mathrm{fs}}_\ell (\kappa_n)$$
in $\mathrm{H}^1_\mathrm{sing}(\mathbb{Q}_\ell, T_{\overline{f}}(1)/I_{n\ell} T_{\overline{f}}(1)) \otimes G_{n\ell}$.
\end{defn}
Let $\KS(T_{\overline{f}}(1), \mathcal{F}_{\mathrm{can}}, \mathcal{P})$ be the $\mathbb{Z}_{f,\lambda}$-module of Kolyvagin systems and an element of $\KS(T_{\overline{f}}(1), \mathcal{F}_{\mathrm{can}}, \mathcal{P})$ is denoted by $\ks = ( \kappa_n )_n$
where $n$ runs over all square-free products of primes in $\mathcal{P}$, i.e. $n \in \mathcal{N}$.

\begin{defn}[generalized Kolyvagin systems; {\cite[Definition 3.1.6]{mazur-rubin-book}}]
Let $k \in \mathbb{N}$ and $\mathcal{P}_k$ be the set of places $\ell \not\in \Sigma(\mathcal{F}_{\mathrm{can}})$ such that
\begin{itemize}
\item $T_{\overline{f}}(1) / \left( \lambda^k T_{\overline{f}}(1) + \left(\mathrm{Fr}_\ell - 1\right) T_{\overline{f}}(1)\right)$ is free of rank one over $\mathbb{Z}_{f,\lambda} / \lambda^k$, and
\item $I_\ell \subset \lambda^k\mathbb{Z}_{f,\lambda}$.
\end{itemize}
Then we have a decreasing filtration
$$\cdots \subset\mathcal{P}_4 \subset\mathcal{P}_3 \subset\mathcal{P}_2 \subset\mathcal{P}_1 .$$
We define the $\mathbb{Z}_{f,\lambda}$-module of \textbf{generalized Kolyvagin system for $(T_{\overline{f}}(1), \mathcal{F}_{\mathrm{can}}, \mathcal{P})$} by
$$\overline{\KS}(T_{\overline{f}}(1), \mathcal{F}_{\mathrm{can}}, \mathcal{P}) := \varprojlim_k \left( \varinjlim_j \KS(T_{\overline{f}}(1) / \lambda^k T_{\overline{f}}(1) ,\mathcal{F}_{\mathrm{can}}, \mathcal{P} \cap \mathcal{P}_j  ) \right)$$ with respect to the functorial maps given in \cite[Remark 3.1.4]{mazur-rubin-book}.
\end{defn}

Let $\chi(T_{\overline{f}}(1), \mathcal{F}_{\mathrm{can}})$ be the core rank for the pair $(T_{\overline{f}}(1), \mathcal{F}_{\mathrm{can}})$ defined in \cite[Definition 4.1.11]{mazur-rubin-book}.
\begin{prop}[{\hspace{1sp}\cite[Proposition 6.2.2]{mazur-rubin-book}}]  \label{prop:core_rank_one}
$\chi(T_{\overline{f}}(1), \mathcal{F}_{\mathrm{can}}) = 1$.
\end{prop}
In our setting, the core rank $\chi(T_{\overline{f}}(1), \mathcal{F}_{\mathrm{can}})$ is equal to the rank of the minus part of $T_{\overline{f}}(1)$ under the complex conjugation over $\mathbb{Z}_{f,\lambda}$, not over $\mathbb{Z}_p$. See \cite[Theorem 5.2.15]{mazur-rubin-book} and \cite[$\S$6.3]{kato-euler-systems}.

\begin{prop}[{\hspace{1sp}\cite[Proposition 5.2.9]{mazur-rubin-book}}] \label{prop:comparison_koly_and_cpltd_koly}
If the core rank $\chi(T_{\overline{f}}(1), \mathcal{F}_{\mathrm{can}}) = 1$, then the canonical map 
$$\KS(T_{\overline{f}}(1), \mathcal{F}_{\mathrm{can}}, \mathcal{P}) \to \KSbar(T_{\overline{f}}(1), \mathcal{F}_{\mathrm{can}}, \mathcal{P})$$
is an isomorphism.
\end{prop}
Thus, we do not distinguish $\KS$ and $\KSbar$ for $T_{\overline{f}}(1)$.

\subsection{From Euler systems to Kolyvagin systems and {$\Lambda$}-adic Kolyvagin systems} \label{subsec:euler_to_kolyvagin}
We recall the Euler system to Kolyvagin system map and the Euler system to $\Lambda$-adic Kolyvagin system map.
Due to Lemma \ref{lem:the_set_of_kolyvagin_primes}, we can omit the conditions on $\mathcal{P}$ in \cite[Theorem 3.2.4 and Theorem 5.3.3]{mazur-rubin-book}.

Let $\ES(T_{\overline{f}}(1), \mathcal{P}, \mathbb{Q}^{\mathrm{ab}})$ be the $\mathbb{Z}_{f,\lambda}\llbracket \mathrm{Gal}(\overline{\mathbb{Q}}/\mathbb{Q}  )\rrbracket$-module of Euler systems as in \cite[$\S$3.2.2]{mazur-rubin-book} where $\mathbb{Q}^{\mathrm{ab}}$ is the maximal abelian extension of $\mathbb{Q}$ in $\overline{\mathbb{Q}}$.
Let $\mathbf{c} = (c_F)_F \in \ES(T_{\overline{f}}(1), \mathcal{P}, \mathbb{Q}^{\mathrm{ab}})$ be Kato's Euler system 
(Definition \ref{defn:kato_euler_systems}) with $c_F \in \mathrm{H}^1 (F, T_{\overline{f}}(1))$ where $F$ runs over finite abelian extensions of $\mathbb{Q}$. Note that we follow the convention of Euler systems in \cite{mazur-rubin-book} not in \cite{rubin-book}.

In other words, the collection of the cohomology classes satisfy the following relation.
For $F'/F$ finite extensions of $\mathbb{Q}$ in $\mathbb{Q}^{\mathrm{ab}}$,
\begin{align*}
\mathrm{Nm}_{F'/F} \left( c_{F'} \right) & = \left( \prod_{\ell \in S' \setminus S} \left( P_\ell \left( \sigma^{-1}_\ell \right) \right) \right) \cdot c_F \\
&= \left( \prod_{\ell \in S' \setminus S} \left( 1 - \overline{a_\ell(f)}\ell^{-1} \sigma^{-1}_\ell + \overline{\psi}(\ell)\ell^{-1} \sigma^{-2}_\ell \right) \right) \cdot c_F
\end{align*}
where $\sigma_\ell \in \mathrm{Gal}(F/\mathbb{Q})$ is the arithmetic Frobenius at $\ell$, and $S$ and $S'$ are the finite sets of finite places which are ramified in $F/\mathbb{Q}$ and $F'/\mathbb{Q}$, respectively. 
Here, we are assuming that $S$ and $S'$ are disjoint from a finite set of finite places including ones dividing $Np$.

\begin{rem} \label{rem:conjugate_euler_systems}
Let $\overline{c_{F}} \in \mathrm{H}^1(F, T_{\overline{f}}(1))$
be the complex conjugation of $c_{F}$.
Then we define
$$c^\pm_{F}  := \frac{1}{2} \cdot \left( c_{F} \pm \overline{c_{F}} \right),$$
respectively.
All the construction and the argument below also work with $c^+_{F}$ and $c^-_{F}$.
\end{rem}

Let $\ks = (\kappa_n)_n \in\KSbar(T_{\overline{f}}(1), \mathcal{F}_{\mathrm{can}}, \mathcal{P})$ be the Kato's Kolyvagin system as the image of $\mathbf{c}^+$ under the map below where $\kappa_n \in \mathrm{H}^1_{ \mathcal{F}(n) } (\mathbb{Q}, T_{\overline{f}}(1))$.
\begin{thm}[{\hspace{1sp}\cite[Theorem 3.2.4]{mazur-rubin-book}}] \label{thm:euler_to_kolyvagin}
There is a canonical Galois equivariant morphism
$$\ES(T_{\overline{f}}(1), \mathcal{P}, \mathbb{Q}^{\mathrm{ab}}) \to \KSbar(T_{\overline{f}}(1), \mathcal{F}_{\mathrm{can}}, \mathcal{P}) $$
such that if $$\mathbf{c}^+ \mapsto \ks$$ then $$\kappa_1 = c^+_{\mathbb{Q}} = c_{\mathbb{Q}}.$$
Furthermore, under Assumption \ref{assu:more_conditions}.(sEZ), the statement holds not only with $\KSbar$ but also with $\KS$.
\end{thm}
Let $\mathcal{F}_\Lambda$ be the Selmer structure for $T_{\overline{f}}(1) \otimes_{\mathbb{Z}_{f,\lambda}} \Lambda$
such that $$\mathrm{H}^1_{\mathcal{F}_\Lambda} (\mathbb{Q}, T_{\overline{f}}(1) \otimes_{\mathbb{Z}_{f,\lambda}} \Lambda) = \mathrm{H}^1 (\mathbb{Q}, T_{\overline{f}}(1) \otimes_{\mathbb{Z}_{f,\lambda}} \Lambda)$$ defined in \cite[Definition 5.3.2]{mazur-rubin-book}.
Let $\KSbar(T_{\overline{f}}(1) \otimes_{\mathbb{Z}_{f,\lambda}} \Lambda, \mathcal{F}_{\Lambda}, \mathcal{P})$ be the $\mathbb{Z}_{f,\lambda}$-module of generalized $\Lambda$-adic Kolyvagin systems and
$\ks^{\infty} = (\kappa^{\infty}_n)_n \in \KSbar(T_{\overline{f}}(1) \otimes_{\mathbb{Z}_{f,\lambda}} \Lambda, \mathcal{F}_{\Lambda}, \mathcal{P})$ be the Kato's $\Lambda$-adic Kolyvagin system for $T_{\overline{f}}(1) \otimes_{\mathbb{Z}_{f,\lambda}} \Lambda$ as the image of the map below where $\kappa^{\infty}_n \in \mathrm{H}^1_{ \mathcal{F}_{\Lambda}(n) } (\mathbb{Q}, T_{\overline{f}}(1) \otimes_{\mathbb{Z}_{f,\lambda}} \Lambda)$. 
\begin{thm}[{\hspace{1sp}\cite[Theorem 5.3.3]{mazur-rubin-book}}] \label{thm:euler_to_Lambda-adic}
There is a canonical homomorphism 
$$\ES(T_{\overline{f}}(1), \mathcal{P}, \mathbb{Q}^{\mathrm{ab}}) \to \KSbar(T_{\overline{f}}(1) \otimes_{\mathbb{Z}_{f,\lambda}} \Lambda, \mathcal{F}_{\Lambda}, \mathcal{P}) $$
such that if $$\mathbf{c}^+ \mapsto \ks^\infty$$
then 
$$\kappa^{\infty}_1 =  \varprojlim_r c^+_{\mathbb{Q}_r} =  \varprojlim_r c_{\mathbb{Q}_r} \in \varprojlim_r \mathrm{H}^1(\mathbb{Q}_r, T_{\overline{f}}(1)) = \mathrm{H}^1(\mathbb{Q}, T_{\overline{f}}(1) \otimes_{\mathbb{Z}_{f,\lambda}} \Lambda ) .$$
\end{thm}
The non-triviality of $\ks^\infty$, in fact, $\kappa^{\infty}_1 \neq 0$, is due to the result of Rohrlich on non-vanishing of twisted $L$-values \cite{rohrlich-nonvanishing} and the dual exponential map.
\begin{rem}
By \cite[Remark 3.24]{kazim-Lambda-adic}, two Selmer structures $\mathcal{F}_\Lambda$ and $\mathcal{F}_{\mathrm{can}}$ induce the same module of generalized Kolyvagin systems for $T_{\overline{f}}(1) \otimes_{\mathbb{Z}_{f,\lambda}} \Lambda$.
\end{rem}

\begin{thm}[{\hspace{1sp}\cite[Theorem 3.23]{kazim-Lambda-adic}}] \label{thm:specialization_surjective}
Under Assumption \ref{assu:abstract} and Assumption \ref{assu:more_conditions}, we have the following statements.
\begin{enumerate}
\item The module $\KSbar(T_{\overline{f}}(1) \otimes_{\mathbb{Z}_{f,\lambda}} \Lambda, \mathcal{F}_{\mathrm{can}}, \mathcal{P})$ is free of rank one over $\Lambda$.
\item The specialization map is surjective and forms the following commutative diagram
\[
{ \scriptsize
\xymatrix{
\KSbar(T_{\overline{f}}(1) \otimes_{\mathbb{Z}_{f,\lambda}} \Lambda, \mathcal{F}_{\mathrm{can}}, \mathcal{P}) \ar@{->>}[r] \ar[d]^-{\simeq}&
\KSbar(T_{\overline{f}}(1), \mathcal{F}_{\mathrm{can}}, \mathcal{P}) \ar[d]^-{\simeq}\\
\Lambda \ar@{->>}[r] & \Lambda / (\gamma - 1) \simeq \mathbb{Z}_{f,\lambda}
}
}
\]
\end{enumerate}
\end{thm}
To sum up, we have the following commutative diagram
\[
{ \scriptsize
\xymatrix{
& \KSbar(T_{\overline{f}}(1) \otimes_{\mathbb{Z}_{f,\lambda}} \Lambda, \mathcal{F}_{\mathrm{can}}, \mathcal{P})   \ar@{->>}[d] & & \ks^\infty \ar@{|->}[d] \\
\ES(T_{\overline{f}}(1), \mathcal{P}, \mathbb{Q}^{\mathrm{ab}}) \ar[r] \ar[ru] & \KSbar(T_{\overline{f}}(1), \mathcal{F}_{\mathrm{can}}, \mathcal{P}) & \mathbf{c} \ar@{|->}[r] \ar@{|->}[ru] & \ks
}
}
\]

\subsection{Primitivity and {$\Lambda$}-primitivity}
For any subquotient $M$ of $T_{\overline{f}}(1) \otimes_{\mathbb{Z}_{f,\lambda}} \Lambda$,
set $\KS(M) := \KS(M, \mathcal{F}_{\mathrm{can}}, \mathcal{P})$ and $\KSbar(M) := \KSbar(M, \mathcal{F}_{\mathrm{can}}, \mathcal{P})$.
Let $\ks \in \KSbar(T_{\overline{f}}(1))$ be the Kato's Kolyvagin system for $T_{\overline{f}}(1)$.
\begin{defn}[{\hspace{1sp}\cite[Definition 4.5.5]{mazur-rubin-book}}] \label{defn:primitivity}
We call $\ks$ \textbf{primitive} if the image of $\ks$ in $\KSbar(T_{\overline{f}}(1)/\lambda T_{\overline{f}}(1) )$ is non-zero.
\end{defn}
Let $\ks^\infty \in \KSbar(T_{\overline{f}}(1) \otimes_{\mathbb{Z}_{f,\lambda}} \Lambda)$ be the $\Lambda$-adic Kato's Kolyvagin system.
\begin{defn}[{\hspace{1sp}\cite[$\S$3.1]{mazur-rubin-book}}]
The \textbf{blind spot of $\ks^\infty$} is the set of ideals $I \subset \Lambda$ such that the image of $\ks^\infty$ under the natural map
$$\KSbar(T_{\overline{f}}(1) \otimes_{\mathbb{Z}_{f,\lambda}} \Lambda) \to \KSbar(T_{\overline{f}}(1) \otimes_{\mathbb{Z}_{f,\lambda}}  \left(  \Lambda  /I  \right) )$$
is zero.
\end{defn}
\begin{defn}[{\hspace{1sp}\cite[Definition 5.3.9]{mazur-rubin-book}}] \label{defn:Lambda-primitivity}
We call $\ks^\infty$ \textbf{$\Lambda$-primitive} if the blind spot of $\ks^\infty$ contains no height-one primes of $\Lambda$.
\end{defn}
The following proposition is a slight variant of \cite[Proposition 4.1 and Proposition 4.2]{kazim-Lambda-adic}.
\begin{prop}[B\"{u}y\"{u}kboduk] \label{prop:Lambda-primitivity}
Under the assumptions in Theorem \ref{thm:main_theorem}, if $\ks$ is primitive, then $\ks^\infty$ is $\Lambda$-primitive.
\end{prop}
\begin{proof}
Let $\mathfrak{p} \subset \Lambda$ be a height one prime ideal.
Consider the commutative diagram
{ \scriptsize
\[
\xymatrix{
\KSbar(T_{\overline{f}}(1) \otimes_{\mathbb{Z}_{f,\lambda}} \Lambda) \ar[r] \ar[dr] \ar[d] & \KSbar(T_{\overline{f}}(1) \otimes_{\mathbb{Z}_{f,\lambda}} \Lambda/\mathfrak{p} ) \ar[d]  & \ks^\infty \ar@{|->}[r]  \ar@{|->}[dr] \ar@{|->}[d]  & \ks^\infty \Mod{\mathfrak{p}} \ar@{|->}[d]  \\
\KSbar(T_{\overline{f}}(1) ) \ar[r] & \KSbar(T_{\overline{f}}(1) / \lambda T_{\overline{f}}(1) ) & \ks = \ks^{\infty} \Mod{(\gamma -1)} \ar@{|->}[r] & \ks \Mod{\lambda} = \overline{\ks^\infty} .
}
\]
}
Since $\ks$ is primitive, $\ks \Mod{\lambda}$ is non-zero. Thus, the residual image $\overline{\ks^\infty}$ is also non-zero.
Thus, $\ks^\infty \Mod{\mathfrak{p}}$ cannot be zero for any height one prime ideal $\mathfrak{p} \subseteq \Lambda$.
\end{proof}

Considering Proposition \ref{prop:reduction-of-kato-main-conjecture}, we have the following statement.
\begin{thm}[{\hspace{1sp}\cite[Theorem 5.3.10.(iii)]{mazur-rubin-book}}] \label{thm:primitivity_kato_main_conjecture}
If $\ks^\infty$ is $\Lambda$-primitive, then the equality of Kato's main conjecture (Conjecture \ref{conj:kato-main-conjecture}) holds for $(T_{\overline{f}}(1), \mathbb{Q}_\infty/\mathbb{Q})$.
\end{thm}
The following theorem directly follows from Theorem \ref{thm:primitivity_kato_main_conjecture} and \cite[$\S$17.13]{kato-euler-systems}.
\begin{thm}[{\hspace{1sp}\cite[Theorem 6.2.7]{mazur-rubin-book}}] \label{thm:primitivity_main_conjecture}
If $a_p(f)$ is a $\lambda$-adic unit with $a_p(f) \not\equiv 1 \Mod{\lambda}$ and $\ks^\infty$ is $\Lambda$-primitive, then the equality of the Iwasawa main conjecture \`{a} la Mazur--Greenberg (Conjecture \ref{conj:iwasawa_main_conjecture}) holds for $(A_f(1), \mathbb{Q}_\infty/\mathbb{Q})$.
\end{thm}

\section{The image of the dual exponential map for unramified extensions} \label{sec:the_image_of_dual_exp}
The goal of this section is to explain the following diagram
\[
{ \scriptsize
\xymatrix@R=1.7em{
\mathrm{Hom}_{\mathbb{Q}_{f, \lambda}}(J_1(N)_{\overline{f},\lambda} (\mathbb{Q}_p(\mu_n)) \otimes_{\mathbb{Z}_{f,\lambda}} \mathbb{Q}_{f, \lambda}, \mathbb{Q}_{f, \lambda} ) \ar[d]^-{\simeq}_-{\textrm{Kummer map}} \\
\mathrm{Hom}_{\mathbb{Q}_{f, \lambda}}(\mathrm{H}^1_f( \mathbb{Q}_p(\mu_n), V_f(1)), \mathbb{Q}_{f, \lambda}) \ar[d]^-{\simeq}_-{\textrm{Tate duality}}  \\
\mathrm{H}^1_s( \mathbb{Q}_p(\mu_n), V_{\overline{f}}(1)) \ar[d]_-{ \substack{\mathrm{exp}^{*} \\(\textrm{Theorem \ref{thm:kato_formula}})}}^-{\simeq} & \mathrm{H}^1_s( \mathbb{Q}_p(\mu_n), T_{\overline{f}}(1)) \ar@{_{(}->}[l]_-{\textrm{\cite[$\S$3.5.(3.14)]{rubin-book}}}^-{\textrm{\cite[Lemma 1.2.2.(ii)]{rubin-book}}} \ar[dd]^-{\simeq}_{ \langle \omega^*_{\overline{f}},\mathrm{exp}^{*}(-)\rangle_{\mathrm{dR}} } \\
S(\overline{f}) \otimes_{\mathbb{Q}_f} \mathbb{Q}_{f, \lambda} \otimes_{\mathbb{Q}_p} \mathbb{Q}_p(\mu_n) \ar[d]^-{\simeq}_-{\langle \omega^*_{\overline{f}},-\rangle_{\mathrm{dR}}} \\
 \mathbb{Q}_{f, \lambda} \otimes_{\mathbb{Q}_p} \mathbb{Q}_p(\mu_n) & \mathscr{L} \ar@{_{(}->}[l]
}
}
\]
and determine the image $\mathscr{L}$ of the $\mathbb{Z}_{f, \lambda} \otimes \mathbb{Z}_p[\mu_n]$-lattice $\mathrm{H}^1_s(\mathbb{Q}_p(\mu_n), T_{\overline{f}}(1))$ in $\mathbb{Q}_{f, \lambda} \otimes_{\mathbb{Q}_p} \mathbb{Q}_p(\mu_n)$.
In order to do this, we compute the image of $J_1(N)_{f,\lambda}(\mathbb{Q}_p(\mu_n)) $ under the composition of the de Rham pairing with $\omega_{\overline{f}}$ and the logarithm map and use Kato's explicit formula (Theorem \ref{thm:kato_formula}) via the Tate local duality. It is an explicit description of the integral structure of \cite[Example 3.11]{bloch-kato}.
Since the de Rham pairing and the period integral are closely related via the Eichler--Shimura isomorphism, we also explain this comparison and the integral canonical periods.
This section can be regarded as a generalization of \cite[Proposition 3.5.1]{rubin-book} to modular abelian varieties of $\mathrm{GL}_2$-type and an explicit description of \cite[Lemma 14.18.(ii)]{kato-euler-systems} for $\mathbb{Q}_p(\mu_n)$.
Note that we crucially use the residual irreducibility of $\overline{\rho}$ and the good reduction property of $f$ at $p$ in this section.
\subsection{The local condition at $p$}
The local condition
$\mathrm{H}^1_f(\mathbb{Q}_p(\mu_n), V_f(1))$
at $p$ is defined by the image of the Kummer map
$$\mathrm{H}^1_f(\mathbb{Q}_p(\mu_n), V_f(1)) := \mathrm{Im}\left( J_1(N)_{\overline{f},\lambda}(\mathbb{Q}_p(\mu_n)) \otimes \mathbb{Q}_{p} \hookrightarrow \mathrm{H}^1(\mathbb{Q}_p(\mu_n), V_f(1))  \right) .$$
Then
 $\mathrm{H}^1_f(\mathbb{Q}_p(\mu_n), V_f(1))$
 and
 $\mathrm{H}^1_f(\mathbb{Q}_p(\mu_n), V_{\overline{f}}(1))$
are orthogonal complements with respect to the local Tate pairing.

\subsection{Kato's explicit formula} \label{subsec:kato_explicit_formula}
Let $K$ be a finite extension of $\mathbb{Q}_p$.
For a de Rham representation $V$ of $G_K$, we recall Fontaine's de Rham functor $\mathbf{D}_{\mathrm{dR},K}(V) := \left(V \otimes_{\mathbb{Q}_p} \mathbf{B}_{\mathrm{dR}} \right)^{G_{K}}$ and
$\mathbf{D}^i_{\mathrm{dR},K}(V) := \left(V \otimes_{\mathbb{Q}_p} t^i \mathbf{B}^+_{\mathrm{dR}} \right)^{G_{K}}$
where $\mathbf{B}_{\mathrm{dR}}$ is the de Rham period ring \`{a} la Fontaine, $\mathbf{B}^+_{\mathrm{dR}} \subseteq \mathbf{B}_{\mathrm{dR}}$ is the valuation ring of $\mathbf{B}_{\mathrm{dR}}$, and $t$ is a uniformizer of $\mathbf{B}_{\mathrm{dR}}$.
For finite extensions $K_1/K_2$ of $\mathbb{Q}_p$, we have an isomorphism
$\mathbf{D}_{\mathrm{dR},K_1} (V) \simeq \mathbf{D}_{\mathrm{dR}, K_2}(V) \otimes_{K_2} K_1$
preserving the de Rham filtration. if $K_1/K_2$ is Galois, then it is also $\mathrm{Gal}(K_1/K_2)$-equivariant.
We write $\mathbf{D}_{\mathrm{dR},n}(V) = \mathbf{D}_{\mathrm{dR},\mathbb{Q}_p(\mu_n)}(V)$ and $\mathbf{D}_{\mathrm{dR}}(V) = \mathbf{D}_{\mathrm{dR},\mathbb{Q}_p}(V)$.
Since our representation is crystalline, these $\mathbf{D}_{\mathrm{dR}}$'s admit the natural action of Frobenius $\varphi$. 

Following \cite[Chapter II, Theorem 1.4.1.(4)]{kato-lecture-1} and \cite[(11.3.4) and Theorem 12.5]{kato-euler-systems}, we have the following formula.
\begin{thm}[Kato's explicit formula]  \label{thm:kato_formula}
The Bloch--Kato dual exponential map
$$\mathrm{exp}^{*}: \mathrm{H}^1(\mathbb{Q}_p(\mu_n), V_{\overline{f}}(1)) \to \mathbf{D}_{\mathrm{dR},n}(V_{\overline{f}}(1))$$
coincides with the composition of maps
\[
{\scriptsize
\xymatrix{
\mathrm{H}^1_s(\mathbb{Q}_p(\mu_n), V_{\overline{f}}) \ar[d]^-{\simeq}_{\textrm{local duality}} & \mathbf{D}^0_{\mathrm{dR},n}(V_{\overline{f}}(1))  \ar[r]_-{\simeq}^-{ M \mapsto M(-1) }  &  \mathbf{D}^1_{\mathrm{dR},n}(V_{\overline{f}}) \ar[d]^-{\simeq}_-{\substack{ \textrm{de Rham-\'{e}tale} \\ \textrm{comparison} \\ \textrm{ \cite[(11.3.4)]{kato-euler-systems} }}} \\
 \mathrm{Hom}_{\mathbb{Q}_{f, \lambda}} \left( \mathrm{H}^1_f(\mathbb{Q}_p(\mu_n), V_f(1)), \mathbb{Q}_{f, \lambda} \right) \ar[r]^-{ \mathrm{Hom}_{\mathbb{Q}_{f, \lambda}} \left( \mathrm{exp}, \mathbb{Q}_{f, \lambda} \right)  } & \mathrm{Hom}_{\mathbb{Q}_{f, \lambda}} \left( \mathbf{D}_{\mathrm{dR},n}(V_f(1)) / \mathbf{D}^0_{\mathrm{dR},n}(V_{f}(1)), \mathbb{Q}_{f, \lambda} \right) \ar[u]^-{\simeq}_-{\mathrm{Tr} \circ \langle - , - \rangle_{\mathrm{dR}}(1)} & S(\overline{f}) \otimes_{\mathbb{Q}_f} \mathbb{Q}_{f,\lambda} \otimes_{\mathbb{Q}_p}\mathbb{Q}_p(\mu_n)
}
}
\]
where 
\[
{\scriptsize
\xymatrix{
\mathrm{Tr} \circ \langle - , - \rangle_{\mathrm{dR}} :
\dfrac{ \mathbf{D}_{\mathrm{dR},n}(V_f(1)) }{ \mathbf{D}^0_{\mathrm{dR},n}(V_{f}(1)) }  \times  \mathbf{D}^{0}_{\mathrm{dR},n}(V_{\overline{f}}(1)) \ar[rr]^-{\langle - , - \rangle_{\mathrm{dR}}} & & \mathbf{D}_{\mathrm{dR}, n}(\mathbb{Q}_{f,\lambda}(1)) \simeq \mathbb{Q}_{f,\lambda} \otimes \mathbb{Q}_p(\mu_n) \ar[r]^-{\mathrm{Tr}} & \mathbb{Q}_{f,\lambda}
}
}
\]
is the composition of the trace map and the $\mathbb{Q}_{f,\lambda}$-sesqui-linear de Rham pairing, i.e. the pairing is $\mathbb{Q}_{f,\lambda}$-conjugate-linear for the second term.
\end{thm}
An explicit description of the paring $\langle - , - \rangle_{\mathrm{dR}}$ is explored by comparing with the local duality and the period integral in (\ref{eqn:comparison-local-duality-dR-paring-period-integral}). See (\ref{eqn:explicit-formula-pairing}) for an explicit formula for the pairing.

\subsection{Tangent spaces, cotangent spaces, and their integral lattices}
Following \cite[Example 3.11]{bloch-kato} and \cite[$\S$2.2.2]{kurihara-invent}, we recall the notion of the tangent spaces in terms of $\mathbf{D}_{\mathrm{dR}}$.
We define the \textbf{tangent space of $J_1(N)_{f,\lambda}(K)$} by $\mathbb{Q}_{f,\lambda}$-vector space
$$\mathbf{D}_{\mathrm{dR},K}(V_f(1)) / \mathbf{D}^0_{\mathrm{dR},K}(V_{f}(1)).$$
If $K/\mathbb{Q}_p$ is Galois, then it admits the natural action of $\mathrm{Gal}(K/\mathbb{Q}_p)$.
Consider the Lie group exponential map
\[
\xymatrix{
\mathbf{D}_{\mathrm{dR},n}(V_f(1)) / \mathbf{D}^0_{\mathrm{dR},n}(V_{f}(1))  \ar[r]^-{\simeq}  & J_1(N)_{\overline{f},\lambda} ( \mathbb{Q}_p(\mu_n) ) \otimes_{\mathbb{Z}_{f,\lambda}} \mathbb{Q}_{f,\lambda} .
}
\]
Also, with the Kummer map, we have the Bloch--Kato exponential map, which yields the following isomorphism of $\mathbb{Q}_{f,\lambda}[\mathrm{Gal}( \mathbb{Q}_p(\mu_n)/ \mathbb{Q}_p ) ]$-modules
\[
\xymatrix{
\mathrm{exp}: \mathbf{D}_{\mathrm{dR},n}(V_f(1)) / \mathbf{D}^0_{\mathrm{dR},n}(V_{f}(1))  \ar[r]^-{\simeq}  & \mathrm{H}^1_f ( \mathbb{Q}_p(\mu_n),  V_f(1)).
}
\]
\begin{defn}
We define the \textbf{canonical integral lattice} 
$$\mathbf{D}_{\mathrm{dR},n}(T_f(1)) / \mathbf{D}^0_{\mathrm{dR},n}(T_{f}(1)) \subseteq \mathbf{D}_{\mathrm{dR},n}(V_f(1)) / \mathbf{D}^0_{\mathrm{dR},n}(V_{f}(1))$$
by the inverse image of the torsion-free part of $\mathrm{H}^1_f ( \mathbb{Q}_p(\mu_n),  T_f(1))$ under the Bloch--Kato exponential map.
\end{defn}
The canonical integral lattice coincides with the $\mathbb{Z}_{f, \lambda}$-component of the integral tangent space of the $\mathbb{Q}_p(\mu_n)$-points of the N\'{e}ron model of $J_1(N)_{\overline{f}}$ since $T_f(1)$ is naturally isormorphic to the $\lambda$-adic Tate module of $J_1(N)_{\overline{f}, \lambda}$.

By the interpretation of modular forms in terms of $p$-adic Hodge theory as in \cite[(11.3.4)]{kato-euler-systems},
we define the \textbf{cotangent space} and normalize its integral one by
\[
\xymatrix{
\mathbf{D}^0_{\mathrm{dR}}(V_f(1))  \simeq \mathbf{D}^1_{\mathrm{dR}}(V_f)  \simeq S(f) \otimes_{\mathbb{Q}_f} \mathbb{Q}_{f,\lambda}, &
\mathbf{D}^0_{\mathrm{dR}}(T_f(1))  \simeq \mathbf{D}^1_{\mathrm{dR}}(T_f)  \simeq \mathbb{Z}_f \cdot \omega_f \otimes_{\mathbb{Z}_f} \mathbb{Z}_{f,\lambda} .
}
\]
More explicitly, the integral lattice $\mathbf{D}_{\mathrm{dR},\mathbb{Q}_p}(T_f(1)) / \mathbf{D}^0_{\mathrm{dR},\mathbb{Q}_p}(T_{f}(1)) \otimes_{\mathbb{Z}_p} \mathbb{Z}_p[\mu_n]$ is generated by the dual basis $\omega^*_{f} \in \mathbf{D}_{\mathrm{dR},n}(V_f(1)) / \mathbf{D}^0_{\mathrm{dR},n}(V_{f}(1))$ to $\omega_f := f(z)dz$ over $\mathbb{Z}_{f,\lambda} \otimes  \mathbb{Z}_p[\mu_n]$ such that $\langle \omega^*_{f}, \omega_f \rangle_{\mathrm{dR}} = 1$.
We will explain later in $\S$\ref{subsec:eichler-shimura} that the dual basis is explicitly described in terms of integral canonical periods.

\subsection{Mod $p$ multiplicity one and integral canonical periods} \label{subsec:mod_p_multi_one}
We recall the notion of integral canonical periods following \cite[$\S$3]{vatsal-integralperiods-2013}.
The existence of integral canonical periods requires a mod $p$ multiplicity one result established by Mazur, Wiles, and others, under the residual irreducibility assumption.

For a module $M$, let $M^\pm$ be the submodule of $M$ on which the complex conjugation acts by $\pm 1$, respectively,  and $\mathfrak{m}$ be the maximal ideal of $\mathbb{T}$ corresponding to $\overline{\rho}$ as in $\S$\ref{subsec:hecke_algebra}.

By \cite[Theorem 2.1.(i)]{wiles} with Condition (Im) in Theorem \ref{thm:main_theorem} and $(N,p)=1$, the Hecke module $\mathrm{H}_1(X_1(N), \mathbb{Z}_p)^{\pm}_{\mathfrak{m}}$ is free of rank one over $\mathbb{T}_\mathfrak{m}$ and let $\gamma^\pm$ be a generator of $\mathrm{H}_1(X_1(N), \mathbb{Z}_p)^{\pm}_{\mathfrak{m}}$ over $\mathbb{T}_\mathfrak{m}$, respectively.
Multiplying by a unit if necessary, we may assume that $\gamma^\pm \in \mathrm{H}_1(X_1(N), \mathbb{Z})^\pm$ following \cite[$\S$3.1]{vatsal-integralperiods-2013}. 
Then the pairing via the period integral 
\[
\xymatrix@R=0em{
 \mathrm{H}_1(X_1(N), \mathbb{Z})^\pm \times S_2(\Gamma_1(N), \mathbb{C}) \ar[r] & \mathbb{C} \\
( \gamma^{\pm} , f ) \ar@{|->}[r] & \int_{\gamma^{\pm}} \omega_f
}
\]
yields the values
$$\Omega^\pm_f := \int_{\gamma^\pm} \omega_f \in \mathbb{C}^\times$$
and we call them the $(\pm)$-part of the \textbf{integral canonical periods of $f$} if $\wp_f \subseteq \mathfrak{m}$.
The periods $\Omega^\pm_f$ are defined up to multiplication by $\mathbb{Q}^\times_f \cap \mathbb{Z}^\times_{f, \lambda}$, i.e. $\lambda$-adic units.  If we define the periods with $f \in S_2(\Gamma_1(N), \overline{\mathbb{Q}}_p)$, the periods depend on the identification $\iota: \mathbb{C} \simeq \overline{\mathbb{Q}}_p$.
Furthermore, the integral canonical periods vary integrally in Hida families (\hspace{1sp}\cite[$\S$3]{epw}). It is the essence of simultaneous vanishing of $\mu$-invariants in Hida families (\hspace{1sp}\cite[Theorem 1]{epw}).

Let $\gamma^{\pm}_f$ be the generator of the free $\mathbb{Z}_{f,\lambda}$-module $\mathrm{H}_1(X_1(N), \mathbb{Z}_p)^{\pm}_{\mathfrak{m}} \otimes_{\mathbb{T}_{\mathfrak{m}}} \mathbb{Z}_{f,\lambda}$ of rank one, which is induced from the chosen generator $\gamma^\pm$.
Then the period integral naturally induces the pairing between one-dimensional $\mathbb{C}$-vector spaces
\[
\xymatrix@R=0em{
\left( \mathrm{H}_1(X_1(N), \mathbb{Z}_p)^{\pm}_{\mathfrak{m}} \otimes_{\mathbb{T}_{\mathfrak{m}}} \mathbb{Z}_{f,\lambda} \otimes_{\mathbb{Z}_{f,\lambda}, \iota^{-1}} \mathbb{C} \right) \times S(f) \otimes_{\mathbb{Q}_f} \mathbb{C} \ar[r] & \mathbb{C}  .
}
\]
Due to the irreducibility of $\overline{\rho}$, it is easy to see that the values 
$\left[ \dfrac{a}{n} \right]_f$, 
$\left[ \dfrac{a}{n} \right]^+_f$, and
$\left[ \dfrac{a}{n} \right]^-_f$ lie in $\mathbb{Z}_{f, \lambda}$
for any integer $a$ and $n$ with $(n,N) = 1$.

\subsection{The de Rham pairing, the period integral, and Eichler--Shimura} \label{subsec:eichler-shimura}
In order to utilize Kato's explicit formula (Theorem \ref{thm:kato_formula}), we need to work on the de Rham side.
Thus, we need to compare the Betti homology of modular curves (appeared in $\S$\ref{subsec:mod_p_multi_one}) and the dual space to the de Rham cohomology of modular curves (appeared in $\S$\ref{subsec:kato_explicit_formula})  via the Betti--de Rham comparison (Eichler--Shimura isomorphism).

From now on, we only cover the $(+)$-part because we focus on the totally real extension $\mathbb{Q}_{\infty}/\mathbb{Q}$.
We summarize the comparison between the Betti side and the de Rham side in the following diagram.
\begin{equation} \label{eqn:comparison-betti-de-rham}
\begin{split}
{ \scriptsize
\xymatrix@R=1.5em{
\mathrm{H}_1(X_1(N), \mathbb{Z}_p)^+ \ar[d] & \gamma^+ \ar@{|->}[d] \\
\mathrm{H}_1(X_1(N), \mathbb{Z}_p)^+ \otimes_{\mathbb{T}} \mathbb{Q}_{f,\lambda} \ar[d]_-{\otimes_{\mathbb{Q}_{f,\lambda}}\mathbb{C}} & \gamma^+_f  \ar@{|->}[dd] \\
\mathrm{H}_1(X_1(N), \mathbb{Z}_p)^+ \otimes_{\mathbb{T}} \mathbb{Q}_{f,\lambda} \otimes_{\mathbb{Q}_{f,\lambda}} \mathbb{C} \ar[d]^-{\simeq}_-{\textrm{Betti--de Rham comparison (Eichler-Shimura)}} \\
\mathrm{Hom}_{\mathbb{Q}_{f,\lambda}}(\mathrm{H}^0(X_1(N)_{\mathbb{Q}_p}, \Omega^1_{X_1(N)_{\mathbb{Q}_p}/\mathbb{Q}_p})\otimes_{\mathbb{T}} \mathbb{Q}_{f,\lambda} , \mathbb{Q}_{f,\lambda}) \otimes_{\mathbb{Q}_{f,\lambda}} \mathbb{C} & \gamma^+_f = \Omega^+_{f} \cdot \omega^*_{\overline{f}}  \\
\mathrm{Hom}_{\mathbb{Q}_{f,\lambda}}(\mathrm{H}^0(X_1(N)_{\mathbb{Q}_p}, \Omega^1_{X_1(N)_{\mathbb{Q}_p}/\mathbb{Q}_p}) \otimes_{\mathbb{T}} \mathbb{Q}_{f,\lambda} , \mathbb{Q}_{f,\lambda}) \ar[u]^-{\otimes_{\mathbb{Q}_{f,\lambda}}\mathbb{C}}  \\
\mathrm{Hom}_{\mathbb{Q}_{f,\lambda}}(S(\overline{f}) \otimes_{\mathbb{Q}_{f}}\mathbb{Q}_{f,\lambda}, \mathbb{Q}_{f,\lambda}) \ar[u]_-{\simeq} &  \\
\mathrm{Hom}_{\mathbb{Q}_{f,\lambda}}(\mathbf{D}^1_{\mathrm{dR}}(V_{\overline{f}}), \mathbb{Q}_{f,\lambda}) \ar[u]_-{\simeq}^-{\textrm{de Rham--\'{e}tale comparison, \cite[(11.3.4)]{kato-euler-systems}}} & \\
\mathrm{Hom}_{\mathbb{Q}_{f,\lambda}}(\mathbf{D}^0_{\mathrm{dR}}(V_{\overline{f}}(1)), \mathbb{Q}_{f,\lambda}) \ar[u]_-{\simeq}^-{M \mapsto M(-1)} & \omega^*_{\overline{f}} \ar@{|->}[uuuu]   
}
}
\end{split}
\end{equation}
We can also easily obtain the $(-)$-part by looking at the whole first de Rham cohomology. The cup product, the de Rham pairing and the period integral can be also compared as follows.
\begin{equation} \label{eqn:comparison-local-duality-dR-paring-period-integral}
\begin{gathered}
{ \small
\xymatrix@C=0.5em{
\mathrm{H}^1_f(\mathbb{Q}_p, V_f(1)) \ar@/_2pc/[d]_-{\mathrm{log}}^-{\simeq} & \times & \mathrm{H}^1_s(\mathbb{Q}_p, V_{\overline{f}}(1)) \ar[rr]^-{\cup} \ar[d]_-{\mathrm{exp}^{*}}^-{\simeq} & & \mathrm{H}^2 (\mathbb{Q}_p, \mathbb{Q}_{f,\lambda}(1)) \simeq \mathbb{Q}_{f,\lambda}  \ar@{=}[d] 
 \\
\mathbf{D}_{\mathrm{dR}}(V_f(1)) / \mathbf{D}^{0}_{\mathrm{dR}}(V_f(1)) \ar@{<-->}[d]_-{\textrm{Eichler--Shimura}}\ar@/_2pc/[u]^-{\mathrm{exp}}_-{\simeq} & \times & \mathbf{D}^{0}_{\mathrm{dR}}(V_{\overline{f}}(1)) \ar[d]^-{\simeq} \ar[rr]^-{\langle -, - \rangle_{\mathrm{dR}}} & & \mathbf{D}_{\mathrm{dR}}(\mathbb{Q}_{f,\lambda}(1)) \simeq \mathbb{Q}_{f,\lambda}  \ar[d]_-{ \Omega^{\pm}_{f} \times }^-{\simeq}  \\
\mathrm{H}_1(X_1(N), \mathbb{Z}_p)^\pm \otimes_{\mathbb{T}} \mathbb{Q}_{f, \lambda}  & \times & S(\overline{f}) \otimes_{\mathbb{Q}_f} \mathbb{Q}_{f, \lambda}  \ar[rr]^-{\int_{\gamma^{\pm}} \omega} & & \Omega^{
\pm}_{f} \cdot \mathbb{Q}_{f,\lambda}  
}
}
\end{gathered}
\end{equation}
where $\mathrm{exp}$ is the Bloch--Kato exponential map defined in \cite[Chapter II, $\S$1.3.4]{kato-lecture-1}.
Then we can find a $\mathbb{Q}_{f,\lambda}$-basis $\omega^*_{\overline{f}}$ of $\mathbf{D}_{\mathrm{dR}}(V_f(1)) / \mathbf{D}^{0}_{\mathrm{dR}}(V_f(1))$ by equality
$$\langle \omega^*_{\overline{f}} , \omega_{\overline{f}} \rangle_{\mathrm{dR}} = \frac{1}{\Omega^{+}_{f}} \cdot \int_{\gamma^+_f} \omega_{\overline{f}} =1.$$
Note that the period of $f$ not of $\overline{f}$ occurs due to the complex conjugation on the second term
and $\omega^*_{\overline{f}}$ also becomes a $\mathbb{Q}_{f,\lambda} \otimes \mathbb{Q}_p(\mu_n)$-basis of $\mathbf{D}_{\mathrm{dR},n}(V_f(1)) / \mathbf{D}^{0}_{\mathrm{dR},n}(V_f(1))$.
For a more refined description of the de Rham pairing and the integral canonical periods, see \cite[$\S$6]{ochiai-two-variable}.

\subsection{The logarithm map and formal groups} \label{subsec:log_formal_groups}
Let $J_1(N)_{\overline{f},1}(\mathbb{Q}_p(\mu_n))$ be the kernel of the reduction of $J_1(N)_{\overline{f}}(\mathbb{Q}_p(\mu_n))$ modulo $\mathfrak{m}_{\mathbb{Q}_p(\mu_n)}$.
Since $\mathbb{Q}_p(\mu_n)/\mathbb{Q}_p$ is unramified, $\widehat{\mathfrak{J}_1(N)}_{\overline{f}}  (\mathfrak{m}_{\mathbb{Q}_p(\mu_n)})$ has no torsion.
Due to the non-existence of the torsion, we are able to make a precise connection between the logarithm map and the formal logarithm map. Also, in order to single out $f$ among its Galois conjugates, we take the $\mathbb{Z}_{f,\lambda}$-component.  Then we have the the following commutative diagram
\begin{equation} \label{eqn:logarithm-commutative}
\begin{gathered}
\xymatrix@R=1em{
J_1(N)_{\overline{f}, \lambda}  (\mathbb{Q}_p(\mu_n)) \otimes_{\mathbb{Z}_p} \mathbb{Q}_p \ar[r]^-{\mathrm{log}} & \mathbf{D}_{\mathrm{dR},\mathbb{Q}_p}(V_f(1)) / \mathbf{D}^0_{\mathrm{dR},\mathbb{Q}_p}(V_{f}(1)) \otimes_{\mathbb{Q}_p} \mathbb{Q}_p(\mu_n) \\ 
J_1(N)_{\overline{f},1, \lambda}   (\mathbb{Q}_p(\mu_n)) \ar[u] & \mathbf{D}_{\mathrm{dR},\mathbb{Q}_p}(T_f(1)) / \mathbf{D}^0_{\mathrm{dR},\mathbb{Q}_p}(T_{f}(1))  \otimes_{\mathbb{Z}_p} \mathbb{Q}_p(\mu_n) \ar@{=}[u] \\
\widehat{\mathfrak{J}_1(N)}_{\overline{f}, \lambda} (\mathfrak{m}_{\mathbb{Q}_p(\mu_n)}) \ar[u]^-{\simeq} \ar[r]^-{\widehat{\mathrm{log}}}_-{\simeq} & \mathbf{D}_{\mathrm{dR},\mathbb{Q}_p}(T_f(1)) / \mathbf{D}^0_{\mathrm{dR},\mathbb{Q}_p}(T_{f}(1)) \otimes_{\mathbb{Z}_p} \widehat{\mathbb{G}}_a(\mathfrak{m}_{\mathbb{Q}_p(\mu_n)})  \ar@{^{(}->}[u] 
}
\end{gathered}
\end{equation}
where $\mathfrak{m}_{\mathbb{Q}_p(\mu_n)} = p\mathbb{Z}_p[\mu_n]$.
\subsection{Computing the size of the image} \label{subsec:computing_image}
Using the local Tate pairing, we identify the integral structures
\[
\xymatrix{
\mathrm{H}^1_s(\mathbb{Q}_p(\mu_n), V_{\overline{f}}(1)) \ar[r]^-{\simeq} & \mathrm{Hom}_{\mathbb{Q}_{f, \lambda}}(J_1(N)_{\overline{f},\lambda}(\mathbb{Q}_p(\mu_n)) \otimes_{\mathbb{Z}_p} \mathbb{Q}_p, \mathbb{Q}_{f, \lambda}) \\
\mathrm{H}^1_s(\mathbb{Q}_p(\mu_n), T_{\overline{f}}(1)) \ar[r]^-{\simeq} \ar@{^{(}->}[u] & \mathrm{Hom}_{\mathbb{Z}_{f, \lambda}}(J_1(N)_{\overline{f},\lambda}(\mathbb{Q}_p(\mu_n)), \mathbb{Z}_{f, \lambda})  . \ar@{^{(}->}[u]
}
\]
The horizontal map has an explicit formula due to Theorem \ref{thm:kato_formula}. In other words, for $z \in \mathrm{H}^1_s(\mathbb{Q}_p(\mu_n), T_{\overline{f}}(1))$, we assign the map
\begin{equation} \label{eqn:explicit-formula-pairing}
x \mapsto \mathrm{Tr}_{ \mathbb{Q}_{f, \lambda} \otimes_{\mathbb{Q}_p} \mathbb{Q}_p(\mu_n) / \mathbb{Q}_{f, \lambda} \otimes_{\mathbb{Q}_p} \mathbb{Q}_p }\left( \langle \mathrm{log} (x), \mathrm{exp}^{*}(z) \rangle_{\mathrm{dR}}  \right) .
\end{equation}
Therefore, in order to compute the lattice
$$\mathscr{L} := \langle \omega^*_{\overline{f}} , \mathrm{exp}^{*} \left( \mathrm{H}^1_s(\mathbb{Q}_p(\mu_n), T_{\overline{f}}(1)) \right) \rangle_{\mathrm{dR}} ,$$
it suffices to compute the (conjugate) reciprocal lattice
$$\langle  \mathrm{log} \left( J_1(N)_{\overline{f},\lambda}(\mathbb{Q}_p(\mu_n)) \right) , \omega_{\overline{f}} \rangle_{\mathrm{dR}} .$$
Because of (\ref{eqn:logarithm-commutative}) and $\widehat{\mathbb{G}}_a(\mathfrak{m}_{\mathbb{Q}_p(\mu_n)}) = \mathfrak{m}_{\mathbb{Q}_p(\mu_n)} = p\mathbb{Z}_p[\mu_n]$,
the image of the formal group under the formal logarithm map is
$$\left\langle  \widehat{\mathrm{log}} \left( \widehat{\mathfrak{J}_1(N)}_{\overline{f},\lambda}( \mathfrak{m}_{\mathbb{Q}_p(\mu_n)}  ) \right) , \omega_{\overline{f}} \right\rangle_{\mathrm{dR}} = \mathbb{Z}_{f,\lambda} \otimes p\mathbb{Z}_p[\mu_n] \subseteq \mathbb{Q}_{f,\lambda} \otimes \mathbb{Q}_p(\mu_n).$$
Let
$\mathfrak{J}_1(N)_{\overline{f},\lambda}(\mathbb{F}_p(\mu_n))$
be the $\mathbb{Z}_{f,\lambda}$-component of the $\mathrm{Gal}(\overline{\mathbb{F}}_p/\mathbb{F}_p(\mu_n) )$-invariant of the reduction of $\mathfrak{J}_1(N)_{\overline{f}}$ at $p$ and we have an exact sequence
\begin{equation} \label{eqn:short_exact}
\xymatrix{
0 \ar[r] & \widehat{\mathfrak{J}_1(N)}_{\overline{f},\lambda}( \mathfrak{m}_{\mathbb{Q}_p(\mu_n)} )  \ar[r] &
J_1(N)_{\overline{f},\lambda}(\mathbb{Q}_p(\mu_n))  \ar[r] &
\mathfrak{J}_1(N)_{\overline{f},\lambda}(\mathbb{F}_p(\mu_n))  \ar[r] & 0 .
}
\end{equation}
Considering the logarithm maps from the above sequence (\ref{eqn:short_exact}), we have the following diagram:
\[
{ \scriptsize
\xymatrix@R=1.5em{
& 0 \ar[d] \\
& \widehat{\mathfrak{J}_1(N)}_{\overline{f},\lambda}(\mathfrak{m}_{\mathbb{Q}_p(\mu_n)}) \ar[rr]^-{\langle \widehat{\mathrm{log}}(-), \omega_{\overline{f}} \rangle_{\mathrm{dR}}}_-{\simeq} \ar[d] & & \mathbb{Z}_{f,\lambda} \otimes_{\mathbb{Q}_{p}} p\mathbb{Z}_p[\mu_n] \\
 J_1(N)_{\overline{f},\lambda}(\mathbb{Q}_p(\mu_n))_{\mathrm{tors}} \ar[d]_-{\simeq} \ar@{^{(}->}[r]& J_1(N)_{\overline{f},\lambda}(\mathbb{Q}_p(\mu_n)) \ar@{->>}[rr]^-{\langle \mathrm{log}(-), \omega_{\overline{f}} \rangle_{\mathrm{dR}}} \ar[d] & & \mathrm{Im} ( \mathrm{log} ) \subseteq  \mathbb{Q}_{f,\lambda} \otimes_{\mathbb{Q}_{p}} \mathbb{Q}_p(\mu_n) \\
\mathrm{ker} ( \overline{\mathrm{log}} ) \ar@{^{(}->}[r] & \mathfrak{J}_1(N)_{\overline{f},\lambda}(\mathbb{F}_p(\mu_n)) \ar@{->>}[rr]^-{\overline{\langle \mathrm{log}(-), \omega_{\overline{f}} \rangle_{\mathrm{dR}}}} \ar[d] & & \mathrm{Im} ( \mathrm{log} ) / \left( \mathbb{Z}_{f,\lambda} \otimes_{\mathbb{Q}_{p}} p\mathbb{Z}_p[\mu_n]  \right)\\
 & 0  
}
}
\]
By the Eichler--Shimura relation (\hspace{1sp}\cite[Corollary 5.15 and Theorem 5.16]{conrad-shimura}), we have
\begin{align*}
\mathfrak{J}_1(N)_{\overline{f},\lambda}(\mathbb{F}_p(\mu_n)) & = \mathrm{ker} \left( \mathrm{Frob}^{n_p}_p - \mathrm{Id} : \mathfrak{J}_1(N)_{\overline{f},\lambda}(\overline{\mathbb{F}}_p)  \to \mathfrak{J}_1(N)_{\overline{f},\lambda}(\overline{\mathbb{F}}_p)  \right) \\
& = \mathrm{ker} \left( (1-\overline{\alpha_p}^{n_p})(1-\overline{\beta_p}^{n_p})  : \mathfrak{J}_1(N)_{\overline{f},\lambda}(\overline{\mathbb{F}}_p) \to \mathfrak{J}_1(N)_{\overline{f},\lambda}(\overline{\mathbb{F}}_p) \right)  
\end{align*}
where $n_p = [\mathbb{F}_p(\mu_n): \mathbb{F}_p]$.
Thus, $(1-\overline{\alpha_p}^{n_p})(1-\overline{\beta_p}^{n_p})$ exactly annihilates $\mathfrak{J}_1(N)_{\overline{f},\lambda}(\mathbb{F}_p(\mu_n))$ and 
we define  $e_{n}$ by the $\lambda$-valuation of a generator of
$\mathrm{Ann}_{ \mathbb{Z}_{f,\lambda} \otimes \mathbb{Z}_p[\mu_n] } \left(  J_1(N)_{\overline{f},\lambda}(\mathbb{Q}_p(\mu_n))_{\mathrm{tors}} \right) $.
Then we have
\begin{align*}
\langle  \mathrm{log} \left( J_1(N)_{f,\lambda}(\mathbb{Q}_p(\mu_n)) \right) , \omega_{\overline{f}} \rangle_{\mathrm{dR}} & =  \frac{ \lambda^{e_{n}}   }{ (1-\overline{\alpha_p}^{n_p})(1-\overline{\beta_p}^{n_p}) } \mathbb{Z}_{f,\lambda} \otimes_{\mathbb{Q}_{p}} p\mathbb{Z}_p[\mu_n]   \\
& \subseteq \mathbb{Q}_{f,\lambda} \otimes \mathbb{Q}_p(\mu_n) .
\end{align*}
\begin{rem}
Both $(1-\overline{\alpha_p}^{n_p})(1-\overline{\beta_p}^{n_p})$ and $\lambda^{e_{n}}$ are non-zero due to the identity elements of $\mathfrak{J}_1(N)_{\overline{f},\lambda}(\mathbb{F}_p(\mu_n))$ and of $J_1(N)_{\overline{f},\lambda}(\mathbb{Q}_p(\mu_n))_{\mathrm{tors}}$, respectively.
\end{rem}
By the duality via the de Rham pairing, we have the following statement.
\begin{prop} \label{prop:the_image}
$$\mathscr{L} := \langle \omega^*_{\overline{f}} , \mathrm{exp}^{*} \left( \mathrm{H}^1_s(\mathbb{Q}_p(\mu_n), T_{\overline{f}}(1)) \right) \rangle_{\mathrm{dR}}  = \frac{1}{p} \cdot \frac{(1-\alpha^{n_p}_p)(1-\beta^{n_p}_p)}{ \lambda^{e_{n}} } \cdot \mathbb{Z}_{f,\lambda}  \otimes_{\mathbb{Q}_{p}}  \mathbb{Z}_p[\mu_n]$$
and it becomes a $\mathbb{Z}_{f,\lambda} \otimes_{\mathbb{Z}_{p}} \mathbb{Z}_p[\mu_n]$-lattice in $\mathbb{Q}_{f,\lambda} \otimes_{\mathbb{Q}_{p}} \mathbb{Q}_p(\mu_n)$.
\end{prop}
\begin{rem} \label{rem:size_of_image}
Since $J_1(N)_{\overline{f},\lambda}(\mathbb{Q}_p(\mu_n))_{\mathrm{tors}} = \mathrm{H}^0(\mathbb{Q}_p(\mu_n), A_f(1))$ and
$\mathrm{ker} ( \overline{\mathrm{log}} )$ are isomorphic, the value
$ \dfrac{(1-\alpha^{n_p}_p)(1-\beta^{n_p}_p)}{ \lambda^{e_{n}} }$
is $\lambda$-integral.
\end{rem}
\begin{rem}
In order to cover the full cyclotomic extension $\mathbb{Q}(\mu_{p^\infty})$, not just $\mathbb{Q}_{\infty}$, it seems that one needs to generalize the computation in this section to $\mathbb{Q}_p(\mu_{np})$. However, since $\mathbb{Q}_p(\mu_{np})/\mathbb{Q}_p$ is a ramified extension, the formal group argument (or the Fontaine--Laffaille theory as in \cite[$\S$4]{bloch-kato}) does not seem to work neatly.
\end{rem}
\section{Explicit description of (residual) Kolyvagin systems from Euler systems} \label{sec:explicit_construction}
We explicitly describe the map from Kato's Euler systems to Kato's Kolyvagin systems modulo $\lambda$ as the mod $\lambda$ version of Theorem \ref{thm:euler_to_kolyvagin}. See \cite[Appendix A]{mazur-rubin-book} for detail.

\subsection{Kolyvagin derivatives} \label{subsec:kolyvagin_derivatives}
Let $n$ be a product of Kolyvagin primes.
Let $c^+_{\mathbb{Q}(\mu_n)} \in \mathrm{H}^1(\mathbb{Q}(\mu_n), T_{\overline{f}}(1))$ be the $(+)$-part of Kato's Euler system at $\mathbb{Q}(\mu_n)$ as in Remark \ref{rem:conjugate_euler_systems}.

For each $\ell$, fix a primitive root $\eta_\ell$
and the corresponding generator $\sigma_{\eta_\ell} \in (\mathbb{Z}/\ell\mathbb{Z})^\times$.
Following \cite[Definition 4.4.1]{rubin-book},
we define the Kolyvagin derivative operator at $\ell$ by
$$D_\ell := \sum_{i=0}^{\ell-2} i \sigma^i_{\eta_\ell} ( =  \sum_{i=1}^{\ell-2} i \sigma^i_{\eta_\ell} ).$$
Then it satisfies relation
$(\sigma_{\eta_\ell} - 1)D_\ell =  \ell - 1 - \mathrm{Tr}_\ell$
where $\mathrm{Tr}_\ell := {\displaystyle \sum_{i=1}^{\ell-1} } \sigma^i_{\eta_\ell} ( =  {\displaystyle \sum_{i=0}^{\ell-2} } \sigma^i_{\eta_\ell} )$.
We define the \textbf{Kolyvagin derivative (at $n$)} by
$$D_n := \prod_{\ell \vert n} D_{\ell} .$$
\subsection{Derived Euler systems and Kolyvagin systems} \label{subsec:derived_euler_n_kolyvagin}

We define \textbf{weak Kolyvagin system $w\kappa_n$ modulo $\lambda$} by the following diagram
\[
{ \scriptsize
\xymatrix@C=0.4em@R=1.5em{
\mathrm{H}^1(\mathbb{Q}(\mu_n), T_{\overline{f}}(1)) \ar[d]^-{D_n} & c^+_{\mathbb{Q}(\mu_n)} \ar@{|->}[d] & \textrm{Euler systems}\\
\mathrm{H}^1(\mathbb{Q}(\mu_n), T_{\overline{f}}(1)) \ar[d]^-{\bmod{\lambda}} & D_n c^+_{\mathbb{Q}(\mu_n)} \ar@{|->}[d] & \textrm{derived Euler systems}\\
\left( \mathrm{H}^1(\mathbb{Q}(\mu_n), T_{\overline{f}}(1)) / \lambda \mathrm{H}^1(\mathbb{Q}(\mu_n), T_{\overline{f}}(1)) \right)^{\mathrm{Gal}( \mathbb{Q}(\mu_n)/\mathbb{Q} )} \ar@{_{(}->}[d] & d^+_n \ar@{|->}[dd] & \textrm{\cite[Lemma 4.4.2]{rubin-book}}\\ 
\left( \mathrm{H}^1( \mathbb{Q}(\mu_n),  T_{\overline{f}}(1)/\lambda T_{\overline{f}}(1) ) \right)^{\mathrm{Gal}( \mathbb{Q}(\mu_n)/\mathbb{Q} )} \ar@{-->}[d]^-{\mathrm{res}^{-1}}&  \\
 \mathrm{H}^1(\mathbb{Q},  T_{\overline{f}}(1) / \lambda T_{\overline{f}}(1) ) \ar@{~>}[d]_-{\textrm{Equation (\ref{eqn:kolyvagin_systems}) below}}^-{\textrm{and Proposition \ref{prop:weak_vs_derived}}} & w\kappa_n \Mod{\lambda} \ar@{|~>}[d] & \textrm{weak Kolyvagin systems modulo $\lambda$} \\
 \mathrm{H}^1(\mathbb{Q},  T_{\overline{f}}(1) / \lambda T_{\overline{f}}(1) ) \otimes G_n & \kappa_n \Mod{\lambda} & \textrm{Kolyvagin systems modulo $\lambda$}
}
}
\]
where $\mathrm{res}^{-1}$ is the inverse of the restriction map in the Hochschild--Serre spectral sequence
defined on the image of the Kolyvagin derivative classes. For the well-definedness of $\mathrm{res}^{-1}$, see \cite[$\S$4.4]{rubin-book}.

We recall the explicit formula for the construction of Kolyvagin systems from weak Kolyvagin systems. See \cite[Appendix A]{mazur-rubin-book} for detail.

Let $\ell$ be a Kolyvagin prime. Let $\mathcal{A}_{\ell}$ be the augmentation ideal of group ring $\left( \mathbb{Z}_{f,\lambda} / I_\ell \right) [G_\ell \otimes \left( \mathbb{Z}_{f,\lambda} / I_\ell \right) ]$.
Then there exists a canonical isomorphism of $\mathbb{Z}_{f,\lambda} / I_\ell$-modules defined by
\[
\xymatrix@R=0em{
\rho_\ell: \mathcal{A}_{\ell} / \mathcal{A}^2_{\ell} \ar[r]^-{\simeq} & G_\ell \otimes \left( \mathbb{Z}_{f,\lambda} / I_\ell \right) \\
\sigma - 1 \ar@{|->}[r] & \sigma \otimes 1
}
\]

Let $n \in \mathcal{N}$ and $\mathfrak{S}(n)$ be the set of permutations of the primes dividing $n$. For $\pi \in \mathfrak{S}(n)$, let ${ \displaystyle d_\pi := \prod_{\pi(\ell)=\ell} \ell }$. Then we define \textbf{Kolyvagin system $\kappa_n$}  by
\begin{equation} \label{eqn:kolyvagin_systems}
\kappa_n := \sum_{\pi \in \mathfrak{S}(n)} \left( \mathrm{sign}(\pi) \left( w\kappa_{d_\pi} \right) \otimes \bigotimes_{\ell \vert (n/d_\pi)} \rho_\ell (P_\ell (\mathrm{Fr}^{-1}_{\pi(\ell)} ) ) \right) \in \mathrm{H}^1(\mathbb{Q}, T_{\overline{f}}(1) / I_n T_{\overline{f}}(1) ) \otimes G_n
\end{equation}
following \cite[(33), Page 80]{mazur-rubin-book}, and $\lbrace \kappa_n : n \in \mathcal{N} \rbrace$ satisfies all the axioms of Kolyvagin systems.
From Equation (\ref{eqn:kolyvagin_systems}), the following proposition is straightforward and shows that the indivisibility of derived Euler systems is equivalent to the primitivity of the corresponding Kolyvagin systems.
\begin{prop} \label{prop:weak_vs_derived}
A derived Euler system $D_n c^+_{\mathbb{Q}(\mu_n)}$ at $\mathbb{Q}(\mu_n)$
 is non-zero modulo $\lambda$ if and only if
the corresponding Kolyvagin system $\kappa_{n}$ is non-zero modulo $\lambda$.
\end{prop}

\section{From Kato's Euler systems to modular symbols} \label{sec:zeta_modular}
\subsection{Kato's Euler systems and the interpolation formula}
We first fix the convention of Kato's Euler system.
Let $\delta^\pm_f \in \mathrm{H}^1(X_1(N), \mathbb{Z}) \otimes \mathbb{T}/\wp_f$ be the dual of $\gamma^{\pm}_f$ defined in $\S$\ref{subsec:eichler-shimura}, respectively.
\begin{defn}[Kato's Euler systems] \label{defn:kato_euler_systems}
We define 
\begin{align*}
c_{\mathbb{Q}(\mu_n)} & := b_1 \cdot {}_{c,d}z^{(p)}_n(f,1,1, \alpha_1, \mathrm{prime}(nNp))^- + b_2 \cdot {}_{c,d}z^{(p)}_n(f,1,1, \alpha_2, \mathrm{prime}(nNp))^+ \\
& \in \mathrm{H}^1_{\mathrm{\acute{e}t}}( \mathrm{Spec}(\mathbb{Z}[1/p, \zeta_n]), j_{*}T_{\overline{f}}(1))
\end{align*}
where 
\begin{itemize}
\item $c$ and $d$ are positive integers with $(cd, nNp) = 1$ and $p \nmid (c-1)(d-1)$,
\item $b_1, b_2 \in \mathbb{Q}_{f, \lambda}$ such that 
$\delta^+_f  = b_1 \cdot \delta_1(f, 1, \alpha_1)^+ (\neq 0)$ and $\delta^-_f = b_2 \cdot \delta_1(f, 1, \alpha_2)^- (\neq 0)$, and
\item  ${}_{c,d}z^{(p)}_n(f,1,1, \alpha, \mathrm{prime}(nNp))^{\pm}$ is the element defined in \cite[(8.1.3) and Example 13.3]{kato-euler-systems}.
\end{itemize}
\end{defn}
\begin{rem} $ $
\begin{enumerate}
\item 
The condition $p \nmid (c-1)(d-1)$ is noticed by Rubin in \cite[Corollary 7.2]{rubin-es-mec}.
The cohomology class $c_{\mathbb{Q}(\mu_n)}$ is independent of $\alpha_1$, $\alpha_2$, $b_1$, and $b_2$, but it depends on $c$ and $d$.
\item
Since $\delta^+_f + \delta^-_f \in \mathrm{H}^1(X_1(N), \mathbb{Z}) \otimes \mathbb{T}/\wp_f$, we have the map
\[
\xymatrix@R=0em{
\mathrm{H}^1(X_1(N), \mathbb{Z}) \otimes_{\mathbb{Z}} \mathbb{Z}_p \otimes \mathbb{T}/\wp_f \ar[r]^-{\simeq} &
\mathrm{H}^1_{\mathrm{\acute{e}t}}(X_1(N), \mathbb{Z}_p) \otimes \mathbb{T}/\wp_f = T_{\overline{f}} \ar[r] & \mathrm{H}^1_{\mathrm{\acute{e}t}}( \mathrm{Spec}(\mathbb{Z}[1/p, \zeta_n]), j_{*}T_{\overline{f}}(1)) \\
\delta^+_f + \delta^-_f \ar@{|->}[rr] & & c_{\mathbb{Q}(\mu_n)}
}
\]
where the first isomorphism is the comparison between Betti and \'{e}tale cohomologies as in \cite[$\S$8.3]{kato-euler-systems} and the second map is an analogue of the map $T_{\overline{f}} \to \mathbb{H}^1(T_{\overline{f}})$ defined by $\gamma \mapsto \mathbf{z}^{(p)}_\gamma$ in \cite[Theorem 12.5.(1) and (4)]{kato-euler-systems}. Since we do not invert any element in the group ring for our convention of $c_{\mathbb{Q}(\mu_n)}$, our case is much simpler than Kato's case, which inverts the elements in the completed group ring arising from the choice of $c$ and $d$ and the bad Euler factors. See \cite[$\S$13.9--13.14]{kato-euler-systems} for details and \cite[Appendix A]{delbourgo-book} for the generalization of \cite[$\S$13.9--13.12]{kato-euler-systems} from $\mathbb{Q}$ to $\mathbb{Q}(\mu_n)$.
See also \cite[Appendix A]{kim-nakamura} for a slightly different choice of Kato's Euler systems.
\item
Since $\delta^{\pm}_f \in \mathrm{H}^1(X_1(N), \mathbb{Z}) \otimes \mathbb{T}/\wp_f$
 is dual to $\gamma^{\pm}_f$ and $\gamma^{\pm}_f = \Omega^{\pm}_f \cdot \omega^*_{\overline{f}}$ as in (\ref{eqn:comparison-betti-de-rham}), we have
 the values of the paring
\[
\xymatrix{
 \langle \gamma^{\pm}_f, \delta^{\pm}_f \rangle = 1, & \langle \gamma^{\pm}_f, \omega_{\overline{f}} \rangle = \langle  \Omega^{\pm}_f \cdot \omega^*_{\overline{f}} ,  \omega_{\overline{f}} \rangle = \Omega^{\pm}_f  
} 
\]
where the pairing is induced from the identifications in (\ref{eqn:comparison-local-duality-dR-paring-period-integral}).
Since we have
$$ \delta^{\pm}_f = \dfrac{1}{\Omega^{\pm}_f} \cdot \omega_{\overline{f}} ,$$
the integral canonical periods naturally appear in the interpolation formula for Kato's Euler system
 (Theorem \ref{thm:kato_interpolation}) below.
\end{enumerate}
\end{rem}

\begin{thm}[{\hspace{1sp}\cite[Theorem 6.6 and Theorem 9.7]{kato-euler-systems}}] \label{thm:kato_interpolation}
Let $\chi$ be a Dirichlet character mod $n$.
Then Kato's Euler system $c_{F} \in \mathrm{H}^1(F, T_{\overline{f}}(1))$
satisfies the following interpolation formula
\begin{equation} \label{eqn:kato_interpolation}
\sum_{b \in (\mathbb{Z}/n\mathbb{Z})^\times}  \chi(b) \cdot \left\langle \omega^*_{\overline{f}},  \mathrm{exp}^{*} \left( \mathrm{loc}_p \left( c_{\mathbb{Q}(\mu_n)} \right)^{\sigma_b} \right) \right\rangle_{\mathrm{dR}} = c \cdot d \cdot (c - \chi(c)) \cdot (d - \chi(d)) \cdot \frac{L^{(Np)}(f, \chi, 1)}{(-2 \pi i)\Omega^{\chi(-1)}_{f}}
\end{equation}
where $c$ and $d$ are positive integers with $(cd, nNp) = 1$ and $p \nmid (c-1)(d-1)$ chosen in Definition \ref{defn:kato_euler_systems}
and $L^{(Np)}(f, \chi, 1)$ is the $Np$-imprimitive $L$-value of $f$ at $s = 1$ twisted by $\chi$.
\end{thm}
Theorem \ref{thm:kato_interpolation} can be refined via the $\pm$-decomposition of the Euler systems as in Remark \ref{rem:conjugate_euler_systems}:
\begin{align} \label{eqn:algebraic_formula} 
\begin{split} 
& \sum_{b \in (\mathbb{Z}/n\mathbb{Z})^\times}  \chi(b) \cdot \left\langle \omega^*_{\overline{f}},  \mathrm{exp}^{*} \left(  \left( c^{\chi(-1)}_{\mathbb{Q}(\mu_n)} \right)^{\sigma_b} \right) \right\rangle_{\mathrm{dR}} \\
= & c \cdot d \cdot (c - \chi(c)) \cdot (d - \chi(d)) \cdot \left( 1 - \frac{a_p(f) \cdot \chi(p)}{p} + \psi(p)\frac{\chi(p)^2}{p} \right) \\
& \cdot \left( \prod_{q \mid N_{\mathrm{sp}}} ( 1-  q^{-1} \chi(q) )  \right) \cdot \left( \prod_{q \mid N_{\mathrm{ns}}} ( 1 + q^{-1} \chi(q) )  \right) \cdot \frac{L(f, \chi, 1)}{(-2 \pi i)\Omega^{\chi(-1)}_{f}} .
\end{split}
\end{align}
We rewrite the last term in Equation (\ref{eqn:algebraic_formula}) in terms of modular symbols. Expanding the Gauss sum in the interpolation formula of Mazur--Tate elements, we have
\[
 \chi(-1) \cdot \frac{L(f,\chi, 1)}{(-2\pi i) \Omega^{\chi(-1)}_{f}} 
 =  \frac{1}{n}  \cdot   \sum_{b \in (\mathbb{Z}/n\mathbb{Z})^\times} \chi(b) \cdot \sigma_b \cdot  \left( \sum_{a \in (\mathbb{Z}/n\mathbb{Z})^\times} \zeta^{a}_n \cdot  \left[ \frac{-a}{n} \right]^{\chi(-1)}_{f} \right) .
\]
We define the values
$$c^{\mathrm{an}, \pm}_{\mathbb{Q}(\mu_n)} :=  \frac{\pm 1}{n} \cdot  \left( \sum_{a \in (\mathbb{Z}/n\mathbb{Z})^\times} \zeta^{a}_n \cdot \left[ \frac{-a}{n} \right]^{\pm}_{f} \right) \in \mathbb{Z}_{f,\lambda} \otimes \mathbb{Z}_p[\mu_n]$$
in order to have
$$ \sum_{b \in (\mathbb{Z}/n\mathbb{Z})^\times} \left(   \sigma_b \left(  c^{\mathrm{an}, \chi(-1)}_{\mathbb{Q}(\mu_n)}  \right) \right) \cdot \chi(b) = \frac{L(f, \chi, 1)}{(-2 \pi i)\Omega^{\chi(-1)}_{f}} .$$
From now on, we ``extract" the Euler factor at $p$ from $c^{\mathrm{an}, \pm}_{\mathbb{Q}(\mu_n)}$.
Since we have
$$L(f,\chi, 1) = { \displaystyle \sum_{b \in (\mathbb{Z}/n\mathbb{Z})^\times} } \left(  (-2\pi i) \cdot \Omega^{\chi(-1)}_{f} \cdot \sigma_b \left(  c^{\mathrm{an}, \chi(-1)}_{\mathbb{Q}(\mu_n)}  \right) \right) \cdot \chi(b) ,$$
the value
$$L(f,b \Mod{n}, 1) := \frac{1}{2} \left( (-2\pi i) \cdot \Omega^{+}_{f} \cdot \sigma_b \left(  c^{\mathrm{an}, +}_{\mathbb{Q}(\mu_n)}  \right) 
+ (-2\pi i) \cdot \Omega^{-}_{f} \cdot \sigma_b \left(  c^{\mathrm{an}, -}_{\mathbb{Q}(\mu_n)}  \right) \right)$$
is the value of the analytic continuation of a suitable partial $L$-series as follows.
\begin{lem}
The value
$$L^{(p)}(f,b \Mod{n}, 1) := \left( 1 - a_p(f) \cdot \sigma^{-1}_p \cdot p^{-1} + \psi(p) \cdot \sigma^{-2}_p \cdot  p^{-1} \right) \cdot L(f,b \Mod{n}, 1)$$
 is the  value of the analytic continuation of the prime-to-$Np$ partial $L$-series at $s= 1$
$${ \displaystyle \sum_{\substack{m \equiv b \Mod{n} \\ (Np,m) =1 }} \frac{a_m(f)}{m^s} } .$$
\end{lem}
\begin{proof}
One can write 
$$L^{(p)}(f,b \Mod{n}, 1) = L(f,b \Mod{n}, 1) - L(f, ep \Mod{np}, 1)$$
where $e \in \mathbb{Z}/n\mathbb{Z}$ satisfies $ep \equiv b \Mod{n}$.
Since $f$ is a Hecke eigenform (at $p$), the straightforward computation yields the conclusion.
\end{proof}
This lemma shows that
\begin{align} \label{eqn:analytic_formula}
\begin{split}
\sum_{b \in (\mathbb{Z}/n\mathbb{Z})^\times} \left(   \sigma_b \left( \left( 1 - a_p(f) \cdot \sigma^{-1}_p \cdot p^{-1} + \psi(p)p \cdot \sigma^{-2}_p \cdot  p^{-2} \right) \cdot c^{\mathrm{an}, \chi(-1)}_{\mathbb{Q}(\mu_n)}  \right) \right) \cdot \chi(b) \\  
 = \left( 1 - \frac{a_p(f) \cdot \chi(p)}{p} + \psi(p)\frac{\chi(p)^2}{p} \right) \cdot \frac{L(f, \chi, 1)}{(-2 \pi i)\Omega^{\chi(-1)}_{f} }.
\end{split}
\end{align}

\subsection{Lifting to group rings} \label{subsec:lifting_coefficients}
Combining Equation (\ref{eqn:algebraic_formula}) and Equation (\ref{eqn:analytic_formula}), we have
\begin{align} \label{eqn:combined_formula}
\begin{split}
& \sum_{b \in (\mathbb{Z}/n\mathbb{Z})^\times}  \chi(b) \cdot \left\langle \omega^*_{\overline{f}},  \mathrm{exp}^{*} \left(  \left( c^{\chi(-1)}_{\mathbb{Q}(\mu_n)} \right)^{\sigma_b} \right) \right\rangle_{\mathrm{dR}} \\
= & c \cdot d \cdot (c - \chi(c)) \cdot (d - \chi(d)) \cdot \left( \prod_{q \mid N_{\mathrm{sp}}} ( 1-  q^{-1} \chi(q) )  \right) \cdot \left( \prod_{q \mid N_{\mathrm{ns}}} ( 1 + q^{-1} \chi(q) )  \right) \cdot  \\
&  \left( 1 - a_p(f) \cdot \sigma^{-1}_p \cdot p^{-1} + \psi(p) \cdot p^{-1} \cdot \sigma^{-2}_p  \right) \cdot \left( \sum_{b \in (\mathbb{Z}/n\mathbb{Z})^\times} \left(   \sigma_b  \cdot c^{\mathrm{an}, \chi(-1)}_{\mathbb{Q}(\mu_n)}  \right)  \cdot \chi(b) \right)
\end{split}
\end{align}
in $\mathbb{Z}_{f, \lambda}[\chi]$ for all characters $\chi$ on $\mathrm{Gal}(\mathbb{Q}(\mu_n)/\mathbb{Q})$
where $\left( 1 - a_p(f) \cdot \sigma^{-1}_p \cdot p^{-1} + \psi(p) \cdot p^{-1} \cdot \sigma^{-2}_p  \right)$
acts on $c^{\mathrm{an}, \chi(-1)}_{\mathbb{Q}(\mu_n)}$.
In order to lift Equality (\ref{eqn:combined_formula}) to group ring
$\mathbb{Z}_{f, \lambda}[\mathrm{Gal}(\mathbb{Q}(\mu_n)/\mathbb{Q})]$,
it suffices to check that
\begin{align*}
&\sum_{b \in (\mathbb{Z}/n\mathbb{Z})^\times}  \chi(b) \cdot \left\langle \omega^*_{\overline{f}},  \mathrm{exp}^{*} \left(  \left( c^{\chi(-1)}_{\mathbb{Q}(\mu_n)} \right)^{\sigma_b} \right) \right\rangle_{\mathrm{dR}} , 
\sum_{b \in (\mathbb{Z}/n\mathbb{Z})^\times} \left(   \sigma_b  \cdot c^{\mathrm{an}, \chi(-1)}_{\mathbb{Q}(\mu_n)}  \right)  \cdot \chi(b) \\
& (c - \chi(c)) \cdot (d - \chi(d)), \left( \prod_{q \mid N_{\mathrm{sp}}} ( 1-  q^{-1} \chi(q) )  \right) \cdot \left( \prod_{q \mid N_{\mathrm{ns}}} ( 1 + q^{-1} \chi(q) )  \right), \textrm{ and} \\
&  \left( 1 - \frac{a_p(f) \cdot \chi(p)}{p} + \psi(p)\frac{\chi(p)^2}{p} \right)\\
\end{align*}
for all $\chi$ lift to 
\begin{align} \label{eqn:lifted_ones}
\begin{split}
&\sum_{b \in (\mathbb{Z}/n\mathbb{Z})^\times}  \sigma^{-1}_b \cdot \left\langle \omega^*_{\overline{f}},  \mathrm{exp}^{*} \left(  \left( c^{\chi(-1)}_{\mathbb{Q}(\mu_n)} \right)^{\sigma_b} \right) \right\rangle_{\mathrm{dR}} , 
\sum_{b \in (\mathbb{Z}/n\mathbb{Z})^\times} \left(   \sigma_b  \cdot c^{\mathrm{an}, \chi(-1)}_{\mathbb{Q}(\mu_n)}  \right)  \cdot \sigma^{-1}_b \\
& (c - \sigma^{-1}_c) \cdot (d - \sigma^{-1}_d), \left( \prod_{q \mid N_{\mathrm{sp}}} ( 1-  q^{-1} \sigma^{-1}_q )  \right) \cdot \left( \prod_{q \mid N_{\mathrm{ns}}} ( 1 + q^{-1} \sigma^{-1}_q )  \right), \textrm{ and} \\
& \left( 1 - a_p(f) \cdot \sigma^{-1}_p \cdot p^{-1} + \psi(p)p \cdot \sigma^{-2}_p \cdot  p^{-2} \right),
\end{split}
\end{align}
respectively.
We follow the idea of \cite[Corollary 5.13]{ota-thesis}.
Since 
$$\mathbb{Q}_{f,\lambda}\otimes\mathbb{Q}(\mu_n) [\mathrm{Gal}( \mathbb{Q}(\mu_n)/\mathbb{Q} )] \simeq \prod_{\xi}\mathbb{Q}_{f,\lambda}\otimes\mathbb{Q}(\mu_n)[\mathrm{Im}\xi]$$
where $\xi$ runs over all characters on $\mathrm{Gal}( \mathbb{Q}(\mu_n)/\mathbb{Q} )$,
the equalities for all $\xi$ imply the equality in $\mathbb{Q}_{f,\lambda}\otimes\mathbb{Q}(\mu_n) [\mathrm{Gal}( \mathbb{Q}(\mu_n)/\mathbb{Q} )]$. Since all the above elements in (\ref{eqn:lifted_ones}) lie in 
$\mathbb{Z}_{f,\lambda} [\mathrm{Gal}( \mathbb{Q}(\mu_n)/\mathbb{Q} )]$, the lifting to the group ring works well.
To sum up, we have equality
\begin{align} \label{eqn:lifted_combined_formula}
\begin{split}
& \sum_{b \in (\mathbb{Z}/n\mathbb{Z})^\times}  \sigma^{-1}_b \cdot \left\langle \omega^*_{\overline{f}},  \mathrm{exp}^{*} \left(  \left( c^{\chi(-1)}_{\mathbb{Q}(\mu_n)} \right)^{\sigma_b} \right) \right\rangle_{\mathrm{dR}} \\
= & c \cdot d \cdot (c - \sigma^{-1}_c) \cdot (d - \sigma^{-1}_d) \cdot \left( \prod_{q \mid N_{\mathrm{sp}}} ( 1-  q^{-1} \sigma^{-1}_q )  \right) \cdot \left( \prod_{q \mid N_{\mathrm{ns}}} ( 1 + q^{-1} \sigma^{-1}_q )  \right) \cdot  \\
&  \left( 1 - a_p(f) \cdot \sigma^{-1}_p \cdot p^{-1} + \psi(p)p \cdot \sigma^{-2}_p \cdot  p^{-2} \right) \cdot \left( \sum_{b \in (\mathbb{Z}/n\mathbb{Z})^\times} \left(   \sigma_b  \cdot c^{\mathrm{an}, \chi(-1)}_{\mathbb{Q}(\mu_n)}  \right)  \cdot \sigma^{-1}_b \right)
\end{split}
\end{align}
in $\mathbb{Z}_{f,\lambda} [\mathrm{Gal}( \mathbb{Q}(\mu_n)/\mathbb{Q} )]$.

\subsection{Kolyvagin derivatives on modular symbols and Kurihara numbers} \label{subsec:williams}
\begin{thm}[Kurihara, Williams; {\cite[two lines above (21) (page 190)]{kurihara-munster}}] \label{thm:computation_KS}
Let $n = \ell_1 \cdot \cdots \cdot \ell_s$ be a square-free product of Kolyvagin primes.
We have the following equalities in $\mathbb{F}_\lambda$
\begin{align*}
D_n \left( \sum_{a \in (\mathbb{Z}/n\mathbb{Z})^\times } \zeta^{ a'}_n \left[ \frac{a}{n}\right]^{\pm}_f \right) & \equiv
\sum_{a \in (\mathbb{Z}/n\mathbb{Z})^\times}  \left( \prod_{\ell \vert n}  \mathrm{log}_{\mathbb{F}_\ell} ( a ) \right) \cdot \left[\frac{a}{n}  \right]^{\pm}_f  \pmod{\lambda} 
\end{align*}
where $a' = \pm a$.
\end{thm}
\begin{proof}
Here is a sketch of the main idea. First, expand ${ \displaystyle \sum_{a \in (\mathbb{Z}/n\mathbb{Z})^\times }  \left[ \frac{a}{n}\right]^{\pm}_f \sigma_a }$ at ${\displaystyle \prod^{s}_{i=1} \left( \sigma_{\eta_{\ell_i}} - 1\right) }$.
Taking Kolyvagin derivative $D_n$ on the expansion, all but the term $\prod^{s}_{i=1} \left( \sigma_{\eta_{\ell_i}} - 1\right)$ vanish.
The higher degree term (for each $\left( \sigma_{\eta_{\ell_i}} - 1\right)$) vanishes after taking $D_n$ due to the relation $(\sigma_{\eta_{\ell_i}} - 1) D_{\ell_i} = \ell_i - 1 - \mathrm{Tr}_{\ell_i}$. 
Also, the lower degree term vanishes using Hecke operators at $\ell_i$.
In other words, we have
\begin{align*}
D_n \left( \sum_{a \in (\mathbb{Z}/n\mathbb{Z})^\times }  \left[ \frac{a}{n}\right]^{\pm}_f  \sigma_a \right) & \equiv
 \sum_{a \in (\mathbb{Z}/n\mathbb{Z})^\times}   \left[\frac{a}{n}  \right]^{\pm}_f \cdot  \prod_{\ell \vert n} \left( D_{\ell} (\eta_{\ell} - 1)  \right) \Mod{\lambda}  \\
& \equiv \sum_{a \in (\mathbb{Z}/n\mathbb{Z})^\times}   \left[\frac{a}{n}  \right]^{\pm}_f \cdot  \prod_{\ell \vert n} \left( \mathrm{log}_{\mathbb{F}_\ell} ( a ) \cdot (-\mathrm{Tr}_\ell) \right) \Mod{\lambda} .
\end{align*}
Considering the action of both sides on $\zeta^{\pm 1}_n$, we have
$$D_n \left( \sum_{a \in (\mathbb{Z}/n\mathbb{Z})^\times }  \left[ \frac{a}{n}\right]^{\pm}_f  \sigma_a \right) \cdot \zeta^{\pm}_n \equiv
 \sum_{a \in (\mathbb{Z}/n\mathbb{Z})^\times}   \left[\frac{a}{n}  \right]^{\pm 1}_f \cdot  \prod_{\ell \vert n} \left( \mathrm{log}_{\mathbb{F}_\ell} ( a ) \cdot (-\mathrm{Tr}_\ell) \right) \cdot \zeta^{\pm 1}_n \Mod{\lambda} .$$
Thus, the conclusion immediately follows. 
\end{proof}
\begin{rem}
This theorem is also observed in \cite[Theorem 9.5]{analytic_kolyvagin} via a purely analytic computation.
Kurihara found the importance of $\widetilde{\delta}_n$ and derived it from Mazur--Tate elements via the mod $p$ Taylor expansion of $\sigma_a$ at $\prod_{\ell \mid n} \left(\sigma_{\eta_\ell} -1\right)$.
See \cite[(21) (page 190) and (65)]{kurihara-munster} and \cite[(2) and (31), (32) (page 346)]{kurihara-iwasawa-2012} for detail. For the expansion of higher degree terms, see \cite{ota-thesis}.
\end{rem}
\subsection{Proof of Theorem {\ref{thm:main_theorem}}} \label{subsec:the_proof}
We give a proof of Theorem \ref{thm:main_theorem}. Here we only work with $c^+_{\mathbb{Q}(\mu_n)}$ and the result with  $c^+_{\mathbb{Q}(\mu_n)}$ is enough to imply the main conjecture for $(A_f(1), \mathbb{Q}_\infty/\mathbb{Q})$ since $\mathbb{Q}_\infty$ is totally real. The following diagram exactly shows what we compute.
\[
{ \scriptsize
\xymatrix@C=0.4em@R=1.5em{
\mathrm{H}^1(\mathbb{Q}(\mu_n), T_{\overline{f}}(1)) \ar[r]^-{D_n} & \mathrm{H}^1(\mathbb{Q}(\mu_n), T_{\overline{f}}(1)) \ar[d]^-{\bmod{\lambda}} \ar[rrrr]^-{\left\langle \omega^*_{\overline{f}},  \mathrm{exp}^{*} \left( \mathrm{loc}_p - \right) \right\rangle_{\mathrm{dR}}} & & & & \mathscr{L} \ar[d]^-{\bmod{\lambda}} \\
& \left( \mathrm{H}^1(\mathbb{Q}(\mu_n), T_{\overline{f}}(1)) / \lambda \mathrm{H}^1(\mathbb{Q}(\mu_n), T_{\overline{f}}(1)) \right)^{\mathrm{Gal}( \mathbb{Q}(\mu_n)/\mathbb{Q} )} \ar[rrrr]^-{\overline{\left\langle \omega^*_{\overline{f}},  \mathrm{exp}^{*} \left( \mathrm{loc}_p - \right) \right\rangle_{\mathrm{dR}}}} & & & & \mathscr{L}/\lambda\mathscr{L} \simeq \mathbb{F}_\lambda \\
 c^+_{\mathbb{Q}(\mu_n)} \ar@{|->}[r] &
  D_nc^+_{\mathbb{Q}(\mu_n)} \ar@{|->}[d] \ar@{|->}[rrrr] & & & & \langle \omega^*_{\overline{f}},D_n \mathrm{exp}^{*}(\mathrm{loc}_p c^{+}_{\mathbb{Q}(\mu_n)}) \rangle_{\mathrm{dR}} \ar@{|->}[d]\\
&  d^+_n \ar@{|->}[rrrr] & & & & \langle \omega^*_{\overline{f}}, D_n \mathrm{exp}^{*}(\mathrm{loc}_p c^{+}_{\mathbb{Q}(\mu_n)}) \rangle_{\mathrm{dR}} \Mod{\lambda} 
}
}
\]
where $\overline{\left\langle \omega^*_{\overline{f}},  \mathrm{exp}^{*} \left( \mathrm{loc}_p - \right) \right\rangle_{\mathrm{dR}}}$ is the induced reduction of $\left\langle \omega^*_{\overline{f}},  \mathrm{exp}^{*} \left( \mathrm{loc}_p - \right) \right\rangle_{\mathrm{dR}}$ modulo $\lambda$. At the end, we use Theorem \ref{thm:computation_KS} to show that
$$\langle \omega^*_{\overline{f}},D_n \mathrm{exp}^{*}( \mathrm{loc}_p c^{+}_{\mathbb{Q}(\mu_n)}) \rangle_{\mathrm{dR}} \Mod{\lambda} = u \cdot \widetilde{\delta}_n$$
where $u \in \mathbb{F}^\times_\lambda$.

\begin{proof}[Proof of Theorem \ref{thm:main_theorem}]
Due to the reduction of proof in $\S$\ref{subsec:the_main_idea}, it suffices to prove $\kappa_n \Mod{\lambda} \neq 0$.
Suppose that $\kappa_n \Mod{\lambda} = 0$ in $\mathrm{H}^1_{\mathcal{F}(n)}(\mathbb{Q}, T_{\overline{f}}(1)/ \lambda T_{\overline{f}}(1)) \otimes G_n$.
Then $$d^+_n = 0 \in \mathrm{H}^1(\mathbb{Q}(\mu_n), T_{\overline{f}}(1))/\lambda \mathrm{H}^1_{\mathcal{F}(n)}(\mathbb{Q}(\mu_n), T_{\overline{f}}(1)).$$
Thus, $D_n c^+_{\mathbb{Q}(\mu_n)} \in \lambda\mathrm{H}^1(\mathbb{Q}(\mu_n), T_{\overline{f}}(1))$.
Taking the dual exponential map and the pairing with $\omega^{*}_{\overline{f}}$, we have
\begin{align*}
& \left\langle \omega^*_{\overline{f}},  D_n \mathrm{exp}^{*} \left( \mathrm{loc}_p  c^{+}_{\mathbb{Q}(\mu_n)} \right) \right\rangle_{\mathrm{dR}} \in \lambda\mathscr{L} \\
\Rightarrow & C_p  \cdot \left( D_n c^{\mathrm{an}, +}_{\mathbb{Q}(\mu_n)} \right) \in \lambda\mathscr{L} & \textrm{(\ref{eqn:lifted_ones})} \\
\Rightarrow  & p \cdot C_p \cdot \left( D_n c^{\mathrm{an}, +}_{\mathbb{Q}(\mu_n)} \right) \in \lambda \cdot \frac{(1-\alpha^{n_p}_p)(1-\beta^{n_p}_p)}{ \lambda^{e_{n}} } \cdot \mathbb{Z}_{f,\lambda} \otimes  \mathbb{Z}_p[\mu_n] & \textrm{Proposition \ref{prop:the_image}}
\end{align*}
where
\begin{align*}
p \cdot C_p =  & c \cdot d \cdot (c - \sigma^{-1}_c) \cdot (d - \sigma^{-1}_d) \cdot \left( \prod_{q \mid N_{\mathrm{sp}}} ( 1-  q^{-1} \sigma^{-1}_q )  \right) \cdot \left( \prod_{q \mid N_{\mathrm{ns}}} ( 1 + q^{-1} \sigma^{-1}_q )  \right) \\ 
& \cdot \left( p - a_p(f) \cdot \sigma^{-1}_p  + \psi(p) \cdot \sigma^{-2}_p  \right) .
\end{align*}
Since $\frac{(1-\alpha^{n_p}_p)(1-\beta^{n_p}_p)}{ \lambda^{e_{n}} } $ is $\lambda$-integral (Remark \ref{rem:size_of_image}), we have
$$ p \cdot C_p  \cdot \left( D_n c^{\mathrm{an}, +}_{\mathbb{Q}(\mu_n)} \right) \Mod{\lambda} = 0  \in  \mathbb{F}_{\lambda} \otimes \mathbb{F}_p(\mu_n) .$$
Due to Theorem \ref{thm:computation_KS}, it is equivalent to
\begin{equation} \label{eqn:mod_p_delta_eqn}
 p \cdot C_p  \cdot \widetilde{\delta}_n  \Mod{\lambda} = 0  \in  \mathbb{F}_{\lambda} \otimes \mathbb{F}_p(\mu_n) ,
\end{equation}
and indeed, $\widetilde{\delta}_n \in \mathbb{F}_\lambda$; thus, the Galois action on $\widetilde{\delta}_n$ becomes trivial.  (This triviality is the analytic incarnation of \cite[Lemma 4.4.2]{rubin-book}.) 
Thus, (\ref{eqn:mod_p_delta_eqn}) is equivalent to
$$ p \cdot \overline{C}_p \cdot  \widetilde{\delta}_n = 0 \in  \mathbb{F}_{\lambda} $$
where
\begin{align*}
p \cdot \overline{C}_p =  c \cdot d \cdot (c - 1) \cdot (d - 1) \cdot \left( \prod_{q \mid N_{\mathrm{sp}}} ( 1-  q^{-1} )  \right) \cdot \left( \prod_{q \mid N_{\mathrm{ns}}} ( 1 + q^{-1} )  \right) \cdot  \left( p - a_p(f)   + \psi(p) \right) .
\end{align*}
Then we have
\begin{align*}
& c \cdot d \cdot (c - 1) \cdot (d - 1)  \in \mathbb{F}^{\times}_{\lambda} & \textrm{Definition \ref{defn:kato_euler_systems}}\\
& \left( \prod_{q \mid N_{\mathrm{sp}}} ( 1-  q^{-1} )  \right) \cdot \left( \prod_{q \mid N_{\mathrm{ns}}} ( 1 +  q^{-1} ) \right)  \in \mathbb{F}^{\times}_{\lambda} & \textrm{Condition ($N$-imp)}\\
&  \left( p - a_p(f)   + \psi(p) \right)  \in \mathbb{F}^{\times}_{\lambda} . & \textrm{Condition (NA)}
\end{align*}
Thus, it implies
$$ \widetilde{\delta}_n = 0 \in  \mathbb{F}_{\lambda} .$$
\end{proof}

\section{Examples} \label{sec:examples}
In this section, we describe (new) explicit examples of the Iwasawa main conjecture of modular forms over the cyclotomic $\mathbb{Z}_p$-extension.
\subsection{Elliptic curves of conductor $<$ 30,000 with non-trivial Shafarevich--Tate groups} \label{subsec:examples_table}
The following corollary \emph{completes} the validity of the main conjecture for elliptic curves of conductor $< 30,000$ with non-trivial $p$-part of the analytic order of the Shafarevich--Tate groups and $p \geq 5$. Note that all such elliptic curves have rank zero and see \cite[$\S$3.8]{grigorov-thesis} for the table. 
\begin{cor} \label{cor:computation}
Under the assumptions of Theorem \ref{thm:main_theorem}, $$\widetilde{\delta}_n \neq 0$$ for some $n$ for all (optimal) elliptic curves over $\mathbb{Q}$ of conductor $< 30,000$ with $p \geq 5$ such that the $p$-part of the analytic order of the Shafarevich--Tate groups is non-trivial.
\end{cor}
\begin{proof}
It suffices to compute the ``computation failed" cases in Grigorov's table in \cite[$\S$3.8]{grigorov-thesis}. The index of elliptic curves follows that of \cite{lmfdb}, not of Cremona's table.
\begin{center}
\begin{tabular}{ccc}
Elliptic curve & $p$ & Theorem \ref{thm:main_theorem} \\
\cite[\href{http://www.lmfdb.org/EllipticCurve/Q/6432/n/1}{6432.n1}]{lmfdb} & 5 & Condition (Tam) breaks \\
\cite[\href{http://www.lmfdb.org/EllipticCurve/Q/13790/c/1}{13790.c1}]{lmfdb} & 11 & $\widetilde{\delta}_{2663 \cdot 2707} \neq 0$ \\ 
\cite[\href{http://www.lmfdb.org/EllipticCurve/Q/15953/b/2}{15953.b2}]{lmfdb} & 5 & $\widetilde{\delta}_{191 \cdot 1021} \neq 0$  \\
\cite[\href{http://www.lmfdb.org/EllipticCurve/Q/16698/i/1}{16698.i1}]{lmfdb}  & 5 & $\widetilde{\delta}_{31 \cdot 131} \neq 0$  \\
\cite[\href{http://www.lmfdb.org/EllipticCurve/Q/17262/f/4}{17262.f4}]{lmfdb} & 5 &  $\widetilde{\delta}_{71 \cdot 181} \neq 0$ \\
\cite[\href{http://www.lmfdb.org/EllipticCurve/Q/18832/c/1}{18832.c1}]{lmfdb} & 7 &  $\widetilde{\delta}_{113 \cdot 379} \neq 0$ \\
\cite[\href{http://www.lmfdb.org/EllipticCurve/Q/22678/j/2}{22678.j2}]{lmfdb} & 5 & Condition ($N$-imp) breaks  \\
\cite[\href{http://www.lmfdb.org/EllipticCurve/Q/23826/k/1}{23826.k1}]{lmfdb} & 5 &  $\widetilde{\delta}_{181 \cdot 401} \neq 0$ \\
\cite[\href{http://www.lmfdb.org/EllipticCurve/Q/24642/a/1}{24642.a1}]{lmfdb} & 5 &  $\widetilde{\delta}_{31 \cdot 61} \neq 0$ \\
\cite[\href{http://www.lmfdb.org/EllipticCurve/Q/28644/h/2}{28644.h2}]{lmfdb} & 5 &  $\widetilde{\delta}_{131 \cdot 161} \neq 0$ 
\end{tabular}
\end{center}
\end{proof}

\subsection{Elliptic curves of conductor $<$ 1,000 of rank $\geq 1$ with $p=5$} \label{subsec:examples_rank_one}
We also confirm the Iwasawa main conjecture for elliptic curves of conductor $< 1,000$ having no square-free part of rank $\geq 1$ with $p=5$.
Note that the Iwasawa main conjecture for elliptic curves with good ordinary reduction of conductor having square-free part holds due to \cite{skinner-urban}.
\begin{cor} \label{cor:computation-2}
Under the assumptions of Theorem \ref{thm:main_theorem},
$$\widetilde{\delta}_n \neq 0$$
for some $n$ for all (optimal) elliptic curves over $\mathbb{Q}$ of conductor $< 1,000$ having no square-free part with $p = 5$ of rank $\geq 1$. Especially, the Iwasawa main conjecture for elliptic curves with good ordinary reduction at $p=5$ of conductor $< 1,000$ is confirmed.
\end{cor}
\begin{proof}
It suffices to confirm the following examples. Indeed, all the elliptic curves below have rank one.
\begin{center}
\begin{tabular}{ccccc}
Elliptic curve & Theorem \ref{thm:main_theorem} & & Elliptic curve & Theorem \ref{thm:main_theorem} \\
\cite[\href{http://www.lmfdb.org/EllipticCurve/Q/196/a/2}{196.a2}]{lmfdb} &
$\widetilde{\delta}_{11} \neq 0$ 
& & \cite[\href{http://www.lmfdb.org/EllipticCurve/Q/648/c/1}{648.c1}]{lmfdb} &
$\widetilde{\delta}_{41} \neq 0$\\
\cite[\href{http://www.lmfdb.org/EllipticCurve/Q/288/b/3}{288.b3}]{lmfdb} &
$\widetilde{\delta}_{151} \neq 0$
& & \cite[\href{http://www.lmfdb.org/EllipticCurve/Q/784/a/1}{784.a1}]{lmfdb} &
$\widetilde{\delta}_{101} \neq 0$\\
\cite[\href{http://www.lmfdb.org/EllipticCurve/Q/289/a/4}{289.a4}]{lmfdb} &
$\widetilde{\delta}_{181} \neq 0$
& & \cite[\href{http://www.lmfdb.org/EllipticCurve/Q/784/b/5}{784.b5}]{lmfdb} &
$\widetilde{\delta}_{151} \neq 0$\\
\cite[\href{http://www.lmfdb.org/EllipticCurve/Q/324/a/2}{324.a2}]{lmfdb} &
$\widetilde{\delta}_{11} \neq 0$
& & \cite[\href{http://www.lmfdb.org/EllipticCurve/Q/784/g/2}{784.g2}]{lmfdb} &
$\widetilde{\delta}_{691} \neq 0$\\
\cite[\href{http://www.lmfdb.org/EllipticCurve/Q/392/a/1}{392.a1}]{lmfdb} &
$\widetilde{\delta}_{61} \neq 0$
& & \cite[\href{http://www.lmfdb.org/EllipticCurve/Q/784/h/1}{784.h1}]{lmfdb} &
$\widetilde{\delta}_{11} \neq 0$\\
\cite[\href{http://www.lmfdb.org/EllipticCurve/Q/392/c/1}{392.c1}]{lmfdb} &
$\widetilde{\delta}_{401} \neq 0$
& & \cite[\href{http://www.lmfdb.org/EllipticCurve/Q/864/b/1}{864.b1}]{lmfdb} &
$\widetilde{\delta}_{11} \neq 0$\\
\cite[\href{http://www.lmfdb.org/EllipticCurve/Q/392/d/4}{392.d4}]{lmfdb} &
$\widetilde{\delta}_{31} \neq 0$
& & \cite[\href{http://www.lmfdb.org/EllipticCurve/Q/864/c/1}{864.c1}]{lmfdb} &
$\widetilde{\delta}_{41} \neq 0$\\
\cite[\href{http://www.lmfdb.org/EllipticCurve/Q/432/b/3}{432.b3}]{lmfdb} &
$\widetilde{\delta}_{11} \neq 0$
& & \cite[\href{http://www.lmfdb.org/EllipticCurve/Q/864/d/1}{864.d1}]{lmfdb} &
$\widetilde{\delta}_{41} \neq 0$\\
\cite[\href{http://www.lmfdb.org/EllipticCurve/Q/432/c/1}{432.c1}]{lmfdb} &
$\widetilde{\delta}_{31} \neq 0$
& & \cite[\href{http://www.lmfdb.org/EllipticCurve/Q/864/e/1}{864.e1}]{lmfdb} &
$\widetilde{\delta}_{241} \neq 0$\\
\cite[\href{http://www.lmfdb.org/EllipticCurve/Q/441/a/2}{441.a2}]{lmfdb} &
$\widetilde{\delta}_{11} \neq 0$
& & \cite[\href{http://www.lmfdb.org/EllipticCurve/Q/864/g/1}{864.g1}]{lmfdb} &
$\widetilde{\delta}_{11} \neq 0$\\
\cite[\href{http://www.lmfdb.org/EllipticCurve/Q/441/f/6}{441.f6}]{lmfdb} &
$\widetilde{\delta}_{41} \neq 0$
& & \cite[\href{http://www.lmfdb.org/EllipticCurve/Q/864/i/1}{864.i1}]{lmfdb} &
$\widetilde{\delta}_{241} \neq 0$\\
\cite[\href{http://www.lmfdb.org/EllipticCurve/Q/484/a/2}{484.a2}]{lmfdb} &
$\widetilde{\delta}_{101} \neq 0$
& & \cite[\href{http://www.lmfdb.org/EllipticCurve/Q/968/a/1}{968.a1}]{lmfdb} &
$\widetilde{\delta}_{131} \neq 0$\\
\cite[\href{http://www.lmfdb.org/EllipticCurve/Q/576/b/5}{576.b5}]{lmfdb} &
$\widetilde{\delta}_{31} \neq 0$
& & \cite[\href{http://www.lmfdb.org/EllipticCurve/Q/968/b/1}{968.b1}]{lmfdb} &
$\widetilde{\delta}_{61} \neq 0$\\
\cite[\href{http://www.lmfdb.org/EllipticCurve/Q/648/a/1}{648.a1}]{lmfdb} &
$\widetilde{\delta}_{61} \neq 0$
& & \cite[\href{http://www.lmfdb.org/EllipticCurve/Q/968/d/1}{968.d1}]{lmfdb} &
$\widetilde{\delta}_{41} \neq 0$\end{tabular}
\end{center}
\end{proof}
\begin{rem}
The above list contains elliptic curves with supersingular reduction at $p = 5$.
Note that \cite[\href{http://www.lmfdb.org/EllipticCurve/Q/648/b/1}{648.b1}]{lmfdb} does not satisfy the assumptions in Theorem  \ref{thm:main_theorem} since its mod $p$ representation  has exceptional  image $S_4$.
\end{rem}

\subsection{Elliptic curves with good ordinary reduction of square-full conductors} \label{subsec:ordinary_examples}
We consider four elliptic curves over $\mathbb{Q}$ found from \cite{lmfdb} as examples and use \cite{sagemath} for computation. Since all the elliptic curves here have no semistable prime in their conductors, Skinner--Urban's work \cite{skinner-urban} does not apply to these examples. X.~Wan's work \cite{wan_hilbert} could apply only if one can find suitable real quadratic fields. From now on, $\lambda$ means Iwasawa $\lambda$-invariants, not a place dividing $p$.
All four elliptic curves $E_i$ ($i=1, \cdots, 4$) share the following properties:
\begin{itemize}
\item $E_i$ is ordinary and non-anomalous at $p$.
\item $E_i[p]$ is surjective.
\item The product of all the Tamagawa factors are not divisible by $p$.
\item The $\mu$-invariant is zero and the $\lambda$-invariant is 2.
\end{itemize}
Furthermore, by the last condition, their Iwasawa main conjectures do not follow immediately from Kato's Euler system divisibility. 
For the first three elliptic curves, the (analytic) order of the $p$-part of their Shafarevich--Tate groups are non-trivial (with rank zero), i.e. $\lambda = \lambda_{\textrm{\cyr SH}} = 2$.
For the last elliptic curve, the rank of the elliptic curve is two, i.e. $\lambda = \lambda_{\mathrm{MW}} = 2$.

\subsubsection{An elliptic curve of conductor 3364 with $p = 7$}
Let
$$E_1: y^2 = x^3 - 4062871 x - 3152083138$$
be an elliptic curve  of conductor $3364 = 2^2 \cdot 29^2$ as in \cite[\href{http://www.lmfdb.org/EllipticCurve/Q/3364/c/1}{Elliptic Curve 3364.c1}]{lmfdb}.
Then we have $\widetilde{\delta}_{\ell} = 0$ for the first 5 Kolyvagin primes $\ell$ = 1289, 1471, 2549, 2591, and 2689, but
$$\widetilde{\delta}_{1289 \cdot 1471} \neq 0 .$$
Thus, the main conjectures for all members of the Hida family of $E_1[p]$ follow.
\subsubsection{An elliptic curve of conductor 10800 with $p = 7$}
Let 
$$E_2: y^2 = x^3 - 1795500 x - 926032500$$
be an elliptic curve of conductor $10800 = 2^4 \cdot 3^3 \cdot 5^2$ as in \cite[\href{http://www.lmfdb.org/EllipticCurve/Q/10800/dl/1}{Elliptic Curve 10800.dl1}]{lmfdb}.
Then we have $\widetilde{\delta}_{\ell} = 0$ for the first 5 Kolyvagin primes $\ell$ =  71, 113, 491, 967, and 1163, but
$$\widetilde{\delta}_{71 \cdot 113}  \neq 0 .$$
Thus, the main conjectures for all members of the Hida family of $E_2[p]$ follow.

\subsubsection{An elliptic curve of conductor 38088 with $p = 11$}
Let 
$$E_3 : y^2 = x^3 - 937309179 x - 11045170357450 $$
be an elliptic curve of conductor $38088 = 2^3 \cdot 3^2 \cdot 23^2$ as in \cite[\href{http://www.lmfdb.org/EllipticCurve/Q/38088/x/1}{Elliptic Curve 38088.x1}]{lmfdb}.
Then we have $\widetilde{\delta}_{\ell} = 0$ for the first 5 Kolyvagin primes $\ell$ = 463, 727, 881, 2707, and 2927, but
$$\widetilde{\delta}_{463 \cdot 727} \neq  0 .$$
Thus, the main conjectures for all members of the Hida family of $E_3[p]$ follow.

\subsubsection{An elliptic curve of conductor 3456 with $p = 5$}
Let 
$$E_4: y^2 = x^3 - 84 x + 304$$
be an elliptic curve of conductor $3456 = 2^7 \cdot 3^3$ as in \cite[\href{http://www.lmfdb.org/EllipticCurve/Q/3456/a/1}{Elliptic Curve 3456.a1}]{lmfdb}.
Then we have $\widetilde{\delta}_{\ell} = 0$ for all the 5 Kolyvagin primes $\ell$ = 191, 211, 311, 401, and 811, but
$$\widetilde{\delta}_{191 \cdot 211} \neq  0 .$$
Thus, the main conjectures for all members of the Hida family of $E_4[p]$ follow.

\subsection{Non-ordinary modular forms} \label{subsec:nonordinary_examples}
These examples shows how Theorem \ref{thm:main_theorem} applies to the non-ordinary setting without considering any $\pm$- or $\sharp/\flat$-Iwasawa theory.
Since their $L$-values are units, we can easily see $\widetilde{\delta}_1 \neq 0$ for these examples. Indeed, the second example is not genuinely new due to \cite[Proposition 6.2]{kurihara-invent}.

\subsubsection{A non-ordinary modular form with $\mathfrak{p}$ dividing 3}
This example is taken from \cite[5, A 3-adic example, $\S$4.1.2]{fouquet-etnc}.
Let 
$$f = q+\sqrt{6}q^3 +q^5 +2q^7 +3q^9 +(-2+\sqrt{6})q^{11} - q^{13} +\sqrt{6}q^{15} +(2-2\sqrt{6})q^{17} + \cdots$$
in $S_2(\Gamma_0(520))$. Then $f$ is a modular form (of finite slope) which is non-ordinary at $\mathfrak{p} = (3+\sqrt{6})$ above 3 in $\mathbb{Q}(\sqrt{6})$.

Since $a_3(f) \neq 0$ and the Hecke field is not $\mathbb{Q}$, neither \cite{wan-main-conj-ss-ec} nor \cite{sprung-main-conj-ss} applies to this example.
If we can verify certain automorphic assumptions of \cite[Theorem 1.4]{wan-main-conj-nonord}, which should be always true, then \cite{wan-main-conj-nonord} would apply.

Since $520 = 2^3 \cdot 5 \cdot 13$, and $a_5(f) = 1$ and $a_{13}(f) = -1$, we have $3 \nmid (5-1)\cdot (13+1)  = 4 \cdot 14$.
Since $\mathfrak{p} \nmid \frac{L(f, 1)}{\Omega^+_f}$, we have $\widetilde{\delta}_1 \neq 0$.
Thus, Kato's main conjecture (Conjecture \ref{conj:kato-main-conjecture}) for $f$ at $\mathfrak{p} = (3+\sqrt{6})$ holds.

\subsubsection{An elliptic curve with good supersingular reduction with $p=3$}
This example is taken from \cite[6, A 3-adic example, $\S$4.1.2]{fouquet-etnc}.
Let $E$ be the elliptic curve over $\mathbb{Q}$ defined by
$$y^2 = x^3 - 67x + 926$$
with conductor $760 = 2^3 \cdot 5 \cdot 19$ as in \cite[\href{http://www.lmfdb.org/EllipticCurve/Q/760/e/1}{Elliptic Curve 760.e1}]{lmfdb}.
Then we know that the residual representation is surjective, $a_3(E) = 3 (\neq 0)$, $a_5(E) = 1$, $a_{19}(E) = -1$, and $\frac{L(E, 1)}{\Omega^+_E} = 2$.
Since  $3 \nmid (5- 1) \cdot (19+1)=80$ and $\widetilde{\delta}_1 \in \mathbb{F}^\times_3$, Kato's main conjecture for $E$ with $p=3$ holds.

\section*{Acknowledgement}
This project grew out from C.K.'s year-long discussion with Karl Rubin when he was at UC Irvine.
C.K. heartily thanks Liang Xiao and K\^{a}zim B\"{u}y\"{u}kboduk  for guiding him to study Kato's Euler systems and for extremely helpful suggestions and strong encouragement, respectively.
C.K. learned many details of Kato's Euler systems from Kentaro Nakamura and Shanwen Wang during ``an explicit week" at KIAS.
C.K. also greatly appreciates Masato Kurihara's constant encouragement and thanks him for pointing out the relation of this work with \cite{kurihara-iwasawa-2012} and valuable comments.
C.K. thanks Ashay Burungale pointing out the analogy with Heegner points; Keunyoung Jeong for figuring out some computation in $\S$\ref{sec:the_image_of_dual_exp} together;  Kazuto Ota for pointing out the non-ordinary generalization; Olivier Fouquet, Minhyong Kim, Robert Pollack, Tadashi Ochiai, and Haining Wang for the helpful discussion and encouragement.
C.K. appreciates the generous hospitality of Ulsan National Institute of Science and Technology (UNIST), Keio University, and Shanghai Center for Mathematical Sciences during visits.
C.K. was partially supported 
by a KIAS Individual Grant (SP054102) via the Center for Mathematical Challenges at Korea Institute for Advanced Study,
by ``Overseas Research Program for Young Scientists" through Korea Institute for Advanced Study, by ``the 10th MSJ-Seasonal Institute 2017" through Mathematical Society of Japan, and by Basic Science Research Program through the National Research Foundation of Korea (NRF-2018R1C1B6007009).

M.K. thanks to Kentaro Nakamura and Shanwen Wang for giving nice lectures about Euler Systems at KIAS. With their lectures, M.K. got a better picture of the subject.
M.K. appreciates Robert Pollack for the useful discussion and encouragement. M.K. also thanks to Byungheup Jeon, Jungyun Lee, and Peter J. Cho for general support and constant encouragement.

H.S. thanks to Ashay Burungale for helpful conversations and comments about modular symbols. H.S. is supported by Basic Science Research Program through the National Research Foundation of Korea (NRF-2017R1A2B4012408).

All we deeply thank the organizers of Iwasawa 2017 for providing us with the intensive atmosphere, which makes it possible for us to finish the first draft during the conference.

All we deeply thank the referee for his or her extremely careful reading and comments. A number of inaccuracies are corrected and the exposition is improved a lot due to the comments.

\bibliographystyle{amsalpha}
\bibliography{kks}

\end{document}